%% file: ex_article-full-2.tex
\documentclass[hidelinks,onefignum,onetabnum]{siamart251216}

\usepackage{float,subfig,dsfont,bbm}  
\usepackage{exscale}
\usepackage{relsize}
\usepackage{color,graphicx}
\usepackage{lineno,hyperref}
\usepackage{amsmath,amssymb,cases} 

\newtheorem{Definition}{Definition}
\newtheorem{Theorem}{Theorem}[section]
\newtheorem{Lemma}{Lemma}[section]
\newtheorem{Proposition}{Proposition}[section]
\newtheorem{Remark}{Remark}[section]

\input{ex_shared}

\ifpdf
\hypersetup{
  pdftitle={Linear-quadratic mixed Stackelberg-zero-sum game for mean-field regime switching system},
  pdfauthor={P. Y. Huang, N. Li, Z. Q. Xu and H. Zheng}
}
\fi

\newcommand{\E}{\mathbb{E}}

\begin{document}

\maketitle

\begin{abstract}
Motivated by a green finance problem, a linear-quadratic Stackelberg differential game for a regime switching system involving one leader and two followers is studied. The two followers engage in a zero-sum differential game, and both the state system and the cost functional incorporate a conditional mean-field term. Applying continuation method and induction method, we first establish the existence and uniqueness of the conditional mean-field forward-backward stochastic differential equations with Markovian switching. Based on it, we prove the unique solvability of Hamiltonian systems associated with the two followers and the leader. Moreover, utilizing stochastic maximum principle, decoupling approach and optimal filtering technique, we obtain the optimal feedback strategies for the two followers and the leader. Employing the theoretical results, we solve the green finance problem with some numerical simulations.
\end{abstract}



\section{Introduction}

\subsection{Motivation}\label{moti}
In recent decades, it is recognized that the traditional models using continuous processes alone are often inadequate to describe the real world. Numerous systems often exhibit both continuous and discrete characteristics. To meet the practical demands, regime switching diffusion modulated by finite state Markov chain has attracted increasing attentions. A regime switching diffusion is a hybrid stochastic process $\left(X(t),\alpha_t\right)_{t\geq 0}$ containing two components: one is continuous-valued diffusion process $X(t)$, whose drift and diffusion coefficients rely on $\alpha_t$; another one is Markov chain $\alpha_t$, which represents the discrete component taking values in a set of isolated points, and it is primarily employed to describe the phenomena involving trend change or system state shift. Indeed, regime switching diffusion provides an effective tool to character the discrete event and system state shift. Some of such financial examples can be seen in Zhou and Yin \cite{ZY2003} for Markowitz mean-variance portfolio selection, Guo and Zhang \cite{Guo} for pricing perpetual American put options, Sethi and Zhang \cite{Sethi} for hierarchical decision making, Song et al. \cite{Song2011} for optimal harvesting problems, Lv et al. \cite{lvaor} for pricing game options, and so on.

{\color{black}{
Nowadays, against the backdrop of intensifying global climate change and the deepening implementation of the ``dual carbon" goals (carbon peaking and carbon neutrality), green finance has emerged as a core institutional force driving systemic socioeconomic transformation. On the one hand, green finance channels capital toward green industries and fosters the development of new quality productive forces. On the other hand, it helps identify and manage climate-related financial risks. Thus, the issue of green finance is both practically significant and academically valuable. In this work, we consider a green finance problem involving the government and two firms. The government acts as the leader, while two firms serve as followers and engage in a zero-sum game. This problem involves participants with heterogeneous statuses and intricate interplay, rendering our analysis more challenging and our contributions more substantial relative to the existing literature.

Now we present the \emph{green finance problem} in detail, which also motivates us to explore a Stackelberg differential game for regime switching system with one leader and two followers in Section \ref{problem-f}. Let $M>0$ denote the total amount of a fixed green financial resource pool, such as the central bank's green re-lending quota and the total volume of free-allocated carbon emission allowances. Within the policy cycle, this fund pool is exogenously fixed and closed-end. Assume that there are two kinds of firms competing for this limited resource, denoted by Firm 1 and Firm 2, and this competition constitutes the core of the zero-sum game between them. If Firm 2 outperforms Firm 1 in green performance, it receives a larger share of the fund $M$, while Firm 1 bears the opportunity cost (i.e., loses the subsidy); conversely, if Firm 1 successfully catches up, the rent flows from Firm 2 to Firm 1. Although the absolute profits of both firms may rise with market growth, the relative competitive surplus extracted from the fixed fund pool is strictly zero-sum. Now we assume that Firm 2 is ahead of Firm 1 and has better green performance. Let $x(\cdot) \in \mathbb R$ denote the green subsidy share difference of Firm 2 over Firm 1 as 
$$x(\cdot):= s_2(\cdot)-s_1(\cdot),$$ where $s_j(\cdot)$ denotes the subsidy share of Firm $j$, $j=1,2$. Intuitively, $x(\cdot) > 0$ means that Firm 2 outcompetes Firm 1 in subsidy allocation, \( x(\cdot) < 0 \) means the reverse, and \( x(\cdot) = 0 \) indicates a symmetric distribution. Of course, $x(\cdot)$ can also be expressed as $s_1(\cdot)-s_2(\cdot)$, representing the green subsidy share difference of Firm 1 over Firm 2.

In reality, a firm's green subsidy share is also significantly influenced by the market states. For simplicity, we assume that the market is in one of two possible states: bull or bear. We employ a Markov chain $\alpha_t$ taking values in $\mathcal M_0= \{1, 2\}$ to denote two different market regimes, i.e., bull and bear markets. In the bull market, subsidy allocation is typically tied to investment intensity, exhibiting a pronounced ``Matthew effect''. At this stage, the advantaged Firm 2 maintains a high level of investment intensity (active expansion), while lagging Firm 1 also sustains a moderate-to-high level (catching up). Given that subsidy amounts are generally commensurate with investment scale, the advantaged firm, by virtue of its larger investment base, secures a disproportionately greater share of subsidies, thereby widening the subsidy gap between the two firms. In contrast, during the bear market, these two firms universally reduce investment, thereby compressing the subsidy share disparity. Besides, the subsidy share difference $x(\cdot)$ is subject to government intervention. Meanwhile, $x(\cdot)$ is also affected by the average competitive imbalance across the industry, conditional on the market state $\alpha_t$. This conditional effect is captured by $\widehat x(\cdot)=\mathbb{E}\left[x(\cdot)|\mathcal F_{t-}^{\alpha}\right]$, where $\mathcal F_t^{\alpha}=\sigma\{\alpha_s:0\leq s\leq t\}$ and $\alpha_t\in\mathcal M_0$. With the notations and interpretations above, assume that $x(\cdot)$ satisfies a stochastic differential equation (SDE) as follows 
\begin{equation}\label{m-1}
\left\{\begin{aligned}
dx(t)=&\ \big[a_0(t,\alpha_{t-})x(t)+\bar{a}_0(t,\alpha_{t-})\widehat x(t)+b_{f,1}(t,\alpha_{t-})u_{f,1}(t)+b_{f,2}(t,\alpha_{t-})u_{f,2}(t) \\
&\  +b_l(t,\alpha_{t-})u_l(t)\big]dt+c_0(t,\alpha_{t-})x(t)dW(t), \\
x(0)=&\ x_0,  
\end{aligned}\right.
\end{equation}
where $x_0\in\mathbb R$, $u_{f,1}(\cdot), u_{f,2}(\cdot)\in\mathbb R$ are the green investment strategies of the two firms, and $u_l(\cdot)\in\mathbb R$ represents the government's strategy to preserve fair competition. $W(\cdot)$ is the Brownian motion representing the exogenous shock to green technology. The coefficients $a_0(\cdot, i), \bar{a}_0(\cdot, i), b_{f,1}(\cdot,i), b_{f,2}(\cdot,i), b_l(\cdot,i), c_0(\cdot,i)$ are deterministic with values in $\mathbb R$, $i\in\mathcal M_0$.

The cost functional for the two firms is:
\begin{equation}\begin{aligned}\label{cost-fm}
J_{f}\left(x_0; u_l(\cdot), u_{f,1}(\cdot), u_{f,2}(\cdot)\right)
=&\ \frac{1}{2}\mathbb{E}\int_0^T\Big[q_f(t,\alpha_{t-})x^2(t)+r_{f,1}(t,\alpha_{t-})u_{f,1}^2(t) \\ 
&\ -r_{f,2}(t,\alpha_{t-})u_{f,2}^2(t)\Big]dt,
\end{aligned}\end{equation}
where $q_f(\cdot,i), r_{f,1}(\cdot,i), r_{f,2}(\cdot,i)>0$, $i\in\mathcal M_0$. The payoff functional $J_{f}$ is defined as Firm 1's relative competitive disadvantage cost and Firm 2's relative advantage. Thus, Firm 1 and Firm 2 want to minimize and maximize $J_{f}$, respectively. In \eqref{cost-fm}, the term $q_f x^2$ measures the difference in green subsidy shares across the two firms. Given that Firm 1 is the lagging one, it naturally seeks to minimize the difference. The convex investment cost $r_{f,1}u_{f,1}^2$ of Firm 1 is purely an expense term, hence Firm 1 wants to minimize it. Firm 2 aims to maximize $J_{f}$, which is equivalent to choosing a strategy $u_{f,2}\in \mathcal U_{f,2}$ to minimize $-J_{f}$. The term $r_{f,2}u_{f,2}^2$ is a positive cost in $-J_{f}$ to Firm 2, and it is the monetary cost of Firm 2's investment. Firm 2 seeks to maximize $J_{f}$, implying that it intends to maximize the difference $q_f x^2$ so as to maintain its advantaged position. 


Indeed, government regulatory interventions shape these two firms' competition for subsidy shares, thereby influencing $x(\cdot)$ as well. From the government's perspective, the goal is to maintain fair competition and dynamic balance in market structure. Suppose that for these two firms, there exists an optimal strategy $\mathbbm{u}^*_{f_{1,2}}[\cdot]$ depending on $u_l(\cdot)$ and $x_0$. Then, the government wants to seek a strategy $u_{l}(\cdot)$ to minimize
\begin{equation}\label{cost-lm}
J_{l}\left(x_0; u_l(\cdot), \mathbbm{u}^*_{f_{1,2}}[u_l(\cdot),x_0]\right)
= \frac{1}{2}\E \int_0^T\left[\bar q_l(t,\alpha_{t-})\widehat x^2(t)+r_l(t,\alpha_{t-})u_l^2(t)\right]dt,  
\end{equation}
where $\bar q_l(\cdot,i), r_l(\cdot,i)>0$, $i\in\mathcal M_0$, $\bar q_l\widehat x^2$ denotes the overall green level of the market, and $r_lu_l^2$ is the social cost of frequent policy changes. The government strategies change frequently will incur institutional costs and may even disrupt the competitive ecosystem. Thus, the government aims to minimize its policy change $u_l^2$ and $\widehat x^2$ over the interval $[0,T]$, thereby ensuring social stability and achieving resource allocation equilibrium.

The above green financial problem can be modeled as a mixed linear-quadratic (LQ) Stackelberg-zero-sum game for a mean-field regime switching system, where one leader interacts with two followers and these two followers engage in a zero-sum game. We will address such problem and provide the corresponding analysis in the following.}}

\subsection{{\color{black}{Notations and problem formulation}}}\label{problem-f}
\subsubsection{{\color{black}{Notations}}}
Let $(\Omega, \mathcal{F}, \mathbb F, \mathbb{P})$ be a complete filtered probability space on which a 1-dimensional standard Brownian motion $W(\cdot)$ and a Markov chain $\alpha_{\cdot}$ are defined. The Markov chain $\alpha_{\cdot}$ takes values in a finite state space $\mathcal M=\{1,2,\cdots, m\}$. Let $Q=(\lambda_{ij})_{i,j\in\mathcal M}$ be the generator (i.e., the matrix of transition rates) of $\alpha_{\cdot}$ with $\lambda_{ij}\geq 0$ for $i\neq j$ and $\sum_{j\in\mathcal M}\lambda_{ij}=0$ for each $i\in \mathcal M$. Assume that $W(\cdot)$ and $\alpha_{\cdot}$ are independent. Let $\mathbb F=\left\{\mathcal F_t: 0\leq t\leq T\right\}$ be generated by $W(\cdot)$ and $\alpha_{\cdot}$. For a stochastic process $h(\cdot)$, we let $\widehat h(t)=\mathbb{E}\left[h(t)|\mathcal F_{t-}^{\alpha}\right]$. Let $\mathbb{R}^l$ be the $l$-dimensional Euclidean space with norm $|\cdot|$. Let $\textbf{0}$ and $\textbf{I}$ represent a zero matrix (or vector) and an identity matrix with suitable dimensions, respectively. Let $A^{\top}$ be the transpose of matrix $A$. For symmetric matrices $A, B\in\mathbb R^{n\times n}$, $A\geq B$, if $x^{\top}Ax\geq x^{\top}Bx$ for any $x\in \mathbb R^n$. For convenience, let $\mathbb{V}$ represent some spaces such as $\mathbb R^n, \mathbb R^{n\times k}$, etc., and we introduce several spaces as follows:
\begin{itemize}
\item 
$L^{\infty}(0,T;\mathbb{V})=\big\{x: [0,T]\rightarrow \mathbb{V}|x(\cdot)$ is uniformly
bounded$\big\}$,

\item
$L_{\mathbb F}^2\left(0,T; \mathbb{V}\right)=\Big\{x: [0,T]\times \Omega\rightarrow \mathbb{V}|x(\cdot)$ is $\mathbb F$-progressively measurable process and satisfies $\mathbb{E}\Big(\int_0^T|x(t)|^2dt\Big) <+\infty\Big\}$,

\item 
$\mathcal S_{\mathbb F}^2\left(0,T; \mathbb{V}\right)=\Big\{x: [0,T]\times \Omega\rightarrow \mathbb{V}|x(\cdot)$ is $\mathbb F$-adapted c$\grave{a}$dl$\grave{a}$g process and satisfies
$\mathbb{E}\Big(\sup\limits_{t\in[0,T]}|x(t)|^2\Big)<+\infty\Big\}$,

\item
\noindent $L^2_{\mathcal F_T}\left(\Omega; \mathbb{V}\right)=\big\{\xi: \Omega\rightarrow \mathbb{V}|\xi$ is $\mathcal F_T$-measurable random variable and satisfies $ \mathbb E|\xi|^2<+\infty\big\}$,

\item
\noindent $L_{\mathbb F}^2\left(\Omega; C\left([0,T]; H\right)\right)=\Big\{x: [0,T]\times \Omega\rightarrow H|x(\cdot)$ is $\mathbb F$-progressively measurable process valued in $H$, such that for almost all $\omega\in \Omega$, $t\rightarrow x(t,\omega)$ is continuous and $\mathbb{E}\Big(\sup\limits_{t\in[0,T]}|x(t)|^2\Big)<+\infty\Big\},$

\item
\noindent $L_{\mathbb F}^2\left(\Omega; L\left(0,T; H\right)\right)=\Big\{x: [0,T]\times \Omega\rightarrow H |x(\cdot)$ is $\mathbb F$-progressively measurable process valued in $H$, such that $\mathbb{E}\left(\int_0^T|g(t)|dt\right)^2<+\infty\Big\}$,

\item
$ N_{\mathbb F}^2\left(0,T; \mathbb R^{3n}\right)=
L_{\mathbb F}^2\left(\Omega; C\left([0,T]; \mathbb R^n\right)\right)\times
\mathcal S_{\mathbb F}^2\left(0,T; \mathbb R^n\right)\times
L_{\mathbb F}^2\left(0,T; \mathbb R^n\right)$,

\item
$\mathcal N_{\mathbb F}^2\left(0,T; \mathbb R^{3n}\right)=
L_{\mathbb F}^2\left(\Omega; L\left(0,T; \mathbb{R}^n\right)\right)\times
L_{\mathbb F}^2\left(\Omega; L\left(0,T; \mathbb{R}^n\right)\right)\times
L_{\mathbb F}^2\left(0,T; \mathbb{R}^n\right)$,

\item
$\mathcal H[0,T]=\mathbb R^n \times L^2_{\mathcal F_T}\left(\Omega; \mathbb R^n\right) \times
\left(\mathcal N_{\mathbb F}^2\left(0,T; \mathbb R^{3n}\right)\right)^2$.
\end{itemize}

For each pair $(i,j)\in \mathcal M\times \mathcal M$ with $i\neq j$, let
$[M_{ij}](t)=\sum_{0\leq s\leq t}\mathbf{1}_{(\alpha_{s-}=i)}\mathbf{1}_{(\alpha_s=j)},$ 
$\langle M_{ij}\rangle(t)=\int_0^t \lambda_{ij}\mathbf{1}_{(\alpha_{s-}=i)}ds,$
where $\mathbf{1}$ represents the indicator function. It follows from Nguyen et al. \cite{Yin2020} that the process $M_{ij}(t)=[M_{ij}](t)-\langle M_{ij}\rangle(t)$ is a purely discontinuous and square-integrable martingale with respect to $\mathcal F_t^{\alpha}$, which is null at the origin. Besides, we denote $M_{ii}(t)=[M_{ii}](t)=\langle M_{ii}\rangle(t)=0$ for each $i\in\mathcal M$ and $t\geq 0$. Define\\
$\mathcal M_{\mathbb F}^2\left(0,T; \mathbb R\right)=\Big\{k=(k_{ij}: i,j\in\mathcal M)|$
$k_{ij}$ is $\mathbb F$-progressively measurable process, $k_{ii}\equiv0$ for $i,j\in\mathcal M$, and $\sum_{i,j\in\mathcal M}\mathbb{E}\Big(\int_0^T|k_{ij}(t)|^2d[M_{ij}](t)\Big)<+\infty\Big\}$.\\
For convenience, let $h(\cdot)=\left(h_{ij}(\cdot)\right)_{i,j\in\mathcal M}$ and
$h(s)\bullet dM(s)=\sum_{i,j\in\mathcal M}h_{ij}(s)dM_{ij}(s).$

\subsubsection{Problem formulation}

Motivated by the green finance problem, we consider a more general Stackelberg differential game of regime switching system with one leader and two followers in this work. Let the state of the system be described by controlled linear conditional mean-field SDE with regime switching:\\
\vspace*{-0.5cm}
\begin{equation}\label{state}
\left\{\begin{aligned}
dx(t)=&\ \big[A(t,\alpha_{t-})x(t)+\bar{A}(t,\alpha_{t-})\widehat x(t)+B_L(t,\alpha_{t-})u_L(t)\\
&\ +\mathbf{B_F}(t,\alpha_{t-})\mathbf{u_{F_{1,2}}}(t)+b(t,\alpha_{t-})\big]dt+\big[C(t,\alpha_{t-})x(t)
+\bar C(t,\alpha_{t-})\widehat x(t) \\
&\ +D_L(t,\alpha_{t-})u_L(t)+\mathbf{D_F}(t,\alpha_{t-})\mathbf{u_{F_{1,2}}}(t)
+\sigma(t,\alpha_{t-})\big]dW(t), \\
x(0)=&\ x_0,
\end{aligned}\right.
\end{equation}
where
$$
\mathbf{B_F}=\left(B_{F,1}, B_{F,2}\right),\quad
\mathbf{D_F}=\left(D_{F,1}, D_{F,2}\right),\quad
\mathbf{u_{F_{1,2}}}=\left(u_{F,1}^{\top}, u_{F,2}^{\top}\right)^{\top},
$$
$u_L(\cdot)$ is the strategy process taken by the leader with values in $\mathbb{R}^{m_0}$, and $u_{F,1}(\cdot), u_{F,2}(\cdot)$ are the strategy processes taken by the two followers with values in $\mathbb{R}^{m_1}$ and $\mathbb{R}^{m_2}$, respectively; $x(\cdot)$ is the state process corresponding to $\left(u_L(\cdot),u_{F,1}(\cdot), u_{F,2}(\cdot)\right)$, $x_0 \in \mathbb{R}$. Moreover, for any $i\in \mathcal M$, $A(\cdot,i), \bar{A}(\cdot,i), C(\cdot,i),$ $\bar C(\cdot,i)\in L^{\infty}(0,T;{\mathbb R})$, $B_L(\cdot,i), D_{L}(\cdot,i)$ $\in L^{\infty}(0,T;\mathbb R^{1\times m_0}),$  $B_{F,1}(\cdot,i), D_{F,1}(\cdot,i)\in L^{\infty}(0,T;\mathbb R^{1\times m_1}),$ $B_{F,2}(\cdot,i), D_{F,2}(\cdot,i)\in L^{\infty}(0,T;\mathbb R^{1\times m_2})$, $b(\cdot,i)$ and $\sigma(\cdot,i)$ are $\mathcal F_t$-adapted processes with one-dimension.

Define
$$\mathcal U_L=\left\{u_L(\cdot)|u_L(\cdot)\in L_{\mathcal F}^2\big(0,T; \mathbb R^{m_0}\big)\right\},
\quad \mathcal U_{F,j}=\left\{u_{F,j}(\cdot)|u_{F,j}(\cdot)\in L_{\mathcal F}^2\big(0,T; \mathbb R^{m_j}\big)\right\}, $$
$j=1, 2$. For simplicity, let $\mathbf{\mathcal U_{F_{1,2}}}=\mathcal U_{F,1}\times\mathcal U_{F,2}$.
By virtue of Lemma 2.1 in Nguyen et al. \cite{Yin2020}, for any given $\left(u_L(\cdot), \mathbf{u_{F_{1,2}}(\cdot)}\right)\in \mathcal U_L\times\mathbf{\mathcal U_{F_{1,2}}}$, \eqref{state} is uniquely solvable. An admissible strategy for the two followers is a mapping $\mathbbm{u}_{F}: \mathcal U_L\times \mathbb R \rightarrow \mathcal U_{F,1}\times \mathcal U_{F,2},$ and the collection of all such strategies is denoted by $\mathcal{U}_{F}[0,T]$.

The cost functional for two followers is
\begin{equation}\label{cost-f}
\begin{aligned}
& J_F\left(x_0; u_L(\cdot), \mathbf{u_{F_{1,2}}}(\cdot)\right)
= \frac{1}{2}\mathbb{E}\bigg\{\int_0^T\Big[Q_F(t,\alpha_{t-})x^2(t)+\bar Q_F(t,\alpha_{t-})\widehat x^2(t) \\
&\qquad +\mathbf{u^{\top}_{F_{1,2}}}(t)\mathbf{R_F}(t,\alpha_{t-})\mathbf{u_{F_{1,2}}}(t)\Big]dt 
  +G_F(\alpha_{T})x^2(T)+\bar G_F(\alpha_{T})\widehat x^2(T)\bigg\} 	
\end{aligned}
\end{equation}
with
$$\mathbf{R_F}=\left(\begin{array}{ll}
R_{F,1} & S_F \\
S_F^{\top} & R_{F,2}
\end{array}\right). $$
The cost functional for the leader is
\begin{equation}\label{cost-l}
\begin{aligned}
& J_L\left(x_0; u_L(\cdot), \mathbf{u_{F_{1,2}}}(\cdot)\right)= \frac{1}{2}\mathbb{E}\bigg\{\int_0^T\Big[Q_L(t,\alpha_{t-})x^2(t)
+\bar Q_L(t,\alpha_{t-})\widehat x^2(t) \\ 
&\qquad\quad +u_{L}^{\top}(t)R_{L}(t,\alpha_{t-})u_{L}(t)\Big]dt
 +G_L(\alpha_{T})x^2(T)+\bar G_L(\alpha_{T})\widehat x^2(T)\bigg\}.	
\end{aligned}
\end{equation}
In \eqref{cost-f}-\eqref{cost-l}, $S_F(\cdot,i)\in L^{\infty}(0,T;\mathbb R^{m_1\times m_2}), R_{F,j}(\cdot,i)\in L^{\infty}(0,T;\mathbb R^{m_j\times m_j})$ ($j=1,2)$, $R_{L}(\cdot,i)$ $\in L^{\infty}(0,T;\mathbb R^{m_0\times m_0})$ are symmetric, and $Q_F(\cdot,i), \bar Q_F(\cdot,i), Q_L(\cdot,i), \bar Q_L(\cdot,i)$,\\ $G_F(i), \bar G_F(i), G_L(i), \bar G_L(i)\in L^{\infty}(0,T;\mathbb R)$, $i\in\mathcal M$; $R_{F,1}(\cdot,i)$ and $R_{F,2}(\cdot,i)$ are positive and negative definite, respectively, $i\in\mathcal M$.

For the two followers, we assume that cost functional \eqref{cost-f} is the loss of Follower 1 and the gain of Follower 2. Therefore, Follower 1 wants to minimize \eqref{cost-f} by seeking a strategy $u^*_{F,1}(\cdot)\in\mathcal U_{F,1}$, while Follower 2 wants to maximize \eqref{cost-f} by finding a strategy $u^*_{F,2}(\cdot)\in\mathcal U_{F,2}$. For any $u_L(\cdot)\in\mathcal U_L$, the two followers aim to find an admissible strategy $\mathbbm{u}^*_{F_{1,2}}[\cdot]\in\mathcal{U}_{F}[0,T]$ such that
\begin{equation}\begin{aligned}\label{saddle}
J_F\left(x_0; u_L(\cdot), u_{F,1}^*(\cdot), u_{F,2}(\cdot)\right)
\leq&\ J_F\left(x_0; u_L(\cdot), \mathbf{u^*_{F_{1,2}}}(\cdot)\right) \\
\leq&\ J_F\left(x_0; u_L(\cdot), u_{F,1}(\cdot), u_{F,2}^*(\cdot)\right),
\end{aligned}\end{equation}
where $\mathbf{u^*_{F_{1,2}}}(\cdot)=\mathbbm{u}^*_{F_{1,2}}[u_L(\cdot),x_0]$ is the outcome of the strategy $\mathbbm{u}^*_{F_{1,2}}$. The above problem involving two followers constitutes a zero-sum differential game, denoted by \textbf{Problem (ZSDG)}. If we can find an admissible strategy $\mathbbm{u}^*_{F_{1,2}}[\cdot]$ satisfying \eqref{saddle}, then we call it an open-loop saddle point of Problem (ZSDG).

Assume that Problem (ZSDG) is well-posed and the optimal strategy $\mathbbm{u}^*_{F_{1,2}}[\cdot]$ exists, the leader wants to seek a strategy $u_L^*(\cdot)\in\mathcal U_L$ satisfying
\begin{equation}\begin{aligned}\label{sc}
J_L\left(x_0; u_L^*(\cdot), \mathbbm{u}^*_{F_{1,2}}[u_L^*(\cdot),x_0]\right)  
= \min_{u_L(\cdot)\in\mathcal{U}_L}J_L\left(x_0; u_L(\cdot), \mathbbm{u}^*_{F_{1,2}}[u_L(\cdot),x_0]\right),
\end{aligned}\end{equation}
subject to \eqref{state} and \eqref{cost-l}. For convenience, the problem for the leader can be seen as the leader's optimal control problem, denoted as \textbf{Problem (LOC)}.

We call the above optimization problem \eqref{state}-\eqref{sc} a stochastic Stackelberg differential game (SDG) for conditional mean-field switching diffusion with one leader and two followers, denote by \textbf{Problem (CMF-SDG)}, in which two followers formulate a zero-sum differential game. If $\big(u_L^*(\cdot), \mathbf{u^*_{F_{1,2}}}(\cdot)\big)\in \mathcal U_L \times \mathbf{\mathcal U_{F_{1,2}}}$ exists, we call it a Stackelberg equilibrium of Problem (CMF-SDG).

\subsection{Literature review and contributions of this paper}
Differential game has attracted considerable interest from diverse areas owing to its important theoretical and practical significance. Differential game was initiated by Isaacs \cite{Isa}, which involves multiple players whose strategies interact dynamically. In 1952, Stackelberg \cite{Sta} introduced Stackelberg game, in  which the leader claims a strategy first, then the follower optimizes his cost functional based on the announcement taken by the leader. Due to the applications in macroeconomics, management sciences, social sciences and biology, Stackelberg game has emerged as a widely studied topic. Castanon and Athans \cite{Cas} focused on a Stackelberg game of the stochastic system, and derived a feedback form for Stackelberg solution. Yong \cite{Yong2002} studied a Stackelberg differential game in the LQ setting, where the coefficients of dynamic system and the cost functionals are random, and derived the open-loop strategies. Li and Yu \cite{LN2018} studied a more generalized Stackelberg game with $N$-level hierarchy, and obtained the unique equilibrium in a closed form. Ishida and Shimemura \cite{Is1, Is2} investigated a kind of Stackelberg game of coupled differential equations composed of lumped and distributed parameter subsystems in deterministic cases, and derived the necessary and sufficient conditions for the open-loop Stackelberg strategies. Feng et al. \cite{Feng} considered a backward Stackelberg differential game with constraints, and obtained an open-loop Stackelberg equilibrium represented by some coupled backward-forward stochastic differential systems with mixed initial-terminal conditions. Huang et al. \cite{HWW} presented two types of modeling uncertainties in a stochastic LQ Stackelberg differential game, and discussed the associated robust Stackelberg strategy design for either the leader or follower. For more information, see \cite{Moon, WS, ZSS, AMC}.

Note that, in the green finance problem involving the government and two firms presented in Subsection \ref{moti}, the relationship among the players involves not only a hierarchical (leader-follower) structure, but also an additional zero-sum game between the two firms. Examples of zero-sum game also include gambling, and options and futures contracts. In fact, zero-sum differential game is also an important branch of game theory. The initial work of zero-sum differential game can be traced back to the work of Ho et al. \cite{Ho}. Then, Fleming and Souganidis \cite{Fle} gave the existence of value function for a two-player zero-sum stochastic differential game. Buckdahn and Li \cite{Buck} investigated a zero-sum stochastic differential game using backward stochastic differential equation (BSDE), which generalized the results obtained in \cite{Fle}. Yu \cite{YZY} investigated a zero-sum stochastic LQ differential game and, by exploiting the special structure of the LQ system, derived an optimal feedback control-strategy pair by solving the associated Riccati equation. Moon \cite{moon1} extended the results of \cite{YZY} to Markov jump systems, and characterized the explicit Nash equilibrium and obtained the optimal game value. Xu et al. \cite{XSZ} studied a class of stochastic LQ leader-follower differential games with a delay in the leader's control, and obtained the open-loop solution. Besides, \cite{Sun2023, 2014SJR} discussed zero-sum differential game in different frameworks.

Notably, in most of the works mentioned above, Stackelberg game and zero-sum game are driven by stochastic systems that incorporate neither abrupt environmental changes nor trend shifts. However, in practical scenarios such as financial markets and ecological systems, the ambient environment is dynamic and continuously evolving. For instance, telegraph noise may perturb epidemic models and induce the system to switch among different environmental regimes. Similarly, the return and volatility of stocks typically differ between bull and bear markets. Nevertheless, an SDE can not adequately characterize such phenomena. Unlike traditional system, regime switching models can accurately capture the infrequent yet impactful behaviors that shape the system's long-term evolution, which have attracted significant research interest in stochastic control theory and its financial applications. Hao et al. \cite{Hao} investigated the maximum principle for a stochastic control problem involving a Markov chain and model uncertainty, and gave the necessary and sufficient conditions for stochastic optimal control. Zhang and Xu \cite{ZPP} investigated a zero-sum game with regime switching. By virtue of a new multidimensional indefinite stochastic Riccati equations together with multidimensional linear BSDEs, \cite{ZPP} derived closed-loop optimal feedback control-strategy pairs for the two players. Wang and Zhang \cite{WBC1} addressed the problem of social optimization in a Markovian setting and gave an asymptotically team-optimal solution.  Lv et al. \cite{lvaor} solved the Dynkin game for switching diffusions in an infinite horizon, and proved that the value of the game is the unique viscosity solution of associated variational inequalities.  See \cite{cath, Ji, lamo, LSY2020, ZQ} for more details.

Motivated by above, we study an LQ mixed Stackelberg differential game for regime switching system with one leader and two followers, in which two followers constitute a zero-sum game. This problem is of important values in both theoretical and practical aspects. To our best knowledge, there exists little literature along this line. Compared with the existing works, the distinctive contributions of this paper are summarized as follows.
\begin{itemize}  
\item
We propose a new type of Stackelberg differential game stemming from the green finance problem, and we give the existence and uniqueness for a system of coupled Riccati equations with Markov chain. Different from the works of Lv et al. \cite{Lv2023}, Shi et al. \cite{WS}, Sun \cite{Sun2021}, Wang et al. \cite{ZSS}, Yong \cite{Yong2002}, the two followers in our work give rise to a system of coupled Riccati equations involving Markov chain. Due to the fact that these Riccati equations are coupled with each other via an interaction term, thus they can not be solved by the traditional Riccati equation theory introduced in Yong and Zhou \cite{Yongzhou}. To do this, introducing some auxiliary equations, presenting some estimates and employing the monotone bounded theorem, the unique solvability of such coupled Riccati equations with Markov chain is derived.

\item
We obtain the unique solvability of conditional mean-field forward-backward stochastic differential equations (CMF-FBSDEs) with respect to regime switc-\\hing. Compared with Li and Yu \cite{LN2018}, Li and Wang \cite{LAMO}, Moon \cite{Moon}, Shi et al. \cite{WS}, the Hamiltonian system for the leader in this paper is of different type and more complicated, which contains two coupled CMF-FBSDEs with regime switching and is extremely challenging to derive its solvability. By introducing domination-monotonicity conditions with conditional expectation, we apply continuation method and induction method to prove the unique solvability of such CMF-FBSDEs. Then, we discuss the Hamiltonian systems of the leader and two followers, which is usually absent in existing research works, such as Shi et al. \cite{WS}, Wang and Yan \cite{AMC}, Wang et al. \cite{ZSS}, etc.

\item
We give the open-loop forms and their closed-loop representations of the optimal strategies for the two followers and the leader, respectively. By virtue of stochastic maximum principle, one first derives the open-loop strategies of the two followers and the leader, which depend on the associated Hamiltonian systems. Further, exploring the relation between the optimal strategy and the optimal state explicitly, introducing some Riccati equations with Markov chain, we obtain the feedback representations of the optimal strategies for the two followers and the leader, not just the open-loop form.

\end{itemize}

The reminder of this paper is organized as follows. We study the unique solvability of the CMF-FBSDEs with Markovian switching in Section \ref{S3}, which plays an important role in analyzing the Hamiltonian systems for the mixed Stackelberg game. Furthermore, we investigate an LQ Stackelberg game with regime switching involving one leader and two followers in Section \ref{S4}. The optimal strategies in the open-loop forms and closed-loop representations are also obtained. We solve the green finance problem in Section \ref{S5}, and give some concluding remarks in Section \ref{S6}. The appendix contains the proofs of the main results.


\section{Unique solvability of CMF-FBSDEs with regime switching}\label{S3}
In this section, we are devoted to investigating the unique solvability of CMF-FBSDEs with regime switching. For LQ Stackelberg game discussed in Section \ref{S4}, the Hamiltonian systems of the two followers and the leader consist of {\color{black}{fully coupled }} CMF-FBSDE and two fully coupled CMF-FBSDEs with regime switching, respectively. Thus, we first discuss the solvability of the nonlinear CMF-FBSDEs with regime switching, which plays a key role in studying the mixed Stackelberg-zero-sum game below.

Consider the nonlinear CMF-FBSDEs with regime switching as follows:
\begin{equation}\label{fbsde}
\left\{\begin{aligned}
dx_j(t)=&\ b_j\Big(t,\Theta(t),\widehat\Theta(t), \alpha_{t-}\Big)dt+\sigma_j\Big(t,\Theta(t),\widehat\Theta(t),\alpha_{t-}\Big)dW(t),\\
dy_j(t)=&\ g_j\left(t,\Theta(t),\widehat\Theta(t),\alpha_{t-}\right)dt+z_j(t)dW(t)+\tilde z_j(t)\bullet dM(t),\\
x_j(0)=&\ \Psi_j\big({\color{black}{\textbf{y}}}(0)\big), \quad y_j(T)=\Phi_j\big({\color{black}{\textbf{x}}}(T),\widehat{\color{black}{{\textbf{x}}}}(T),\alpha_{T}\big),\quad  j=1,2,
\end{aligned}\right.
\end{equation}
where $\Psi_j: \mathbb R^{n}\rightarrow \mathbb R^n$, $\Phi_j: \mathbb R^{2n}\times\mathcal M \rightarrow \mathbb R^n$, $b_j,\sigma_j,g_j:[0,T]\times \mathbb R^{6n}\times\mathcal M \rightarrow \mathbb R^n$, and
{\color{black}{
\begin{equation} 
\left\{\begin{aligned}
&
\Theta(\cdot)=\left(\Theta_1^{\top}(\cdot), \Theta_2^{\top}(\cdot)\right)^{\top},\quad 
\widehat\Theta(\cdot)=\left(\widehat{\Theta}_1^{\top}(\cdot), 
\widehat{\Theta}_2^{\top}(\cdot)\right)^{\top}, \\
& \Theta_j(\cdot)=\left(x_j^{\top}(\cdot), y_j^{\top}(\cdot), z_j^{\top}(\cdot)\right)^{\top},\quad
\widehat{\Theta}_j(\cdot)=\left(\widehat{x}_j^{\top}(\cdot), \widehat{y}_j^{\top}(\cdot), \widehat{z}_j^{\top}(\cdot)\right)^{\top},\quad j=1,2,   \\
&
\textbf{x}(T)=\left(x_1^{\top}(T), x_2^{\top}(T)\right)^{\top},\quad
\widehat{\textbf{x}}(T)=\left(\widehat{x}_1^{\top}(T), \widehat{x}_2^{\top}(T)\right)^{\top},\quad
\textbf{y}(0)=\left(y_1^{\top}(0),y_2^{\top}(0)\right)^{\top}.\nonumber
\end{aligned}\right.
\end{equation}

Note that in \eqref{fbsde}, both SDEs and BSDEs contain the forward parts $x_j(\cdot)$ and the backward parts $\left(y_j(\cdot), z_j(\cdot), \tilde z_j(\cdot)\right)$ for $j=1,2$, through the coupling terms $\Theta(\cdot)$ and $\widehat\Theta(\cdot)$. Consequently, \eqref{fbsde} is a system of two fully coupled CMF-FBSDEs. }}

For simplicity, let
\begin{equation} 
\left\{\begin{aligned}
&\Gamma_j=\left(g_j^{\top}, b_j^{\top}, \sigma_j^{\top}\right)^{\top},\quad j=1,2,\\
&\Gamma=\left(\Gamma_1^{\top}, \Gamma_2^{\top}\right)^{\top},\quad
\Psi=\left(\Psi_1^{\top}, \Psi_2^{\top}\right)^{\top},\quad 
\Phi=\left(\Phi_1^{\top},\Phi_2^{\top}\right)^{\top}. \nonumber
\end{aligned}\right.
\end{equation}

Now we give the Lipschitz condition and domination-monotonicity conditions. 

\noindent\textbf{Assumption (A1)}
1) For any ${\color{black}{\textbf{x}, \textbf{y}}}\in\mathbb R^{2n}$ and any $i\in\mathcal M$, $\Psi({\color{black}{\textbf{y}}})$ is deterministic and $\Phi({\color{black}{\textbf{x}}},\widehat{{\color{black}{\textbf{x}}}},i)$ is $\mathcal F_T$-measurable. For any $\Theta\in\mathbb R^{6n}$ and any $i\in\mathcal M$, $\Gamma\big(\cdot,\Theta,\widehat\Theta,i\big)$ is $\mathbb F$-progressively measurable. Besides, $\Psi(\mathbf{0})\in \mathbb R^n$, $\Phi(\mathbf{0},\mathbf{0},i)\in  L^2_{\mathcal F_T}\left(\Omega; \mathbb R^{n}\right)$, and $\Gamma(\cdot,\mathbf{0},\mathbf{0},i)\in N_{\mathbb F}^2\left(0,T; \mathbb R^{3n}\right)\times N_{\mathbb F}^2\left(0,T; \mathbb R^{3n}\right)$.  2) For any $i\in\mathcal M$, $\Psi, \Phi$ and $\Gamma$ are uniformly Lipschitz continuous with respect to ${\color{black}{\textbf{y}}},  ({\color{black}{\textbf{x}}},\widehat{{\color{black}{\textbf{x}}}})$ and $\big(\Theta,\widehat\Theta\big)$, respectively.

\noindent\textbf{Assumption  (A2)}
There exist two nonnegative constants $\mu_1\geq 0, \nu_1\geq 0$ satisfying $\mu_1> 0, \nu_1= 0$ or 
$\mu_1=0, \nu_1>0$, constants $\mu_2>0$, {\color{black}{$K_0>0$, $K_0'>0$}}, for any $i\in\mathcal M$ and $j=1,2$, matrices $\widetilde M_j,\widetilde G_j(i), \widetilde G_j'(i), \widetilde A_1(\cdot,i), \widetilde A_1'(\cdot,i), \widetilde B_1(\cdot,i), \widetilde B_1'(\cdot,i), \widetilde C_1(\cdot,i)$, $\widetilde C_1'(\cdot,i), \widetilde L_2(\cdot,i), \widetilde L_2'(\cdot,i)$ with suitable dimensions, where \
$\widetilde L_2(\cdot,i)=\big(\widetilde A_2(\cdot,i),\widetilde B_2(\cdot,i),\\ \widetilde C_2(\cdot,i)\big)$ and $\widetilde L_2'(\cdot,i)=\big(\widetilde A_2'(\cdot,i),\widetilde B_2'(\cdot,i),\widetilde C_2'(\cdot,i)\big)$, such that

1) (domination condition)
For any $(\omega, t,i)\in\Omega \times [0,T]\times \mathcal M$, and any $\Theta_j,\widehat{\Theta}_j, \Theta_j',$\\
$\widehat\Theta_j'\in{\mathbb R}^{3n}$,  $j=1,2$,
\begin{equation}\label{2.2}
\left\{\begin{aligned}
& \left|\Psi_j\left(y_j,y_{3-j}\right)-\Psi_j\left(y_j',y_{3-j}\right)\right|
\leq \frac{1}{\mu_1}\left|\widetilde M_1\Delta y_j\right|,\\
& \left|\Phi_j\left(x_j,\widehat x_j,x_{3-j},{\widehat x}_{3-j},i\right)-\Phi_j\left(x_j',\widehat x_j',x_{3-j},{\widehat x}_{3-j},i\right)\right|  \\
&\qquad\qquad\qquad\qquad\qquad\qquad\qquad\qquad \qquad\quad\quad
\leq {\color{black}{K_0'}}\left|\widetilde G_1(i)\Delta x_j
+\widetilde G_1'(i)\Delta \widehat x_j\right|,\\
& \left|g_j\Big(t,x_j,y_j,z_j,\widehat x_j,\widehat y_j,\widehat z_j,\Theta_{3-j},\widehat\Theta_{3-j},i\Big)-
g_j\Big(t,x_j',y_j,z_j,\widehat x_j',\widehat y_j,\widehat z_j,\Theta_{3-j},\widehat\Theta_{3-j},i\Big)\right|\\
&\qquad\qquad\qquad\qquad\qquad\qquad\qquad\qquad \qquad\quad\quad
\leq {\color{black}{K_0'}}\left|\widetilde A_1(t,i)\Delta x_j
+\widetilde A_1'(t,i)\Delta {\widehat x}_j\right|,\\
&\left|f_j\Big(t,x_j,y_j,z_j,\widehat x_j,\widehat y_j,\widehat z_j,\Theta_{3-j},\widehat\Theta_{3-j},i\Big)-
f_j\Big(t,x_j,y_j',z_j',\widehat x_j,\widehat y_j',\widehat z_j',
\Theta_{3-j},\widehat\Theta_{3-j},i\Big)\right| \\
&\qquad\qquad\qquad\quad\quad
\leq \frac{1}{\mu_1}\Big|\widetilde B_1(t,i)\Delta y_j+\widetilde B_1'(t,i)\Delta \widehat y_j 
 +\widetilde C_1(t,i)\Delta z_j+\widetilde C_1'(t,i)\Delta \widehat z_j\Big|,  \\
\end{aligned}\right.
\end{equation}
and 
\begin{equation}\label{A3-1}
\left\{\begin{aligned}
& \left|\Psi_1\left(y_1,y_2\right)-\Psi_1\left(y_1,y_2'\right)\right|
\leq {\color{black}{K_0}}\left|\widetilde M_2\Delta y_2\right|,\\
& \left|\Phi_1\left(x_1,\widehat x_1,x_2,\widehat x_2,i\right)-\Phi_1\left(x_1,\widehat x_1,x_2',
\widehat x_2',i\right)\right| \leq {\color{black}{K_0}}\left|\widetilde G_2(i)\Delta x_2+\widetilde G_2'(i)\Delta\widehat x_2\right|,\\
& \Big|\Gamma_1\Big(t,\Theta_1,\widehat\Theta_1,\Theta_2,\widehat\Theta_2,i\Big)
-\Gamma_1\Big(t,\Theta_1,\widehat\Theta_1,\Theta_2',\widehat\Theta_2',i\Big)\Big|  \\
&\qquad\qquad\qquad\qquad\qquad\qquad\qquad\qquad \leq {\color{black}{K_0}}\Big|\widetilde L_2(t,i)\Delta\Theta_2+\widetilde L_2'(t,i)\Delta \widehat\Theta_2\Big|, 
\end{aligned}\right.
\end{equation}
where $f_j=b_j,\sigma_j$, $\Delta h_j=h_j-h_j'$, $\Delta \widehat h_j=\widehat h_j-\widehat h_j'$ with $h=x_j,\widehat x_j, y_j,\widehat y_j, z_j,\widehat z_j$, $\Delta\Theta_2=\Theta_2-\Theta_2'$ and 
$\Delta \widehat\Theta_2=\widehat\Theta_2-\widehat\Theta_2'$.

2) (monotonicity condition) 
For any $(\omega, t,i)\in\Omega \times [0,T]\times\mathcal M$, any $\Theta_j, \widehat \Theta_j, \Theta_j', \widehat\Theta_j'$ $\in{\mathbb R}^{3n}$, $j=1,2$, and any $\Theta, {\widehat\Theta},  \Theta', \widehat\Theta'\in{\mathbb R}^{6n}$,
\begin{equation}\label{2.4} 
\left\{\begin{aligned}
&\langle\Psi_1\left(y_1,y_2\right)-\Psi_1\left(y_1',y_2\right),\Delta y_1\rangle
\leq -\mu_1\left|\widetilde M_1 \Delta y_1\right|^2,\\
&\mathbb{E}\left[\langle\Phi_1\left(x_1,\widehat x_1,x_2,\widehat x_2,i\right)-\Phi_1\left(x_1',
\widehat x_1',x_2,\widehat x_2,i\right),\Delta x_1\rangle|\mathcal F_{T-}^{\alpha}\right]\\
&\qquad\qquad\qquad\qquad\qquad\qquad\qquad\qquad\qquad
\geq \nu_1\mathbb{E}\left[\left|\widetilde G_1(i)\Delta x_1+\widetilde G_1'(i)\Delta \widehat x_1\right|^2\Big|\mathcal F_{T-}^{\alpha}\right],\\
&\mathbb{E}\left[\left\langle \Gamma_1\left(t,\Theta_1,\widehat\Theta_1,\Theta_2,\widehat\Theta_2,i\right)
-\Gamma_1\left(t,\Theta_1',\widehat\Theta_1',\Theta_2,\widehat\Theta_2,i\right),\Delta\Theta_1\right\rangle
\Big|\mathcal F_{t-}^{\alpha}\right]\\
&\qquad\quad
 \leq -\nu_1\mathbb{E}\left[\left|\widetilde A_1(t,i)\Delta x_1
+\widetilde A_1'(t,i)\Delta \widehat x_1\right|^2\Big|\mathcal F_{t-}^{\alpha}\right]\\
&\qquad\qquad
-\mu_1\mathbb{E}\left[\left|\widetilde B_1(t,i)\Delta y_1+\widetilde B_1'(t,i)\Delta\widehat y_1+\widetilde C_1(t,i)\Delta z_1+\widetilde C_1'(t,i)\Delta \widehat z_1 \right|^2\Big|\mathcal F_{t-}^{\alpha}\right],
\end{aligned}\right.
\end{equation}

\begin{equation}\label{2.5}
\left\{\begin{aligned}
&\langle\Psi_2\left(y_1,y_2\right)-\Psi_2\left(y_1,y_2'\right),\Delta y_2\rangle
\geq \mu_1\left|\widetilde M_1 \Delta y_2\right|^2,\\
&\mathbb{E}\left[\langle\Phi_2\left(x_1,\widehat x_1,x_2,\widehat x_2,i\right)
-\Phi_2\left(x_1,\widehat x_1,x_2',\widehat x_2',i\right),\Delta x_2\rangle|\mathcal F_{T-}^{\alpha}\right] \\
&\qquad\qquad\qquad\qquad\qquad\qquad\qquad\qquad\quad
\leq -\nu_1\mathbb{E}\left[\left|\widetilde G_1(i)\Delta x_2+\widetilde G_1'(i)\Delta\widehat x_2\right|^2\Big|\mathcal F_{T-}^{\alpha}\right],\\
&\mathbb{E}\left[\left\langle \Gamma_2\left(t,\Theta_1,\widehat\Theta_1,\Theta_2,\widehat\Theta_2,i\right)
-\Gamma_2\left(t,\Theta_1,\widehat\Theta_1,\Theta_2',\widehat\Theta_2',i\right),\Delta\Theta_2\right\rangle
\Big|\mathcal F_{t-}^{\alpha}\right]\\
&\qquad\quad 
\geq \nu_1\mathbb{E}\left[\left|\widetilde A_1(t,i)\Delta x_2
+\widetilde A_1'(t,i)\Delta \widehat x_2\right|^2\Big|\mathcal F_{t-}^{\alpha}\right]\\
&\qquad \qquad 
 +\mu_1\mathbb{E}\left[\left|\widetilde B_1(t,i)\Delta y_2
+\widetilde B_1'(t,i)\Delta\widehat y_2+\widetilde C_1(t,i)\Delta z_2
+\widetilde C_1'(t,i)\Delta\widehat z_2\right|^2\Big|\mathcal F_{t-}^{\alpha}\right],
\end{aligned}\right.
\end{equation}
and
\begin{equation}\label{A3-2}
\left\{\begin{aligned}
& \left\langle \Psi({\color{black}{\textbf{y}}})-\Psi({\color{black}{\textbf{y}}}'), \Lambda\Delta {\color{black}{\textbf{y}}}\right\rangle \leq -\mu_2\left|\widetilde M_2\Delta y_2\right|^2,\\
& \mathbb{E}\left[\left\langle \Phi\left({\color{black}{\textbf{x}}}, \widehat{\color{black}{{\textbf{x}}}},i\right)-\Phi\left({\color{black}{\textbf{x}}}',
\widehat{\color{black}{{\textbf{x}}}}',i\right),
\Lambda\Delta {\color{black}{\textbf{x}}}\right\rangle |\mathcal F_{T-}^{\alpha}\right]
\geq \mu_2\mathbb{E}\left[\left|\widetilde G_2(i)\Delta x_2+\widetilde G_2'(i)\Delta\widehat x_2\right|^2\Big|\mathcal F_{T-}^{\alpha}\right],\\
& \mathbb{E}\left[\left\langle \Gamma\left(t, \Theta,\widehat{\Theta},i\right)-\Gamma\left(t, \Theta',\widehat\Theta',i\right),\Lambda\Delta\Theta\right\rangle\Big|\mathcal F_{t-}^{\alpha}\right] \\
&\qquad\qquad\qquad\qquad\qquad\qquad\qquad
 \leq -\mu_2\mathbb{E}\left[\left|\widetilde L_2(t,i)\Delta \Theta_2
+\widetilde L_2'(t,i)\Delta\widehat\Theta_2\right|^2\Big|\mathcal F_{t-}^{\alpha}\right], 
\end{aligned}\right.
\end{equation}
where $\Lambda$ is a matrix to represent the inversion transform:
\[
\left(\begin{array}{ccccc}
\textbf{0} &\textbf{0} & \cdots &\textbf{I} \\
\vdots &\vdots  & \ddots & \vdots  \\
\textbf{0} &\textbf{I}  & \cdots  &\textbf{0} \\
\textbf{I} & \textbf{0} &\cdots &\textbf{0} \\
\end{array}\right),\]
and
\[\Lambda\Delta {\color{black}{\textbf{y}}}=
\left(\begin{array}{c}
\Delta y_2 \\
\Delta y_1 \\
\end{array}\right),\quad
\Lambda\Delta {\color{black}{\textbf{x}}}=
\left(\begin{array}{c}
\Delta x_2 \\
\Delta x_1 \\
\end{array}\right),\quad
\Lambda\Delta \Theta=
\left(\begin{array}{c}
\Delta \Theta_2 \\
\Delta \Theta_1 \\
\end{array}\right).\quad\]

It is remarkable that due to the presence of conditional expectation terms in \eqref{fbsde}, the domination and monotonicity conditions in Assumption (A2) are essentially different from the existing works, where the monotonicity conditions are in the form of conditional expectation. Utilizing continuation method and induction method, we can obtain the unique solvability of such CMF-FBSDEs. Now we present the existence and uniqueness result for CMF-FBSDEs \eqref{fbsde}.

\begin{Theorem}\label{eufbsde}
Let (A1)-(A2) hold with $(\Psi,\Phi,\Gamma)$. Then, \eqref{fbsde} admits a unique solution $\pi=\left(x_1,y_1,z_1,\tilde z_1,x_2,y_2,z_2,\tilde z_2\right)\in $ $\left(N_{\mathbb F}^2
\left(0,T; \mathbb R^{3n}\right)\times \mathcal M_{\mathbb F}^2\left(0,T; \mathbb R\right)\right)^2$. Besides, if there exists another solution $\pi'=\left(x_1',y_1',z_1',\tilde z_1',x_2',y_2',z_2',\tilde z_2'\right)\in \big(N_{\mathbb F}^2\left(0,T; \mathbb R^{3n}\right)$ $\times\mathcal M_{\mathbb F}^2\left(0,T; \mathbb R\right)\big)^2$ to  \eqref{fbsde} with $\left(\Psi',\Phi',\Gamma'\right)$ satisfying (A1)-(A2). Then, we have
\end{Theorem}
\begin{equation}\label{5-1}
\begin{aligned}
\mathbb{E}\bigg[\sup\limits_{t\in[0,T]}\left|\Delta{\color{black}{ \textbf{x}}}(t)\right|^2+\sup\limits_{t\in[0,T]}\left|\Delta {\color{black}{\textbf{y}}}(t)\right|^2
+\int_0^T\left|\Delta{\color{black}{ \textbf{z}}}(t)\right|^2dt+\int_0^T\left|\Delta{\color{black}{ \tilde{\textbf{z}}}}(t)\right|^2dt\bigg]
\leq K\mathbb{E}\left[\mathcal J\right],
\end{aligned}
\end{equation}
\emph{where} 
\begin{equation}\label{}
\begin{aligned}
\mathcal J=&\ \left|\Psi\left({\color{black}{\textbf{y}}}'(0)\right)-\Psi'\left({\color{black}{\textbf{y}}}'(0)\right)\right|^2+
\left|\Phi\left(\Lambda'_T\right)-\Phi'\left(\Lambda'_T\right)\right|^2 \\
&\ +\Big(\int_0^T\left|g\left(\pmb{\Theta}'_t\right)
-g'\left(\pmb{\Theta}'_t\right)\right|dt\Big)^2 
+\Big(\int_0^T\left|b\left(\pmb{\Theta}'_t\right)
-b'\left(\pmb{\Theta}'_t\right)\right|dt\Big)^2 \\
&\ +\int_0^T\left|\sigma\left(\pmb{\Theta}'_t\right)
-\sigma'\left(\pmb{\Theta}'_t\right)\right|^2dt,  \nonumber
\end{aligned}
\end{equation}
{\color{black}{$\textbf{v}=\left(v_1^{\top}, v_2^{\top}\right)^{\top}$ and $\Delta \textbf{v}=\textbf{v}-\textbf{v}'$ with $\textbf{v}=\textbf{x}, \textbf{y}, \textbf{z}, \tilde{\textbf{z}}$,}}
\emph{$K$ is a constant depending on $T$, the Lipschitz constants, $\mu_1,\nu_1,\mu_2$, and the bounds of all coefficients in (A1)-(A2), and }
\begin{equation}\left\{\begin{aligned}\label{fuhao}
& \Lambda_T'=\left({\color{black}{\textbf{x}}}'(T),\widehat{\color{black}{{\textbf{x}}}}'(T),\alpha_T\right),\\ 
& \pmb{\Theta}'_t=\big(t,{\Theta'(t)},{\widehat\Theta'(t)},\alpha_{t-}\big).
\end{aligned}\right.\end{equation}

\emph{Proof:} See Appendix A. \hfill$\square$

\section{LQ mixed Stackelberg-zero-sum game for stochastic system with regime switching}\label{S4}
To deal with Problem (CMF-SDG), this section is divided into two subsections. First, we study Problem (ZSDG) between the two followers, and establish their equilibrium strategies. Subsequently, we solve the leader's optimization Problem (LOC), and then complete the Stackelberg game solution framework.

\subsection{Optimal strategies of zero-sum differential game for two followers}
We present a definition as follows.
\begin{Definition}\label{Def1}
The open-loop lower value $V_{F}^{-}(x_0)$ and the open-loop upper value $V_{F}^{+}(x_0)$ at the initial state $x_0\in\mathbb R$ are defined by 
\begin{equation*}\label{}
\left\{\begin{aligned}
&V_{F}^{-}(x_0)=\sup_{u_{F,2}(\cdot)\in\mathcal{U}_{F,2}}\inf_{u_{F,1}(\cdot)\in\mathcal{U}_{F,1}}J_F\left(x_0; u_L(\cdot), \mathbf{u_{F_{1,2}}}(\cdot)\right), \\
&V_{F}^{+}(x_0)=\inf_{u_{F,1}(\cdot)\in\mathcal{U}_{F,1}}\sup_{u_{F,2}(\cdot)\in\mathcal{U}_{F,2}}J_F\left(x_0; u_L(\cdot), \mathbf{u_{F_{1,2}}}(\cdot)\right). 
\end{aligned}\right.
\end{equation*} 
If $V_{F}^{-}(x_0)=V_{F}^{+}(x_0)$, then we call it an open-loop value at the initial state $x_0$.
\end{Definition}

In what follows, we start to solve Problem (ZSDG), and give the representation form of the strategies for two followers. To do this, we give the following assumptions.

\noindent\textbf{Assumption (F1)}
$V_{F}^{+}(x_0)$ and $V_{F}^{-}(x_0)$ are finite.

\noindent\textbf{Assumption (F2)}
For any $i\in\mathcal M$, $\mathbf{R_F}(\cdot,i)$ is invertible.

\noindent\textbf{Assumption (F3)}
For any $i\in\mathcal M$, $Q_F(\cdot, i)$, $G_F(i)$,
$\mathbf{B_F}(\cdot,i)\mathbf{R_F^{-1}}(\cdot,i)\mathbf{B_F^{\top}}(\cdot,i)$, $\mathbf{D_F}(\cdot,i)$ \\
$\times\mathbf{R_F^{-1}}(\cdot,i)\mathbf{D_F^{\top}}(\cdot,i)$, 
and $\mathbf{B_F}(\cdot,i)\mathbf{R_F^{-1}}(\cdot,i)\mathbf{D_F^{\top}}(\cdot,i)$ are all non-negative, 
and $\bar Q_F(\cdot,i)=$\\ $\bar G_F(i)=0$. Moreover, for any $i\in\mathcal M$,
$\mathbf{B_F^{\top}}(\cdot,i)\mathbf{D_F}(\cdot,i)=\mathbf{D_F^{\top}}(\cdot,i)\mathbf{B_F}(\cdot,i)$.

{\color{black}{Hereafter, for notation convenience, $x(\cdot)$ in system \eqref{state} is denoted by $x_1(\cdot)$.}} We now present the following theorem that provides the open-loop saddle point for the two followers. 
\begin{Theorem}\label{OL}
Let Assumptions (F1) and (F2) hold. Then, a strategy $\mathbf{u^*_{F_{1,2}}}(\cdot)\in\mathbf{\mathcal U_{F_{1,2}}}$ is an open-loop saddle point if and only if
\begin{equation}\label{u-pq} 
\mathbf{u^*_{F_{1,2}}}(t)=-\mathbf{R_F^{-1}}(t,\alpha_{t-})\big(\mathbf{B^{\top}_F}(t,\alpha_{t-})
{\color{black}{y_1}}(t)+\mathbf{D_F^{\top}}(t,\alpha_{t-}){\color{black}{z_1}}(t)\big),
\end{equation}
where $\big({\color{black}{x_1}}(\cdot), {\color{black}{y_1}}(\cdot), {\color{black}{z_1}}(\cdot), {\color{black}{\tilde z_1}}(\cdot)\big)$ is the solution to Hamiltonian system
\begin{equation}\label{p}
\left\{\begin{aligned}
& d{\color{black}{x_1}}(t)=\big[A(t,\alpha_{t-}){\color{black}{x_1}}(t)+\bar{A}(t,\alpha_{t-})
{\color{black}{\widehat x_1}}(t)+B_L(t,\alpha_{t-})u_L(t)
-\mathbf{B_F}(t,\alpha_{t-})\\
&\quad \times\mathbf{R_F^{-1}}(t,\alpha_{t-})\big(\mathbf{B^{\top}_F}(t,\alpha_{t-}){\color{black}{y_1}}(t) 
+\mathbf{D_F^{\top}}(t,\alpha_{t-}){\color{black}{z_1}}(t)\big)+b(t,\alpha_{t-})\big]dt \\
&\quad +\big[C(t,\alpha_{t-}){\color{black}{x_1}}(t) 
+\bar C(t,\alpha_{t-}){\color{black}{\widehat x_1}}(t)
+D_L(t,\alpha_{t-})u_L(t)-\mathbf{D_F}(t,\alpha_{t-})\mathbf{R_F^{-1}}(t,\alpha_{t-})  \\
&\quad\times \big(\mathbf{B^{\top}_F}(t,\alpha_{t-}){\color{black}{y_1}}(t) 
+\mathbf{D_F^{\top}}(t,\alpha_{t-}){\color{black}{z_1}}(t)\big)+\sigma(t,\alpha_{t-})\big]dW(t), \\
&d{\color{black}{y_1}}(t)= -\big[A(t,\alpha_{t-}){\color{black}{y_1}}(t)+\bar A(t,\alpha_{t-})
{\color{black}{\widehat y_1}}(t)+C(t,\alpha_{t-}){\color{black}{z_1}}(t)
+\bar C(t,\alpha_{t-}){\color{black}{\widehat z_1}}(t)    \\
&\quad  +Q_F(t,\alpha_{t-}){\color{black}{x_1}}(t)+\bar Q_F(t,\alpha_{t-})
{\color{black}{\widehat x_1}}(t)\big]dt+{\color{black}{z_1}}(t)dW(t)
+{\color{black}{\tilde z_1}}(t)\bullet dM(t),\\
& {\color{black}{x_1}}(0)= x_0,\quad 
{\color{black}{y_1}}(T)=G_F(\alpha_T){\color{black}{x_1}}(T)+\bar G_F(\alpha_T){\color{black}{\widehat x_1}}(T).
\end{aligned}\right.
\end{equation}
\end{Theorem}

\emph{Proof:} See Appendix B. \hfill$\square$

Notice that stochastic Hamiltonian system \eqref{p} is a {\color{black}{fully coupled}} CMF-FBSDE with Markovian switching. {\color{black}{Moreover, \eqref{p} can be regarded as a special case of \eqref{fbsde}, where the functions $b_1, \sigma_1, g_1, \Psi_1,\Phi_1$ are of linear forms, and $b_2=\sigma_2=g_2=\Psi_2=\Phi_2=0$.
Under Assumption (F3), we can check that \eqref{p} is uniquely solvable.}} In the following, we will decouple \eqref{p} and present a feedback representation of the open-loop saddle point for the two followers.

Introduce the auxiliary equations as follows:
\begin{equation}\label{eq1}
\left\{\begin{aligned}
&\dot{P}_F(t,i)+\left(2A(t,i)+C^2(t,i)\right)P_F(t,i)+Q_F(t,i)-\mathbb{B}_F^{\top}(t,i)
\mathbb R_F^{-1}(t,i)\mathbb{B}_F(t,i) \\
&\quad +\sum_{j\in\mathcal M}\lambda_{ij}\big[P_F(t,j)-P_F(t,i)\big]=0, \\
& \mathbb R_F(t,i) \ \text{is invertible}, \\
&P_F(T,i)=G_F(i), \quad i\in\mathcal M,
\end{aligned}\right.
\end{equation}
\vspace*{-0.3cm}
\begin{equation}\label{eq2}
\left\{\begin{aligned}
&\dot{\bar P}_F(t,i)+2\left(A(t,i)+\bar A(t,i)\right)\bar P_F(t,i)
+\left(2\bar A(t,i)+\bar C^2(t,i)+2C(t,i)\bar C(t,i)\right) \\
&\quad \times P_F(t,i)+\bar Q_F(t,i)-\mathbb{B}_F^{\top}(t,i) \mathbb R_F^{-1}(t,i) \bar{\mathbb{B}}_F(t,i) 
-\bar{\mathbb{B}}_F^{\top}(t,i) \mathbb R_F^{-1}(t,i) \bar{\mathbb{B}}_F(t,i) \\
&\quad -\bar{\mathbb{B}}_F^{\top}(t,i)\mathbb R_F^{-1}(t,i)\mathbb{B}_F(t,i)
+\sum_{j\in\mathcal M}\lambda_{ij}\Big[\bar P_F(t,j)-\bar P_F(t,i)\Big]=0, \\
&\bar P_F(T,i)=\bar G_F(i), \quad i\in\mathcal M,
\end{aligned}\right.
\end{equation}
\vspace*{-0.3cm}
\begin{equation}\label{eq3}
\left\{\begin{aligned}
& d\phi(t)=-\Big[\widetilde{\Phi}^A_{\mathbb B\mathbf B}(t,\alpha_{t-})\phi(t)
+\widetilde{\Phi}^{\bar A}_{\bar{\mathbb B}\mathbf B}
(t,\alpha_{t-})\widehat\phi(t)+\widetilde{\Phi}^C_{\mathbb B\mathbf D}(t,\alpha_{t-})\phi_W(t) \\
&\quad +\widetilde{\Phi}^{\bar C}_{\bar{\mathbb B}\mathbf D}(t,\alpha_{t-})\widehat{\phi}_W(t)  
+\widetilde{\Psi}^C_{P\mathbb B}(t,\alpha_{t-})u_L(t)
+\widetilde{\Psi}^{\bar C}_{\bar P\bar{\mathbb B}}(t,\alpha_{t-})\widehat u_L(t) \\
&\quad +\widetilde{\Gamma}_{C\mathbb B}(t,\alpha_{t-})\sigma(t,\alpha_{t-})
+\widetilde{\Gamma}_{\bar C\bar{\mathbb B}}(t,\alpha_{t-})\widehat\sigma(t,\alpha_{t-}) 
+P_F(t,\alpha_{t-})b(t,\alpha_{t-})\\
&\quad +\bar P_F(t,\alpha_{t-})\widehat b(t,\alpha_{t-})\Big]dt
+\phi_W(t)dW(t)+\phi_M(t)\bullet dM(t), \\
&\phi(T)=0,
\end{aligned}\right.
\end{equation}
where
\vspace*{-0.3cm}
\begin{equation}\label{BRF}
\left\{\begin{aligned}
& \mathbb R_F(\cdot,i)= \mathbf{R_F}+\mathbf{D_F^{\top}}P_F\mathbf{D_F},\quad 
\mathbb{B}_F(\cdot,i)=\mathbf{B_F^{\top}}P_F+\mathbf{D_F^{\top}}P_FC, \\
& \bar{\mathbb{B}}_F(\cdot,i)=\mathbf{B_F^{\top}}\bar P_F+\mathbf{D_F^{\top}}P_F\bar C, \quad
\widetilde{\Phi}^X_{YZ}(\cdot,i)=X-Y_F^{\top}\mathbb{R}_F^{-1}Z_F^{\top},\\
& \widetilde{\Psi}^X_{YZ}(\cdot,i)=Y_FB_L+XP_FD_L-Z_F^{\top}\mathbb{R}_F^{-1}\mathbf{D_F}P_FD_L,\\ 
& \widetilde{\Gamma}_{XY}(\cdot,i)=XP_F-Y_F^{\top}\mathbb{R}_F^{-1}\mathbf{D_F}P_F.  
\end{aligned}\right.
\end{equation}

Note that \eqref{p} is fully coupled, it is difficult to obtain the solution explicitly. By virtue of introduction 
of \eqref{eq1}-\eqref{eq3}, ${\color{black}{y_1}}(\cdot)$ and ${\color{black}{z_1}}(\cdot)$ given by \eqref{p} can be expressed as 
\vspace*{-0.3cm}
\begin{equation}\left\{\begin{aligned}\label{pq-x}
{\color{black}{y_1}}(t)=&\ P_F(t,\alpha_{t-}){\color{black}{x_1}}(t)+\bar P_F(t,\alpha_{t-})
{\color{black}{\widehat x_1}}(t)+\phi(t), \\
{\color{black}{z_1}}(t)=&\ P_F(t,\alpha_{t-})\big[C(t,\alpha_{t-}){\color{black}{x_1}}(t)
+\bar C(t,\alpha_{t-}){\color{black}{\widehat x_1}}(t)+D_L(t,\alpha_{t-})u_L(t) \\
&\ +\mathbf{D_{F}}(t,\alpha_{t-})\mathbf{u_{F_{1,2}}^*}(t)
+\sigma(t,\alpha_{t-})\big]+\phi_W(t),
\end{aligned}\right.\end{equation}
where $P_F(\cdot,i), \bar P_F(\cdot,i), \big(\phi(\cdot),\phi_W(\cdot)\big)$ are represented by \eqref{eq1}-\eqref{eq3}, $i\in\mathcal M$. Based on \eqref{pq-x}, the strategy $\mathbf{u^*_{F_{1,2}}}(\cdot)$ in \eqref{u-pq} can be expressed as an affine function of state process ${\color{black}{x_1}}(\cdot)$ and its conditional expectation ${\color{black}{\widehat x_1}}(\cdot)$. In this sense, we can obtain a state feedback form for $\mathbf{u^*_{F_{1,2}}}(\cdot)$.

It is remarkable that, unlike the classical Riccati equations in Yong and Zhou \cite{Yongzhou}, \eqref{eq1} is a system of Riccati equations with Markov chain, and they are coupled via the term $\sum_{j\in\mathcal M}\lambda_{ij}\big[P_F(\cdot,j)-P_F(\cdot,i)\big]$. As a result, the methods addressed in Yong and Zhou \cite{Yongzhou} are unavailable to prove the solvability of \eqref{eq1}. Thus, \eqref{eq1} and \eqref{eq2} are essentially different from the classical Riccati equations. Next, we focus on the solvability of \eqref{eq1}-\eqref{eq3}.

For the coefficients in \eqref{state}-\eqref{cost-f}, for any $i\in\mathcal M$ and $k=1,2$, there exist constants 
$q_0, g_0, \bar c_2, c_3\geq0$ and $c_1>0$, such that
\begin{equation}\begin{aligned}
&  D_{F,k}^{\top}(\cdot,i)D_{F,k}(\cdot,i)\leq \bar c_2\mathbf{I}_{m_k},\quad |G_F(i)|\leq g_0, \quad
 2|A(\cdot,i)|+C^2(\cdot,i)+\sum_{j\in\mathcal M}|\lambda_{ij}|\leq c_1, \\
& |Q_F(\cdot,i)|\leq q_0,\quad \left(B_{F,k}(\cdot,i)+D_{F,k}(\cdot,i)C(\cdot,i)\right)\left(B_{F,k}^{\top}(\cdot,i)
+D_{F,k}^{\top}(\cdot,i)C(\cdot,i)\right)\leq c_3.  \nonumber
\end{aligned}\end{equation}
\vspace*{-0.5cm}
\begin{equation*}
\text{Define} \quad \varrho=\left[\left(e^{c_1mT}-1\right)\frac{q_0}{c_1m}+g_0e^{c_1mT}\right]
\frac{4c_3\left(e^{c_1mT}-1\right)}{c_1m},  \quad
\bar\varrho=\frac{c_1m\varrho}{2c_3\left(e^{c_1mT}-1\right)}.
\end{equation*}
\noindent\textbf{Assumption (F4)}
There exists a constant $\underline{c_2}\geq 0$ with $\underline{c_2}\leq \bar c_2$, such that $D_{F,k}^{\top}(\cdot,i)$\\
$\times D_{F,k}(\cdot,i)\geq \underline{c_2}\mathbf{I}_{m_k},$  $i\in\mathcal M$, $k=1,2$.

\noindent\textbf{Assumption (F5)}
$R_{F,1}(\cdot,i) \geq (\varrho+\bar\varrho\bar c_2)\mathbf{I}_{m_1}$ and {\color{black}{$R_{F,2}(\cdot,i)\leq -(\varrho+\bar\varrho\bar c_2)\mathbf{I}_{m_2}$}}, $i\in\mathcal M$.

\begin{Theorem}\label{riccati}
With Assumptions (F1)-(F5), \eqref{eq1}-\eqref{eq3} admit unique solutions. Moreover, suppose that for any $i\in\mathcal M$, $\mathbf{R_F}(\cdot,i)+\mathbf{D_F^{\top}}(\cdot,i)\mathbf{D_F}(\cdot,i)P_F(\cdot,i)$ is invertible, then the optimal strategy $\mathbf{u^*_{F_{1,2}}}(\cdot)$ is uniquely solvable with closed-loop representation 
\begin{equation}\begin{aligned}\label{closed-ll}
\mathbf{u^*_{F_{1,2}}}(t)=&\ -\Big(\mathbf{R_F}(t,\alpha_{t-})
+\mathbf{D_F^{\top}}(t,\alpha_{t-})\mathbf{D_F}(t,\alpha_{t-})P_F(t,\alpha_{t-})\Big)^{-1} \\
&\ \times\Big[\Big(\mathbf{B_F^{\top}}(t,\alpha_{t-})  
+\mathbf{D_F^{\top}}(t,\alpha_{t-})C(t,\alpha_{t-})\Big)P_F(t,\alpha_{t-})x_1^*(t) \\
&\ +\Big(\mathbf{B_F^{\top}}(t,\alpha_{t-})\bar P_F(t,\alpha_{t-})
+\mathbf{D_F^{\top}}(t,\alpha_{t-})P_F(t,\alpha_{t-})C(t,\alpha_{t-})\Big)\widehat x_1^*(t) \\
&\ +\mathbf{B_F^{\top}}(t,\alpha_{t-})\phi(t)+\mathbf{D_F^{\top}}(t,\alpha_{t-})
\Big(P_F(t,\alpha_{t-})D_L(t,\alpha_{t-})u_L(t) \\
&\ +P_F(t,\alpha_{t-})\sigma(t,\alpha_{t-})+\phi_W(t)\Big)\Big],
\end{aligned}\end{equation} 
where $x_1^*(\cdot)$ is the optimal state satisfying 
\begin{equation}\label{closed-x}
\left\{\begin{aligned}
d{\color{black}{x_1^*}}(t)=&\ \big[A(t,\alpha_{t-}){\color{black}{x_1^*}}(t)+\bar{A}(t,\alpha_{t-})
{\color{black}{\widehat x_1^*}}(t)+B_L(t,\alpha_{t-})u_L(t)+\mathbf{B_F}(t,\alpha_{t-})\mathbf{u^*_{F_{1,2}}}(t) \\
&\ +b(t,\alpha_{t-})\big]dt+\big[C(t,\alpha_{t-}){\color{black}{x_1^*}}(t) 
+\bar C(t,\alpha_{t-}){\color{black}{\widehat x_1^*}}(t)+D_L(t,\alpha_{t-})u_L(t) \\
&\ +\mathbf{D_F}(t,\alpha_{t-})\mathbf{u^*_{F_{1,2}}}(t)+\sigma(t,\alpha_{t-})\big]dW(t), \\
{\color{black}{x_1^*}}(0)=&\ x_0,
\end{aligned}\right.
\end{equation}
and $P_F(\cdot,i), \bar P_F(\cdot,i), \big(\phi(\cdot),\phi_W(\cdot)\big)$ are given by \eqref{eq1}-\eqref{eq3}.
\end{Theorem}

\emph{Proof:} See Appendix C. \hfill$\square$

\subsection{Optimal strategy for the leader}
We now proceed to investigate Problem (LOC) for the leader. This is a stochastic optimal control problem governed by CMF-FBSDE \eqref{p} with the aim of minimizing the cost functional
\begin{equation}\label{}
\begin{aligned}
&J_L\big(x_0; u_L(\cdot), \mathbf{u^*_{F_{1,2}}}(\cdot)\big)
=\frac{1}{2}\mathbb{E}\bigg\{\int_0^T\Big[Q_L(t,\alpha_{t-}){\color{black}{x_1^2}}(t)
+\bar Q_L(t,\alpha_{t-}){\color{black}{\widehat x_1^2}}(t) \\
&\qquad\qquad\qquad\qquad  +u_{L}^{\top}(t)R_{L}(t,\alpha_{t-})u_{L}(t)\Big]dt 
+G_L(\alpha_{T}){\color{black}{x_1^2}}(T)+\bar G_L(\alpha_{T}){\color{black}{\widehat x_1^2}}(T)\bigg\}.
\nonumber	
\end{aligned}
\end{equation}

\noindent\textbf{Assumption (L1)}
$Q_L(\cdot, i)\geq 0,\ \bar Q_L(\cdot,i)\geq 0,\ R_L(\cdot,i)>\textbf{0}, \ G_L(i)\geq 0,\ \bar G_L(i)\geq 0,\ i\in\mathcal M$. 

\noindent\textbf{Assumption (L2)}
(i) For any $i\in\mathcal M$, $B_L^{\top}(\cdot,i)D_L(\cdot,i)=D_L^{\top}(\cdot,i)B_L(\cdot,i)$, and\\ 
$B_L(\cdot,i)R_L^{-1}(\cdot,i)D_L^{\top}(\cdot,i)\geq 0$.  
(ii) For any $i\in\mathcal M$, $\bar C(\cdot,i)=0$, $D_L(\cdot,i)=\mathbf{D_F}(\cdot,i)=\textbf{0}$, $B_L\neq \textbf{0}$,  $\mathbf{B_F}(\cdot,i)\mathbf{R_F^{-1}}(\cdot,i)\mathbf{B_F^{\top}}(\cdot,i)\neq 0$.

Introduce these two fully coupled CMF-FBSDEs
\begin{equation}\label{2fbsde}
\left\{\begin{aligned}
& d{\color{black}{x_1}}(t)
=\big[A(t,\alpha_{t-}){\color{black}{x_1}}(t)+\bar{A}(t,\alpha_{t-}){\color{black}{\widehat x_1}}(t)
-\mathbf{B_F}(t,\alpha_{t-})\mathbf{R_F^{-1}}(t,\alpha_{t-}) \\
&\quad \times\big(\mathbf{B_F^{\top}}(t,\alpha_{t-}){\color{black}{y_1}}(t)
+\mathbf{D_F^{\top}}(t,\alpha_{t-}){\color{black}{z_1}}(t)\big)-B_L(t,\alpha_{t-})R_L^{-1}(t,\alpha_{t-})
\big(B_L^{\top}(t,\alpha_{t-})\\
&\quad \times {\color{black}{y_2}}(t)+D_L^{\top}(t,\alpha_{t-}){\color{black}{z_2}}(t)\big)
+b(t,\alpha_{t-})\big]dt+\big[C(t,\alpha_{t-}){\color{black}{x_1}}(t)
+\bar{C}(t,\alpha_{t-}){\color{black}{\widehat x_1}}(t) \\
&\quad-\mathbf{D_F}(t,\alpha_{t-})\mathbf{R_F^{-1}}(t,\alpha_{t-})
\big(\mathbf{B_F^{\top}}(t,\alpha_{t-}){\color{black}{y_1}}(t) 
 +\mathbf{D_F^{\top}}(t,\alpha_{t-}){\color{black}{z_1}}(t)\big)-D_L(t,\alpha_{t-})\\
&\quad \times R_L^{-1}(t,\alpha_{t-})\big(B_L^{\top}(t,\alpha_{t-}){\color{black}{y_2}}(t)
+D_L^{\top}(t,\alpha_{t-}){\color{black}{z_2}}(t)\big)+\sigma(t,\alpha_{t-})\big]dW(t), \\
& d{\color{black}{y_1}}(t)= -\big[A(t,\alpha_{t-}){\color{black}{y_1}}(t)
+\bar A(t,\alpha_{t-}){\color{black}{\widehat y_1}}(t)
+C(t,\alpha_{t-}){\color{black}{z_1}}(t)+\bar C(t,\alpha_{t-}){\color{black}{\widehat{z}_1}}(t)  \\
&\quad +Q_F(t,\alpha_{t-}){\color{black}{x_1}}(t)+\bar Q_F(t,\alpha_{t-}){\color{black}{\widehat x_1}}(t)\big]dt+{\color{black}{z_1}}(t)dW(t)+{\color{black}{\tilde z_1}}(t)\bullet dM(t), \\
&d{\color{black}{x_2}}(t)=\big[A(t,\alpha_{t-}){\color{black}{x_2}}(t)
+\bar{A}(t,\alpha_{t-}){\color{black}{\widehat x_2}}(t)
+\mathbf{B_F}(t,\alpha_{t-})\mathbf{R_F^{-1}}(t,\alpha_{t-})\big(\mathbf{B_F^{\top}}(t,\alpha_{t-})\\
&\quad \times {\color{black}{y_2}}(t)+\mathbf{D_F^{\top}}(t,\alpha_{t-}){\color{black}{z_2}}(t)\big)\big]dt 
 +\big[C(t,\alpha_{t-}){\color{black}{x_2}}(t)+\bar{C}(t,\alpha_{t-}){\color{black}{\widehat x_2(t)}} \\
&\quad +\mathbf{D_F}(t,\alpha_{t-})\mathbf{R_F^{-1}}(t,\alpha_{t-})
\big(\mathbf{B_F^{\top}}(t,\alpha_{t-}){\color{black}{y_2}}(t)
+\mathbf{D_F^{\top}}(t,\alpha_{t-}){\color{black}{z_2}}(t)\big)\big]dW(t), \\
& d{\color{black}{y_2}}(t)=-\big[A(t,\alpha_{t-}){\color{black}{y_2}}(t)+\bar A(t,\alpha_{t-}){\color{black}{\widehat y_2}}(t)+C(t,\alpha_{t-}){\color{black}{z_2}}(t)
+\bar C(t,\alpha_{t-}){\color{black}{\widehat z_2}}(t) \\ 
&\quad -Q_F(t,\alpha_{t-}){\color{black}{x_2}}(t)-\bar Q_F(t,\alpha_{t-}){\color{black}{\widehat x_2}}(t)+Q_L(t,\alpha_{t-}){\color{black}{x_1}}(t)
+\bar Q_L(t,\alpha_{t-}){\color{black}{\widehat x_1}}(t)\big]dt \\
&\quad +{\color{black}{z_2}}(t)dW(t)+{\color{black}{\tilde z_2}}(t)\bullet dM(t), \\
& {\color{black}{x_1}}(0)=x_0,\quad {\color{black}{y_1}}(T)=G_F(\alpha_T){\color{black}{x_1}}(T)
+\bar G_F(\alpha_T){\color{black}{\widehat x_1}}(T), \quad {\color{black}{x_2}}(0)=0,\\
& {\color{black}{y_2}}(T)=G_L(\alpha_T){\color{black}{x_1}}(T)
+\bar G_L(\alpha_T){\color{black}{\widehat x_1}}(T)-G_F(\alpha_T){\color{black}{x_2}}(T)
-\bar G_F(\alpha_T){\color{black}{\widehat x_2}}(T).
\end{aligned}\right.
\end{equation}
Note that {\color{black}{\eqref{2fbsde} is a linear form of \eqref{fbsde}.}} In what follows, we discuss the unique solvability of \eqref{2fbsde}.

\begin{Lemma}\label{L4.3}
Let Assumptions (F2), (F3), (L1) and (L2)-\emph{(i)} hold, then Assumptions (A1) and (A2) hold true for the coefficients of \eqref{2fbsde}.
\end{Lemma}

\emph{Proof:} See Appendix D. \hfill$\square$

Now we present the open-loop optimal strategy for the leader.

\begin{Theorem}\label{th-F}
Let Assumptions (L1) and (L2)-\emph{(i)} hold. Then, Problem (LOC) admits a unique optimal strategy given by
\begin{equation}\label{uL*}
u_L^*(t)=-R_L^{-1}(t,\alpha_{t-})\big(B_L^{\top}(t,\alpha_{t-}){\color{black}{y_2}}(t)
+D_L^{\top}(t,\alpha_{t-}){\color{black}{z_2}}(t)\big),
\end{equation}
where $\left({\color{black}{y_2}}(\cdot),{\color{black}{z_2}}(\cdot)\right)$ is given by \eqref{2fbsde}. Moreover, $\big(\mathbf{u^*_{F_{1,2}}}(\cdot),u_L^*(\cdot)\big)$ defined by \eqref{u-pq} and \eqref{uL*} is the unique open-loop Stackelberg equilibrium of Problem (CMF-SDG).
\end{Theorem}

\emph{Proof:} See Appendix E. \hfill$\square$

Now we expect to give the feedback form of $u_L^*(\cdot)$ by some Riccati equations. Denote
\begin{equation}\label{COEF}
\begin{gathered}
\mathcal X=
\begin{pmatrix}
{\color{black}{x_1}} \\
{\color{black}{x_2}} \\
\end{pmatrix},\quad  
\mathcal Y=
\begin{pmatrix}
{\color{black}{y_2}} \\
{\color{black}{y_1}} \\
\end{pmatrix},\quad
\mathcal Z=
\begin{pmatrix}
{\color{black}{z_2}} \\
{\color{black}{z_1}} \\
\end{pmatrix},\quad 
\widetilde {\mathcal Z}=
\begin{pmatrix}
{\color{black}{\tilde z_2}} \\
{\color{black}{\tilde z_1}} \\
\end{pmatrix},\quad 
\mathcal X_0=
\begin{pmatrix}
x_0 \\
0 \\
\end{pmatrix},\\
\mathcal B_1=\begin{pmatrix}
-B_LR_L^{-1}B_L^{\top}  & -\mathbf{B_F}\mathbf{R_F^{-1}}\mathbf{B_F^{\top}} \\
\mathbf{B_F}\mathbf{R_F^{-1}}\mathbf{B_F^{\top}}  & 0  \\
\end{pmatrix}, \quad
\mathcal B_2=\begin{pmatrix}
-B_LR_L^{-1}D_L^{\top}  & -\mathbf{B_F}\mathbf{R_F^{-1}}\mathbf{D_F^{\top}} \\
\mathbf{B_F}\mathbf{R_F^{-1}}\mathbf{D_F^{\top}}  & 0  \\
\end{pmatrix}, \\
\mathcal D_1=\begin{pmatrix}
-D_LR_L^{-1}B_L^{\top}  & -\mathbf{D_F}\mathbf{R_F^{-1}}\mathbf{B_F^{\top}} \\
\mathbf{D_F}\mathbf{R_F^{-1}}\mathbf{B_F^{\top}}  & 0  \\
\end{pmatrix}, \quad
\mathcal D_2=\begin{pmatrix}
-D_LR_L^{-1}D_L^{\top}  & -\mathbf{D_F}\mathbf{R_F^{-1}}\mathbf{D_F^{\top}} \\
\mathbf{D_F}\mathbf{R_F^{-1}}\mathbf{D_F^{\top}}  & 0  \\
\end{pmatrix}, \\
\mathcal H=\begin{pmatrix}
 H & 0\\
0 & H\\
\end{pmatrix} \ \text{with} \ H=A, \bar A, C, \bar C,\quad
\mathcal K=\begin{pmatrix}
K_L & -K_F  \\
K_F & 0 \\
\end{pmatrix} \  \text{with} \ K=Q, \bar Q, G, \bar G, \\
\mathcal B=\begin{pmatrix}
B_L \\
0 \\
\end{pmatrix}, \quad
\mathcal D=\begin{pmatrix}
D_L \\
0 \\
\end{pmatrix}, \quad
\mathcal E=\begin{pmatrix}
1 \\
0 \\
\end{pmatrix}.
\end{gathered}
\end{equation}
Then, \eqref{2fbsde} and \eqref{uL*} are written as
\begin{equation}\label{XY}
\left\{\begin{aligned}
&d\mathcal X(t)=\big[\mathcal A(t,\alpha_{t-})\mathcal X(t)+\bar{\mathcal A}(t,\alpha_{t-})\widehat{\mathcal X}(t)+\mathcal B_1(t,\alpha_{t-})\mathcal Y(t)+\mathcal B_2(t,\alpha_{t-})\mathcal Z(t) \\
&\quad +\mathcal{E} b(t)\big]dt+\big[\mathcal C(t,\alpha_{t-})\mathcal X(t)+\bar{\mathcal C}(t,\alpha_{t-})\widehat{\mathcal X}(t)+\mathcal D_1(t,\alpha_{t-})\mathcal Y(t)+\mathcal D_2(t,\alpha_{t-})\mathcal Z(t) \\
&\quad +\mathcal{E}\sigma(t)\big]dW(t), \\
&d\mathcal Y(t)=-\big[\mathcal Q(t,\alpha_{t-})\mathcal X(t)+\bar{\mathcal Q}(t,\alpha_{t-})\widehat{\mathcal X}(t)+\mathcal A(t,\alpha_{t-})\mathcal Y(t)+\bar{\mathcal A}(t,\alpha_{t-})\widehat{\mathcal Y}(t) \\
&\quad +\mathcal C(t,\alpha_{t-})\mathcal Z(t)+\bar{\mathcal C}(t,\alpha_{t-})\widehat{\mathcal Z}(t)\big]dt+\mathcal Z(t)dW(t)+\widetilde{\mathcal Z}(t)\bullet dM(t),\\
& \mathcal X(0)={\mathcal X}_0, \quad
\mathcal Y(T)=\mathcal G(\alpha_{T})\mathcal X(T)+\bar{\mathcal G}(\alpha_{T})\widehat{\mathcal X}(T),
\end{aligned}\right.
\end{equation}
and
\begin{equation}\label{opuL}
u_L^*(t)=-R_L^{-1}(t,\alpha_{t-})\big[\mathcal{B}^{\top}(t,\alpha_{t-})\mathcal Y(t)
+\mathcal{D}^{\top}(t,\alpha_{t-})\mathcal Z(t)\big].
\end{equation}

To decouple \eqref{XY}, we introduce the following auxiliary equations:
\begin{equation}\label{PL-1}
\left\{\begin{aligned}
& \dot{P}_L(t,i)+P_L(t,i)\mathcal A(t,i)+\mathcal A(t,i)P_L(t,i)+P_L(t,i)\mathcal B_1(t,i)P_L(t,i)
+\mathcal Q(t,i) \\
&\quad +\sum_{j\in\mathcal M}\lambda_{ij}\big[P_L(t,j)-P_L(t,i)\big]=\textbf{0},\\
&P_L(T,i)=\mathcal G(i), \quad i\in\mathcal M,
\end{aligned}\right.
\end{equation}
\begin{equation}\label{PL-2}
\left\{\begin{aligned}
&\dot{\bar P}_L(t,i)+\bar P_L(t,i)\mathcal B_1(t,i)\bar P_L(t,i) 
+\Big[2\big(\mathcal A(t,i)+\bar{\mathcal A}(t,i)\big)+P_L(t,i)\mathcal B_1(t,i)\Big]\bar P_L(t,i) \\
&\quad +\bar P_L(t,i)\Big[\mathcal B_1(t,i)+\mathcal B_2(t,i)\mathcal C(t,i)\Big]P_L(t,i) 
 +2\mathcal A(t,i)P_L(t,i)+\bar{\mathcal Q}(t,i) \\
&\quad +\sum_{j\in\mathcal M}\lambda_{ij}\Big[\bar P_L(t,j)-\bar P_L(t,i)\Big]=\textbf{0}, \\
&\bar P_L(T,i)=\bar{\mathcal G}(i), \quad i\in\mathcal M,
\end{aligned}\right.
\end{equation}
and
\begin{equation}\label{PL-3}
\left\{\begin{aligned}
&d\tau(t)=-\Big\{\mathbb{B}^{1}_{{\mathcal A}P_L}(t,\alpha_{t-})\tau(t) 
+\mathbb{B}^{1}_{\bar{\mathcal A}\bar P_L}(t,\alpha_{t-})\widehat{\tau}(t)
 +\mathbb{B}^{2}_{{\mathcal C}P_L}(t,\alpha_{t-})P_L(t,\alpha_{t-})  \\
&\quad \times \big(P_L(t,\alpha_{t-})\mathcal{E}\sigma(t,\alpha_{t-})+\tau_W(t)\big) 
  +\bar P_L(t,\alpha_{t-})\mathcal B_2(t,\alpha_{t-})P_L(t,\alpha_{t-}) \\
&\quad\times \big(P_L(t,\alpha_{t-})\mathcal{E}\widehat{\sigma}(t,\alpha_{t-})+\widehat{\tau}_W(t)\big) 
+P_L(t,\alpha_{t}){\mathcal E}b(t,\alpha_{t-}) \\
&\quad +\bar P_L(t,\alpha_{t-}){\mathcal E}\widehat b(t,\alpha_{t-})\Big\}dt  
+\tau_W(t)dW(t)+\tau_M(t)\bullet dM(t),\\
&\tau(T)=\textbf{0},
\end{aligned}\right.
\end{equation}
where $\mathbb{B}^{j}_{XY}(t,i)=X+Y\mathcal B_j$, $j=1,2$. 

In fact, $\mathcal Y(\cdot)$ and $\mathcal Z(\cdot)$ given by \eqref{XY} can be expressed as 
\begin{equation}\left\{\begin{aligned}\label{YZ-X}
&\mathcal Y(t)=P_L(t,\alpha_{t-})\mathcal X(t)+\bar P_L (t,\alpha_{t-})\widehat{\mathcal X}(t)+\tau(t), \\
&\mathcal Z(t)=P_L(t,\alpha_{t-})\big[\mathcal C(t,\alpha_{t-})\mathcal X(t)
 +\mathcal E(t,\alpha_{t-})\sigma(t)\big]+\tau_W(t),
\end{aligned}\right.\end{equation}
where $P_L(\cdot,i), \bar P_L(\cdot,i), \big(\tau(\cdot),\tau_W(\cdot)\big)$ are given by \eqref{PL-1}-\eqref{PL-3}, $i\in\mathcal M$. With \eqref{YZ-X}, the strategy $u^*_L(\cdot)$ in \eqref{opuL} can be expressed as an affine function of the process $\mathcal X(\cdot)$ and its conditional expectation $\widehat{\mathcal X}(\cdot)$. In this sense, we can obtain the feedback form of $u^*_L(\cdot)$.

\begin{Theorem}\label{th3.5}
With Assumptions (L1)-(L2), \eqref{PL-1}-\eqref{PL-3} have unique solutions. 
Moreover, the optimal strategy $u^*_L(\cdot)$ is uniquely solvable with closed-loop representation 
\begin{equation}\begin{aligned}\label{closed-ld}
u_L^*(t)=&\ -R_L^{-1}(t,\alpha_{t-})\mathcal{B}^{\top}(t,\alpha_{t-})
\Big[P_L(t,\alpha_{t-})\mathcal X^*(t)+\bar P_L (t,\alpha_{t-})\widehat{\mathcal X}^*(t)+\tau(t)\Big],
\end{aligned}\end{equation}
where $\mathcal X^*(\cdot)$ satisfies
\begin{equation}\label{closed-xx}
\left\{\begin{aligned}
& d\mathcal X^*(t)=\Big[\big(\mathcal A(t,\alpha_{t-})+\mathcal B_1(t,\alpha_{t-})P_L(t,\alpha_{t-})\big)
\mathcal X^*(t)+\big(\bar{\mathcal A}(t,\alpha_{t-})+\mathcal B_1(t,\alpha_{t-}) \\
&\ \times\bar P_L(t,\alpha_{t-})\big)\widehat{\mathcal X}^*(t)
+\mathcal B_1(t,\alpha_{t-})\tau(t)+\mathcal{E} b(t)\Big]dt 
 +\Big[\mathcal C(t,\alpha_{t-})\mathcal X^*(t)+\mathcal{E}\sigma(t)\Big]dW(t), \\
& \mathcal X^*(0)={\mathcal X}_0,
\end{aligned}\right.
\end{equation}
and $P_L(\cdot,i), \bar P_L(\cdot,i), \tau(\cdot)$ are given by \eqref{PL-1}-\eqref{PL-3}.
\end{Theorem}

\emph{Proof:} See Appendix F. \hfill$\square$

\begin{Remark}\label{yuqi}
Since \eqref{eq3} is a BSDE, the value $\left(\phi(t),\phi_W(t),\phi_M(t)\right)$ at time $t$ depends on $\{u_L(s)|s\in[0,T]\}$. Thus, $\mathbf{u^*_{F_{1,2}}}(\cdot)$ given by \eqref{closed-ll} depends on $\{u_L(s)|s\in[0,T]\}$, which means that $\mathbf{u^*_{F_{1,2}}}(\cdot)$ is anticipating, i.e., $\mathbf{u^*_{F_{1,2}}}(\cdot)$ could depend on the future values $\{u_L(s)|s\in[0,T]\}$. See Yong \cite{Yong2002} for more information about ``anticipating''. For the leader, the optimal strategy $u_L^*(\cdot)$ given by \eqref{closed-ld} is nonanticipating. Once $u_L^*(\cdot)$ is determined in \eqref{closed-ld}, substituting 
\eqref{closed-ld} into \eqref{closed-x} yields that the resulting $\mathbf{u^*_{F_{1,2}}}(\cdot)$ in \eqref{closed-ll} is nonanticipating. 
\end{Remark}

{\color{black}{
\section{Green finance problem}\label{S5}
In this section, we employ the theoretical results obtained in Section \ref{S4} to solve the green finance problem given in Subsection \ref{moti}. For notational convenience, $x(\cdot)$ in system \eqref{m-1} is also denoted 
by $x_1(\cdot)$. Indeed, the green finance problem can be regarded as a special case of Problem (CMF-SDG) in Subsection \ref{problem-f}, where
\begin{equation}
\left\{\begin{aligned}\label{coef}
& A(\cdot, i)=a_0(\cdot, i),\  \bar A(\cdot, i)=\bar a_0(\cdot, i),\
B_L(\cdot, i)=b_l(\cdot, i), \ \mathbf{B^{\top}_F}(\cdot, i)=\mathbf{b_f^{\top}}(\cdot, i), \\
&\mathbf{b_f}(\cdot, i)=\left(b_{f,1}(\cdot, i), b_{f,2}(\cdot, i)\right), \ b(\cdot, i)=0,\
C(\cdot, i)=c_0(\cdot, i),  \\
& Q_F(\cdot, i)=q_f(\cdot, i),\ \bar C(\cdot, i)=D_L(\cdot, i)=\sigma(\cdot, i)=\bar Q_F(\cdot, i)=0, \
 \mathbf{D^{\top}_F}(\cdot, i)=\textbf{0},\\
& \mathbf{R_F}(\cdot, i)=\mathbf{r_f}(\cdot, i),\
\mathbf{r_f}(\cdot, i)=diag\left(r_{f,1}(\cdot, i), r_{f,2}(\cdot, i)\right), \
 G_F(i)=\bar G_F(i)=0,\\
&  Q_L(\cdot, i)=0,\ \bar Q_l(\cdot, i)=\bar q_l(\cdot, i),\ R_L(\cdot, i)=r_l(\cdot, i),\
 G_L(i)=\bar Q_L(i)=0,
\end{aligned}\right.
\end{equation}
$m_0=1$ in $\mathcal U_L$, and $m_j=1$ in $\mathcal U_{F,j}$ for $j=1,2$.
Applying Theorems \ref{OL} and \ref{th-F}, we obtain that the open-loop strategies of the two firms and the
government are represented by
\begin{equation}\label{uf}
\mathbf{u^*_{f_{1,2}}}(t)
=-\mathbf{r_f^{-1}}(t,\alpha_{t-})\mathbf{b^{\top}_f}(t,\alpha_{t-})y_1(t),
\end{equation}
and
\begin{equation}\label{ul}
u_l^*(t)=-r_l^{-1}(t,\alpha_{t-})b_l(t,\alpha_{t-})y_2(t), 
\end{equation}
respectively. In \eqref{uf}, $\mathbf{u^*_{f_{1,2}}}(\cdot)=\big(u^*_{f,1}(\cdot), u^*_{f,2}(\cdot)\big)^{\top}$, $y_1(\cdot)$ and $y_2(\cdot)$ are given by \eqref{p} and \eqref{2fbsde}, respectively, and they satisfy the coefficients in \eqref{coef}.

Next, we give the feedback forms for the open-loop strategies. By virtue of Theorem \ref{riccati}, we get the closed-loop representation of the open-loop strategy $\mathbf{u^*_{f_{1,2}}}(\cdot)$ for the two firms as follows
\begin{equation}\label{uf.}
\mathbf{u^*_{f_{1,2}}}(t)
=-\mathbf{r_f^{-1}}(t,\alpha_{t-})\mathbf{b^{\top}_f}(t,\alpha_{t-})
\left[P_F(t,\alpha_{t-})x_1^*(t)+\bar P_F(t,\alpha_{t-})\widehat x_1^*(t)+\phi(t)\right],
\end{equation} 
where $x_1^*(\cdot)$ is the optimal state satisfying 
\begin{equation}
\left\{\begin{aligned}
dx_1^*(t)=&\ \big[\mathbb A_1(t,\alpha_{t-})x_1^*(t)+\mathbb A_2(t,\alpha_{t-})\widehat x_1^*(t)
+b_l(t,\alpha_{t-})u_l(t)+\mathbb A_3(t,\alpha_{t-})\big]dt \\
&\ +c_0(t,\alpha_{t-})x_1^*(t)dW(t), \\
x_1^*(0)=&\ 0, \nonumber
\end{aligned}\right.
\end{equation}
$\mathbb A_1=a_0-\mathbf{b_fr_f^{-1}b^{\top}_f}P_F$, 
$\mathbb A_2=\bar a_0-\mathbf{b_fr_f^{-1}b^{\top}_f}\bar P_F$,
$\mathbb A_3=-\mathbf{b_fr_f^{-1}b^{\top}_f}\phi$, 
$P_F(\cdot,i)$, $\bar P_F(\cdot,i)$, and $\phi(\cdot)$ are given by \eqref{eq1}-\eqref{eq3}, and all the coefficients satisfy \eqref{coef}. 
Besides, it follows from Theorem \ref{th3.5} that the closed-loop representation of the strategy for the government is
\begin{equation}\label{ul.}
u_l^*(t)=-r_l^{-1}(t,\alpha_{t-})\mathbbm{b^{\top}}(t,\alpha_{t-})
\left[P_L(t,\alpha_{t-})\mathcal X^*(t)+\bar P_L (t,\alpha_{t-})\widehat{\mathcal X}^*(t)+\tau(t)\right],
\end{equation}
%
where $\mathcal X^*(\cdot)$ satisfies
\begin{equation}
\left\{\begin{aligned}
d\mathcal X^*(t)=&\ \Big\{\left[\mathbbm{a}_0(t,\alpha_{t-})+\mathbbm{b}_1(t,\alpha_{t-})
P_L(t,\alpha_{t-})\right]\mathcal X^*(t) \\
&\ +\left[\bar{\mathbbm{a}}_0(t,\alpha_{t-})+\mathbbm{b}_1(t,\alpha_{t-})
\bar P_L(t,\alpha_{t-})\right]\widehat {\mathcal X}^*(t)\\
&\ +\mathbbm{b}_1(t,\alpha_{t-})\tau(t)\Big\}dt
+\mathbbm{c}_0(t,\alpha_{t-})\mathcal X^*(t)dW(t), \\
\mathcal X^*(0)=&\ \mathcal X_0,\nonumber
\end{aligned}\right.
\end{equation}
$P_L(\cdot,i), \bar P_L(\cdot,i), \tau(\cdot)$ are given by \eqref{PL-1}-\eqref{PL-3} with the coefficients in \eqref{coef}, respectively, $i\in\mathcal M$, and
\begin{equation}
\begin{gathered}
\mathbbm{v}=\begin{pmatrix}
 v & 0\\
0 & v\\
\end{pmatrix} \ \text{with} \ \mathbbm{v}=\mathbbm{a}_0, \bar{\mathbbm{a}}_0, \mathbbm{c}_0,\quad
\mathbbm{b}_1=\begin{pmatrix}
-r_l^{-1}b_l^2  & -\mathbf{b_f}\mathbf{r_f^{-1}}\mathbf{b_f^{\top}} \\
\mathbf{b_f}\mathbf{r_f^{-1}}\mathbf{b_f^{\top}}  & 0  \\
\end{pmatrix}.
\end{gathered}\nonumber
\end{equation}
According to Remark \ref{yuqi}, $\mathbf{u^*_{f_{1,2}}}(\cdot)$ in \eqref{uf.} can also be represented in a nonanticipating way.


It is remarkable that the open-loop strategies \eqref{uf}-\eqref{ul} are expressed in terms of the solutions to the coupled CMF-FBSDEs \eqref{p} and \eqref{2fbsde}, respectively. However, solving such coupled CMF-FBSDEs directly is technically challenging. To decouple CMF-FBSDEs \eqref{p} and \eqref{2fbsde}, we introduce equations \eqref{eq1}-\eqref{eq3} and \eqref{PL-1}-\eqref{PL-3}. The numerical solutions of Riccati equations and BSDEs \eqref{eq1}-\eqref{eq3} and \eqref{PL-1}-\eqref{PL-3} can be obtained by Euler's method and Monte Carlo approach. With relations \eqref{pq-x} and \eqref{YZ-X}, the numerical solutions for \eqref{p} and \eqref{2fbsde} can also be obtained. Thus, based on \eqref{pq-x} and \eqref{YZ-X}, the closed-loop representations \eqref{uf.}-\eqref{ul.} are consistent with the open-loop strategies. Moreover, in this sense, our work also provides an approach for solving the coupled CMF-FBSDEs.

In the following, we solve the green finance problem numerically with certain coefficients as follows. Let the generator of the Markov chain 
$\mathcal M_0=\{1,2\}$ be
$$
Q=\left(\begin{array}{ll}
-0.1 & 0.1\\
0.2 & -0.2
\end{array}\right), $$
and the initial state $x_0=2$. $``1"$ and $``2"$ in $\mathcal M_0=\{1,2\}$ represent the bull and bear markets, respectively. For illustration purpose, let $T=1$, and the coefficients of the dynamic equation \eqref{m-1} are given as follows:
\begin{equation} 
\begin{aligned}
&a_0(\cdot,1)=0.5, \ \bar a_0(\cdot,1)=0.4, \ b_{f,1}(\cdot,1)=-0.4, \ b_{f,2}(\cdot,1)=0.2,\  b_l(\cdot,1)=-0.5, \\
&a_0(\cdot,2)=0.1, \ \bar a_0(\cdot,2)=0.2, \  b_{f,1}(\cdot,2)=-0.3, \ b_{f,2}(\cdot,2)=0.1,\ b_l(\cdot,2)=-0.3, \\
& c_0(\cdot,1)=0.5, \ c_0(\cdot,2)= 0.2. \nonumber
\end{aligned}
\end{equation}
Consider the cost functionals defined by \eqref{cost-fm}-\eqref{cost-lm} with
\begin{equation} 
\begin{aligned}
& q_f(\cdot,1)=0.1,\ r_{f,1}(\cdot,1)=3.4, \ r_{f,2}(\cdot,1)=-3.4, \ \bar q_l(1)=1, \  r_l(\cdot,1)=5, \\
& q_f(\cdot,2)=0.2,\ r_{f,1}(\cdot,2)=3.6, \ r_{f,2}(\cdot,2)=-3.8,  \ \bar q_l(2)=0.9, \ r_l(\cdot,2)=1. \nonumber
\end{aligned}
\end{equation}
For the sake of clarity, in what follows, we will use $h(\cdot,1)$ and $h(\cdot,2)$ to represent the values of $h(\cdot)$ at state 1 and state 2, respectively, where $h=u_{f,1}^*, u_{f,2}^*, u_l^*$.

Employing Euler's method, we obtain a sample path of Markov chain $\alpha_{\cdot}$ in Fig. \ref{alpha}, which shows that Markov chain $\alpha_{\cdot}$ switches between state 1 and state 2. By Euler-Maruyama and Monte Carlo approaches, we plot the curves of the optimal strategies $u^*_{f,1}(\cdot), u^*_{f,2}(\cdot)$ and $u_l^*(\cdot)$ given by \eqref{uf.} and \eqref{ul.}, shown in Fig.~\ref{uf12} and Fig.~\ref{ul-x}(a). In Fig.~\ref{uf12}(a), the red dotted line is the graph of $u^*_{f,1}(\cdot,1)$, the blue dotted line is the graph of $u^*_{f,1}(\cdot,2)$, and the solid red and blue line is the graph of $u^*_{f,1}(\cdot)$, whose values switch between the red dotted line $u^*_{f,1}(\cdot,1)$ and the blue dotted line $u^*_{f,1}(\cdot,2)$. Similarly, in Fig.~\ref{uf12}(b), 
the values of $u^*_{f,2}(\cdot)$ also switch between the two dotted lines. From Fig.~\ref{uf12}, one sees that the two firms demonstrate greater investment intensity during state 1 (bull market) relative to state 2 (bear market). Besides, irrespective of the market states, the advantaged Firm 2 consistently exhibits a higher level of investment intensity than the lagging Firm 1. In contrast to the bull market, the competition between the two firms becomes more intense in the bear market. As a result, the government increases its order-maintenance strategy in the bear market to preserve fair competition and social stability, shown in Fig.~\ref{ul-x}(a). In addition, the curve of subsidy share difference $x(\cdot)$ is plotted in Fig.~\ref{ul-x}(b). It can be seen from Fig.~\ref{ul-x}(b) that, the subsidy share disparity expands and exhibits higher volatility during state 1 (bull market), yet contracts and becomes less volatile during state 2 (bear market).

Finally, we conduct robustness analyses with respect to $b_l(\cdot,i), b_{f,1}(\cdot,i), b_{f,2}(\cdot,i)$ and $r_l(\cdot,i)$ given by \eqref{m-1} and \eqref{cost-lm}, $i\in\mathcal M_0$. Here $b_l(\cdot,i), b_{f,1}(\cdot,i), b_{f,2}(\cdot,i)$ $(i\in\mathcal M_0)$ are the coefficients of the government intervention strategy, the green investment strategies for the lagging Firm 1 and the advantaged Firm 2, respectively. And $r_l$ is the weight of the term $u_l^2$ in \eqref{cost-lm} and captures the social cost incurred by frequent policy changes.

\begin{itemize}
\item
Take $b_l(t,1)=-0.5, b_l(t,2)=-0.3$ and $b_l(t,1)=-0.1, b_l(t,2)=-0.05$ to illustrate the influence of different $b_l(\cdot,i)$ $(i\in\mathcal M_0)$ values on $x(\cdot)$, as shown in Fig.~\ref{x-r12}. A decrease in $|b_l(\cdot,i)|$ $(i\in\mathcal M_0)$ means that the effectiveness of government intervention weakens, causing the subsidy gap to shift from a ``controllable convergent state'' (see Fig.~\ref{x-r12}(a)) to an ``uncontrollable divergent state'' (see Fig.~\ref{x-r12}(b)).

\item
Take $b_{f,1}(t,1)=-0.4, b_{f,1}(t,2)=-0.3$ and $b_{f,1}(t,1)=-0.1, b_{f,1}(t,2)=-0.01$ to demonstrate the impact of varying $b_{f,1}(\cdot,i)$ $(i\in\mathcal M_0)$ on $x(\cdot)$, shown in Fig.~\ref{bf1-x}. A reduction in $|b_{f,1}(\cdot,i)|$ $(i\in\mathcal M_0)$ diminishes the negative marginal effect of Firm 1's investment on $x(\cdot)$, thereby decelerating the convergence of $x(\cdot)$ and rendering the subsidy gap less amenable to contraction (see Fig.~\ref{bf1-x}(b)).

\item
Take $b_{f,2}(t,1)=0.2, b_{f,2}(t,2)=0.1$ and $b_{f,2}(t,1)=0.02, b_{f,2}(t,2)=0.01$ to show the influence of varying $b_{f,2}(\cdot,i)$ $(i\in\mathcal M_0)$ on $x(\cdot)$, shown in Fig.~\ref{bf2-x}. As $b_{f,2}(\cdot,i)$ $(i\in\mathcal M_0)$ decreases, the scope for the advantaged Firm 2 to further widen the subsidy gap through its investment becomes constrained, which in turn leads to a deceleration in the growth rate of $x(\cdot)$ (see 
Fig.~\ref{bf2-x}(b)). Unlike Fig.~\ref{bf2-x}(a), Fig.~\ref{bf2-x}(b) exhibits a more contained subsidy gap, 
with $x(\cdot)\leq 1$ for $t\geq 0.6$.

\item
Take $r_l(t,1)=5, r_l(t,2)=1$ and $r_l(t,1)=0.5, r_l(t,2)=0.1$ to illustrate the influence of different $r_l(\cdot,i)$ $(i\in\mathcal M_0)$ on $u_l^*(\cdot)$ and $x(\cdot)$, as shown in Figs.~\ref{ul-12} and \ref{x-12}. A smaller $r_l(\cdot,i)$ leads to lower social costs and prompts the government to increase its order-maintenance strategy, which is consistent with Fig.~\ref{ul-12}(b). As a result, a higher level of the government's strategy also narrows the subsidy share difference $x(\cdot)$, as illustrated in Fig.~\ref{x-12}(b).

\end{itemize}
}}

\begin{figure}[htbp]
\centering
\vspace*{-0.1cm}
\includegraphics[height=4cm, width=8cm]{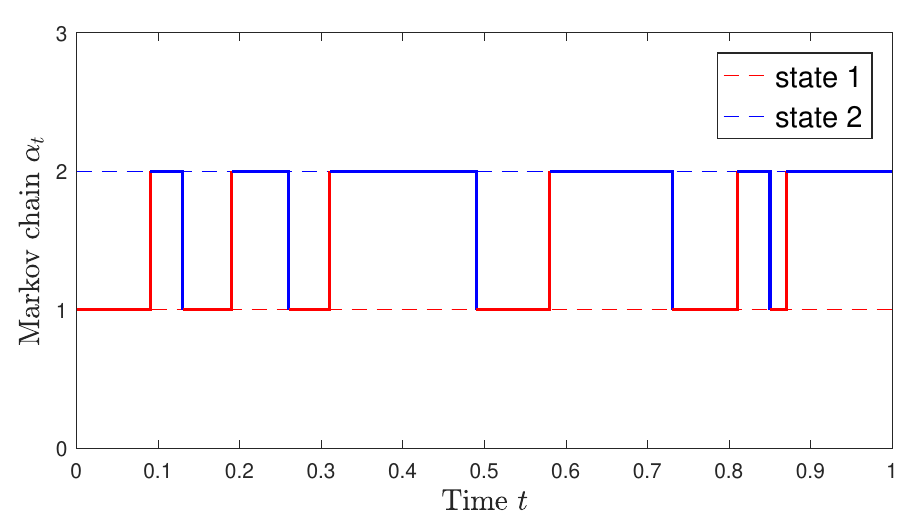}
\vspace*{-0.2cm}
\caption{{\color{black}{Numerical solution of $\alpha_t$.}} }\label{alpha}
\end{figure}

\vspace*{-0.8cm}

\begin{figure}[htbp]
\vspace*{-0.5cm}
\subfloat[A trajectory of $u^*_{f,1}(t)$.]
{
\includegraphics[height=4cm, width=6cm]{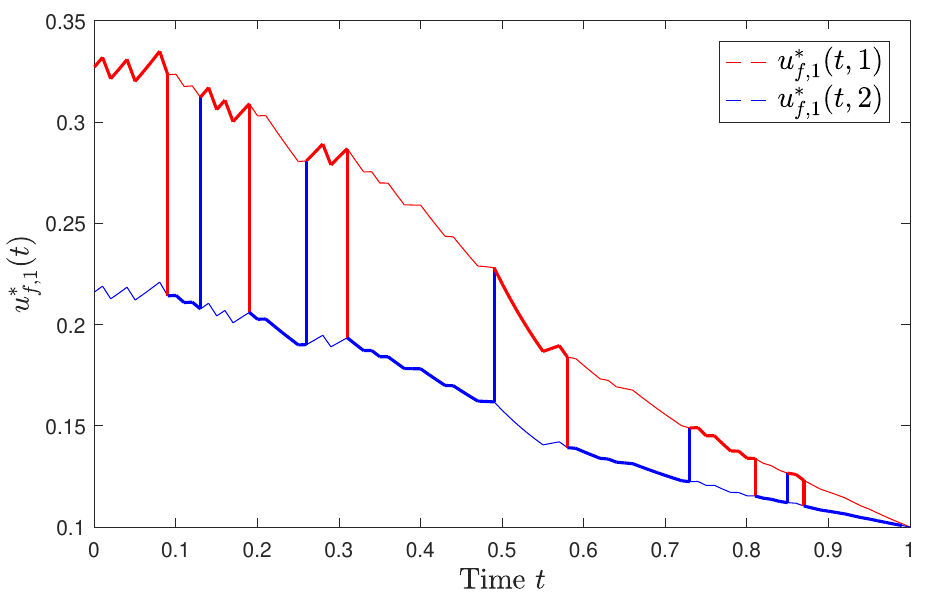}   
}
\subfloat[A trajectory of $u^*_{f,2}(t)$.]
{
\includegraphics[height=4cm, width=6cm]{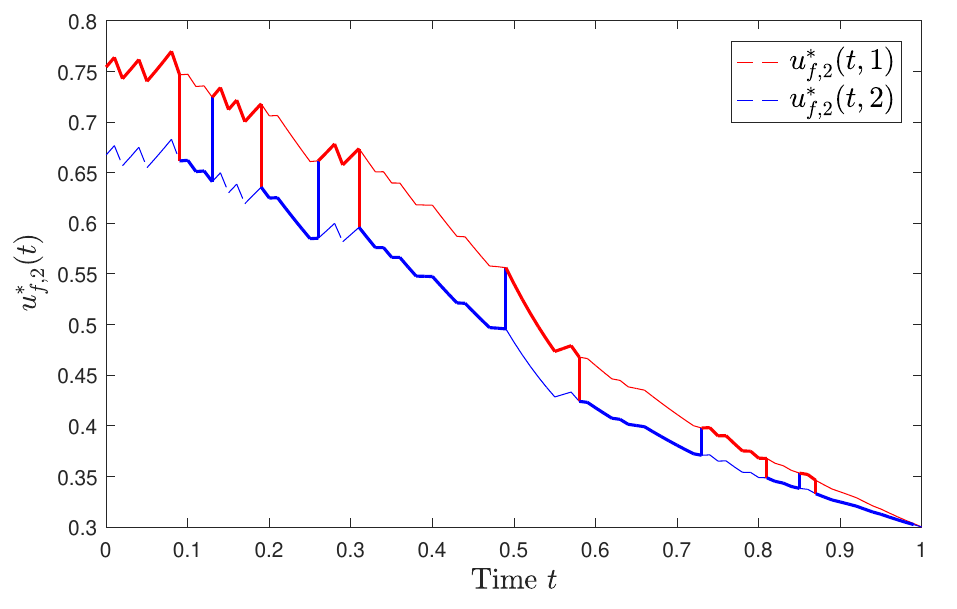}
}
\caption{{\color{black}{Numerical solutions of $u^*_{f,1}(t)$ and $u^*_{f,2}(t)$.}}}\label{uf12}
\end{figure}

\vspace*{-0.8cm}

\begin{figure}[htbp]
\vspace*{-0.5cm}
\subfloat[A trajectory of $u^*_l(t)$.]
{
\includegraphics[height=4cm, width=6cm]{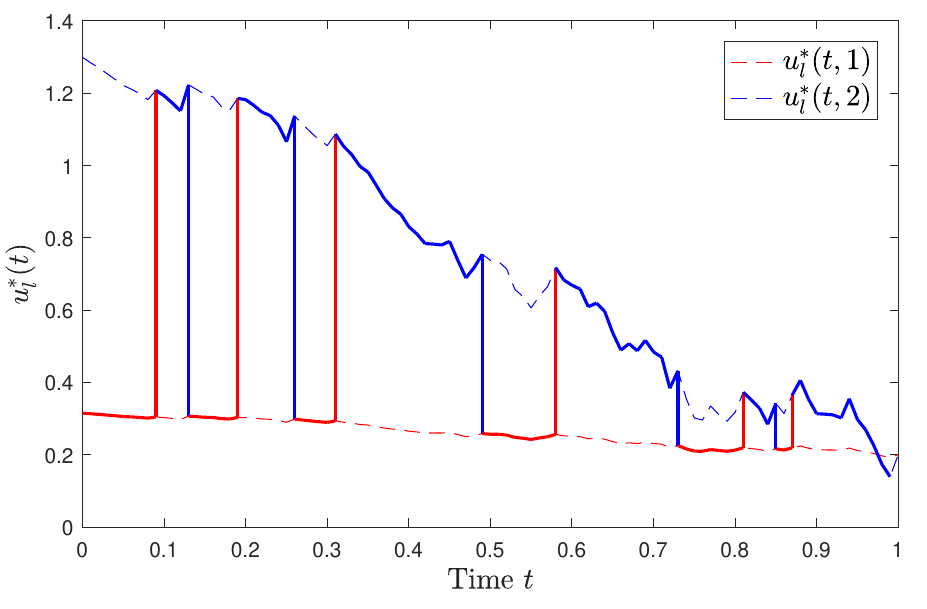}   
}
\subfloat[A trajectory of $x(t)$.]
{
\includegraphics[height=4cm, width=6cm]{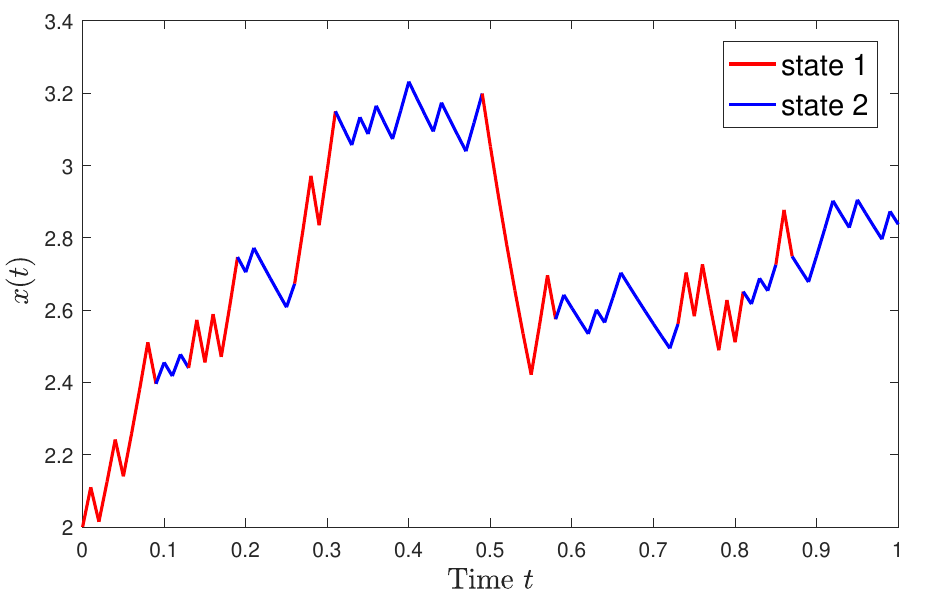}
}
\caption{{\color{black}{Numerical solutions of $u^*_l(t)$ and $x(t)$.}} }\label{ul-x}
\end{figure}

\vspace*{-0.8cm}

\begin{figure}[!h]
\vspace*{-0.5cm}
\subfloat[\centering A trajectory of $x(t)$ with\\ $b_l(t,1)=-0.5$ and $b_l(t,2)=-0.3$.]
{
\includegraphics[height=4cm, width=6cm]{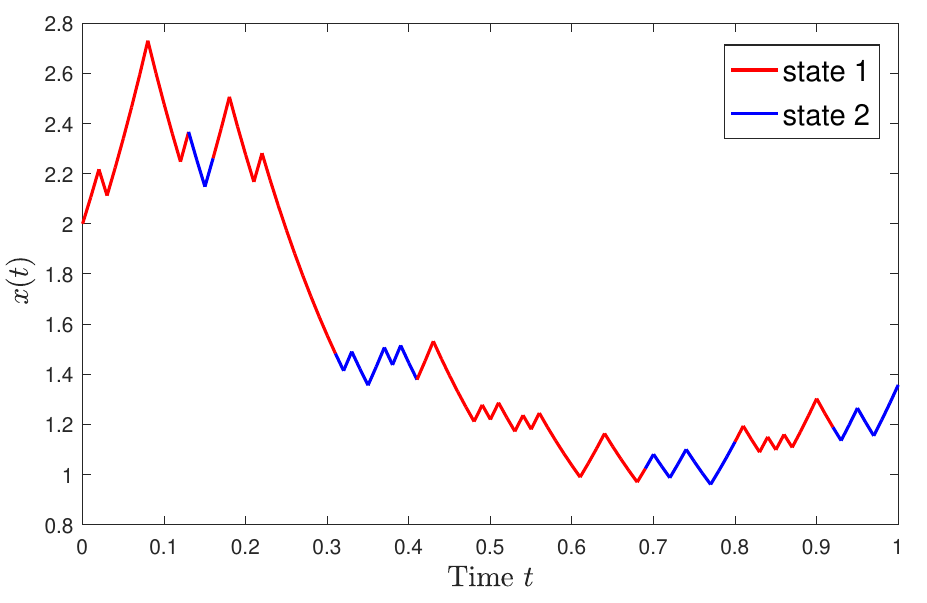}   
}
\subfloat[\centering A trajectory of $x(t)$ with\\ $b_l(t,1)=-0.1$ and $b_l(t,2)=-0.05$.]
{
\includegraphics[height=4cm, width=6cm]{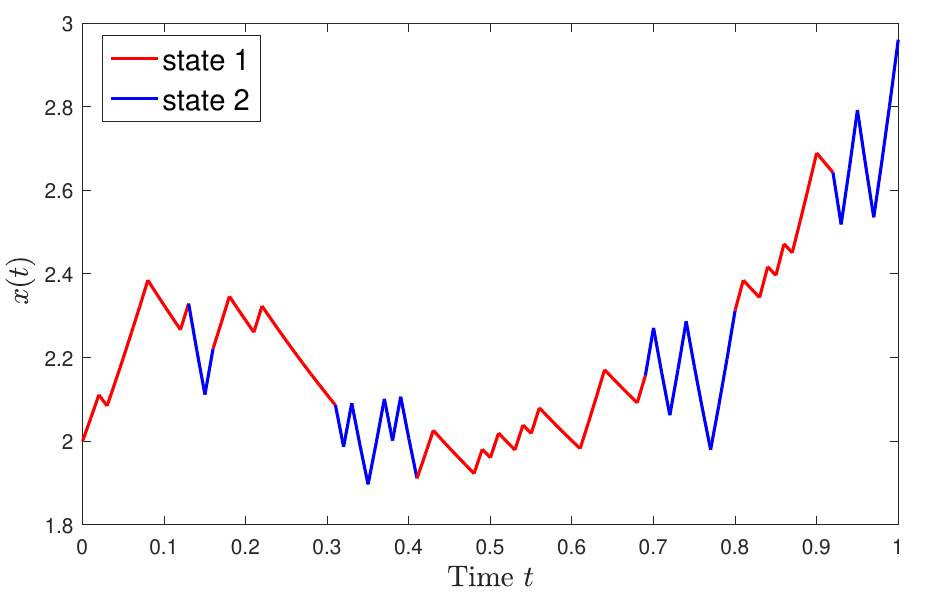}
}
\caption{{\color{black}{Numerical solutions of $x(t)$ with different $b_l(t,i)$, $i\in\mathcal M_0$.}}}\label{x-r12}
\end{figure}

\vspace*{-0.8cm}

\begin{figure}[!h]
\vspace*{-0.5cm}
\subfloat[\centering A trajectory of $x(t)$ with\\ $b_{f,1}(\cdot,1)=-0.4$ and $b_{f,1}(\cdot,2)=-0.3$.]
{
\includegraphics[height=4cm, width=6cm]{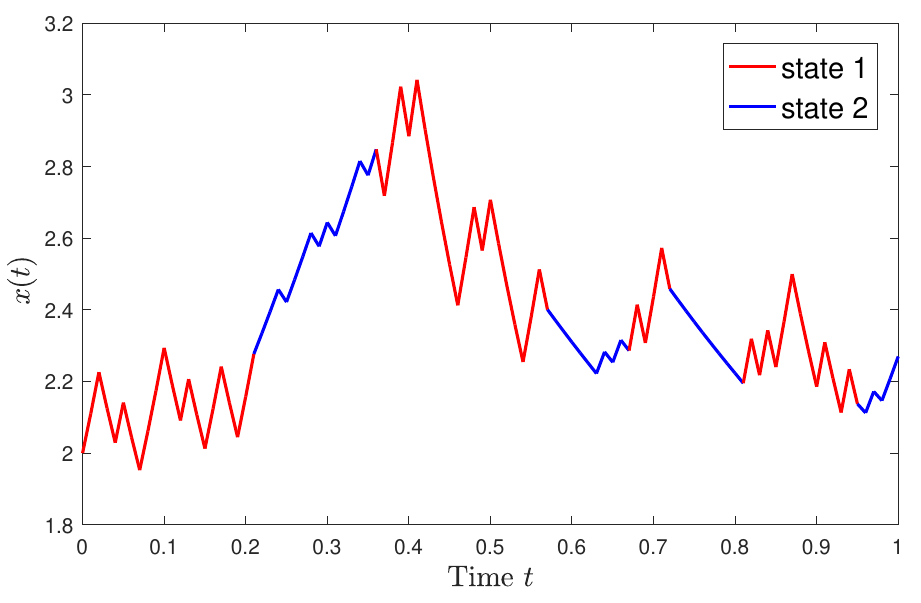}   
}
\subfloat[\centering A trajectory of $x(t)$ with\\ $b_{f,1}(\cdot,1)=-0.1$ and $b_{f,1}(\cdot,2)=-0.01$.]
{
\includegraphics[height=4cm, width=6cm]{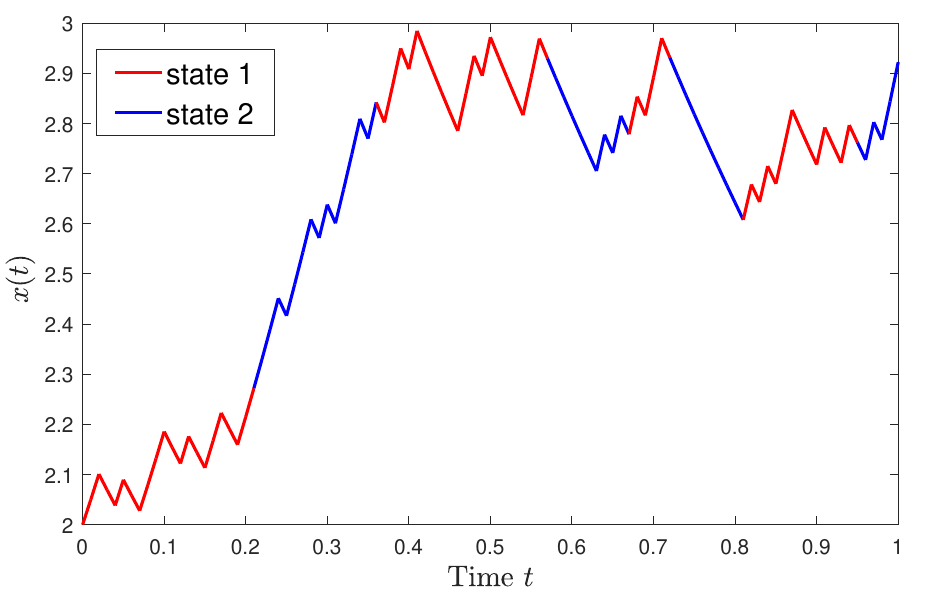}
}
\caption{{\color{black}{Numerical solutions of $x(t)$ with different $b_{f,1}(t,i)$, $i\in\mathcal M_0$.}}}\label{bf1-x}
\end{figure}

\vspace*{-0.5cm}

\begin{figure}[!h]
\vspace*{-0.5cm}
\subfloat[\centering A trajectory of $x(t)$ with\\ $b_{f,2}(\cdot,1)=0.2$ and $b_{f,2}(\cdot,2)=0.1$.]
{
\includegraphics[height=4cm, width=6cm]{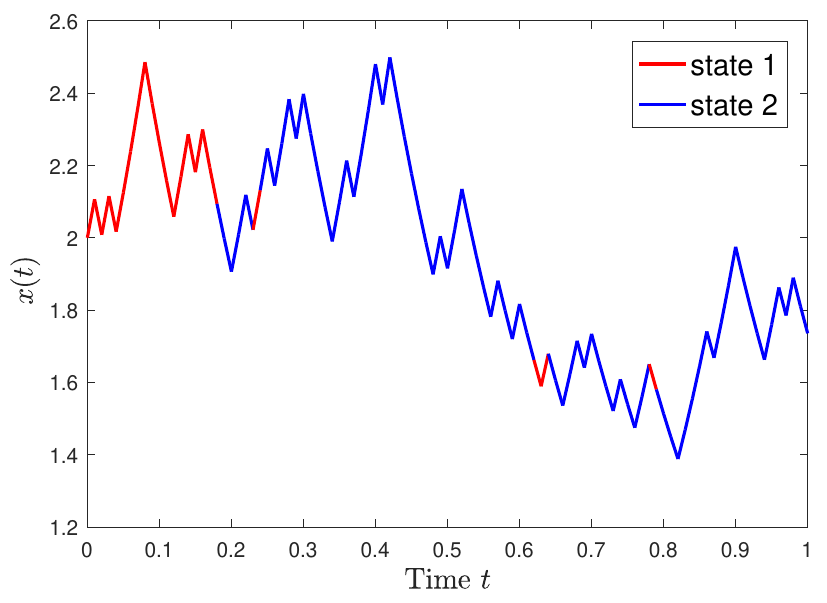}   
}
\subfloat[\centering A trajectory of $x(t)$ with\\ $b_{f,2}(\cdot,1)=0.02$ and $b_{f,2}(\cdot,2)=0.01$.]
{
\includegraphics[height=4cm, width=6cm]{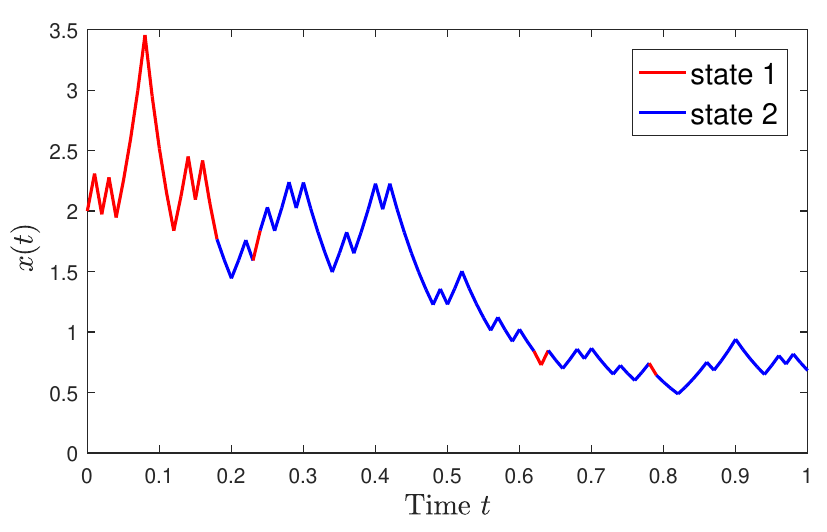}
}
\caption{{\color{black}{Numerical solutions of $x(t)$ with different $b_{f,2}(t,i)$, $i\in\mathcal M_0$.}}}\label{bf2-x}
\end{figure}

\vspace*{-0.8cm}

\begin{figure}[htbp]
\vspace*{-0.5cm}
\subfloat[\centering A trajectory of $u^*_{l}(t)$ with\\ $r_l(t,1)=5$ and $r_l(t,2)=1$.]
{
\includegraphics[height=4cm, width=6cm]{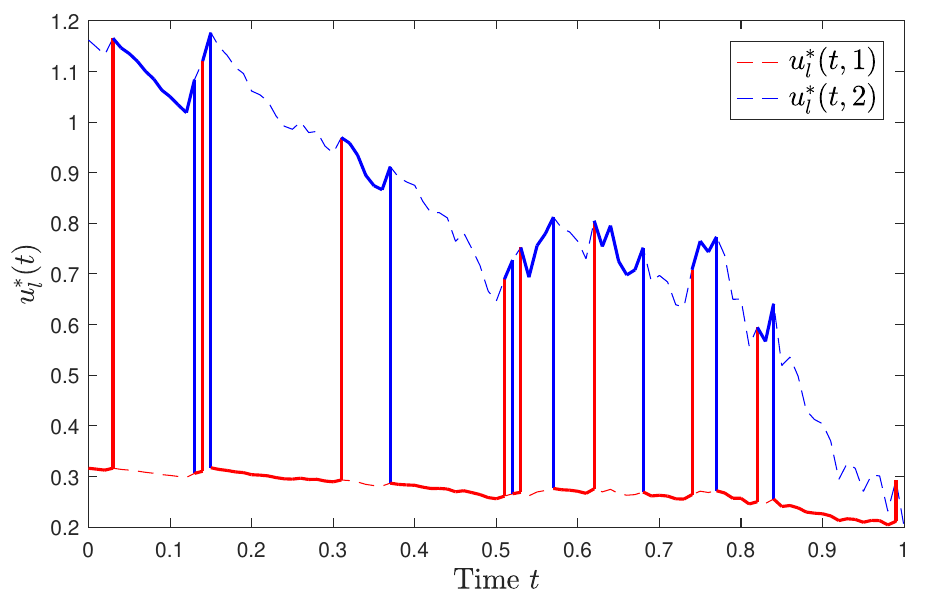}   
}
\subfloat[\centering A trajectory of $u^*_{l}(t)$ with\\ $r_l(t,1)=0.5$ and $r_l(t,2)=0.1$.]
{
\includegraphics[height=4cm, width=6cm]{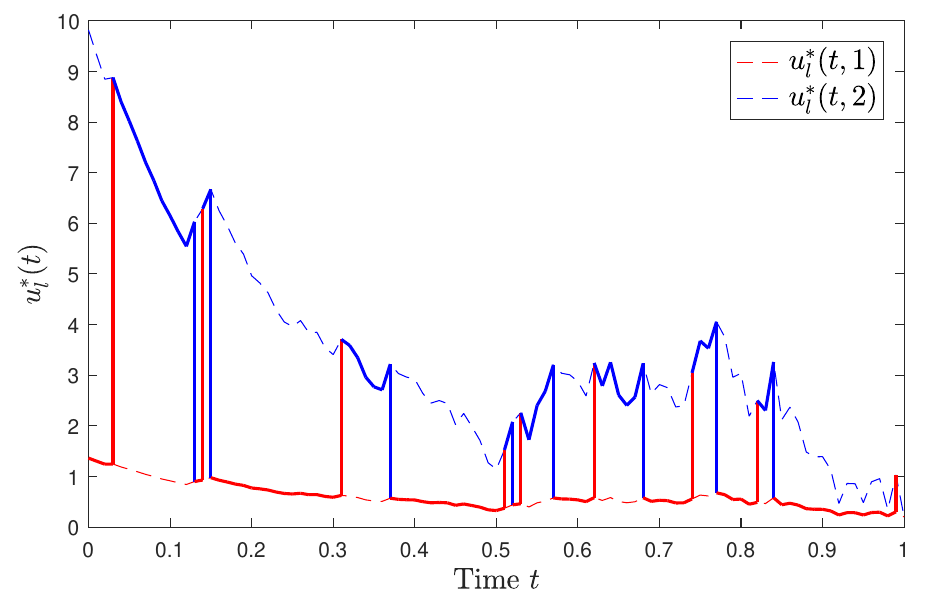}
}
\caption{{\color{black}{Numerical solutions of $u^*_{l}(t)$ with different $r_l(t,i)$, $i\in\mathcal M_0$.}} }\label{ul-12}
\end{figure}

\vspace*{-0.8cm}

\begin{figure}[!h]
\vspace*{-0.5cm}
\subfloat[\centering A trajectory of $x(t)$ with\\ $r_l(t,1)=5$ and $r_l(t,2)=1$.]
{
\includegraphics[height=4cm, width=6cm]{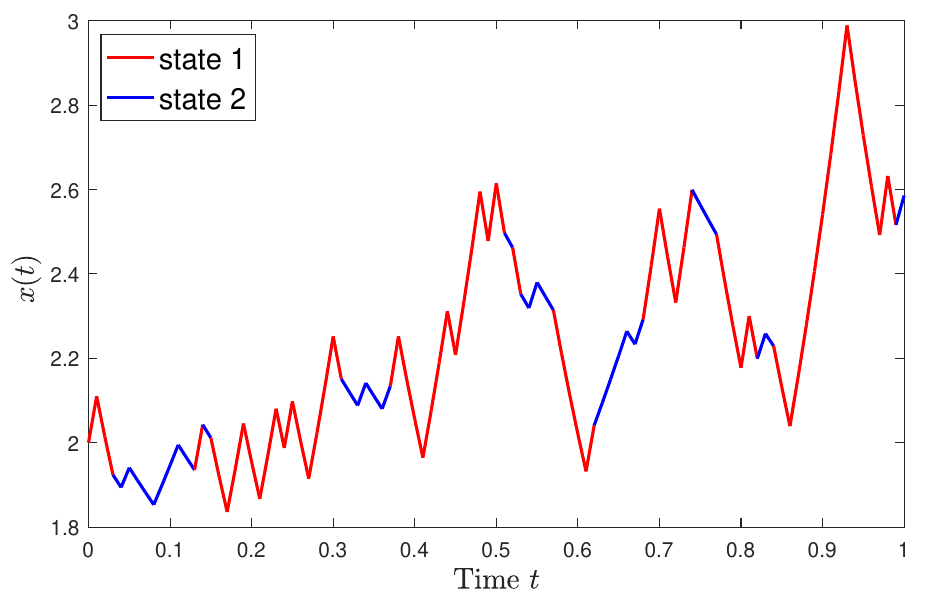}   
}
\subfloat[\centering A trajectory of $x(t)$ with\\ $r_l(t,1)=0.5$ and $r_l(t,2)=0.1$.]
{
\includegraphics[height=4cm, width=6cm]{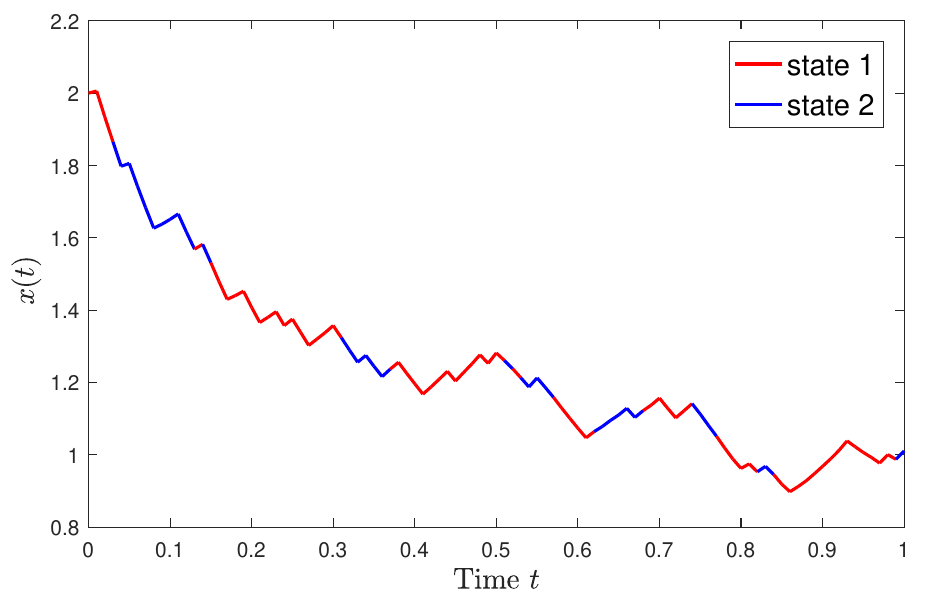}
}
\caption{{\color{black}{Numerical solutions of $x(t)$ with different $r_l(t,i)$, $i\in\mathcal M_0$.}}}\label{x-12}
\end{figure}

\vspace*{5.5cm}
\section{Conclusion and outlook}\label{S6}
This paper investigates an LQ mixed Stackelberg differential game with conditional mean-field and regime switching, where one leader and two followers involve. By domination-monotonicity conditions with conditional expectation, continuation method and induction method, we obtain the unique solvability of the CMF-FBSDEs. Further, applying stochastic maximum principle and decoupling approach, we derive the optimal strategies for the two followers and the leader.

Except for noncooperative game studied in this work, the study of cooperative game also has important research value. We will consider the cooperative game with conditional mean-field and regime switching in our future work.

\section*{Appendix A}\qquad \qquad\qquad

To prove Theorem \ref{eufbsde}, we first present a proposition.
\begin{Proposition}\label{prop1}
If $b_2(\cdot,\cdot,\cdot,\cdot)=\sigma_2(\cdot,\cdot,\cdot,\cdot)=g_2(\cdot,\cdot,\cdot,\cdot)
=\Psi_2(\cdot)=\Phi_2(\cdot,\cdot,\cdot)=0$, $x_2=y_2=z_2=\tilde z_2=0$,
Theorem \ref{eufbsde} holds true.
\end{Proposition}

\emph{Proof of Proposition \ref{prop1}:}  By virtue of Theorem 3.1 of Peng and Wu \cite{PW15} and Lemma 3.1 of Yong \cite{Yong21}, the result can be derived. We omit the detailed proof. \hfill$\square$

For any $i\in\mathcal M$, let
\begin{equation*} 
\left\{\begin{aligned}
& \mathcal L_1(\cdot,i)=\big(\sqrt{\nu_1}\widetilde A_1(\cdot,i), \sqrt{\mu_1}\widetilde B_1(\cdot,i), \sqrt{\mu_1}\widetilde C_1(\cdot,i)\big)^{\top} \\
&\qquad\qquad\times\big(\sqrt{\nu_1}\widetilde A_1(\cdot,i), \sqrt{\mu_1}\widetilde B_1(\cdot,i), \sqrt{\mu_1}\widetilde C_1(\cdot,i)\big),\\
& \mathcal L_2(\cdot,i)=\mu_2\Lambda \big({\widetilde L_2(\cdot,i)}\big)^{\top}\widetilde L_2(\cdot,i),\quad 
\mathcal M_1=\mu_1{\widetilde M_1}^{\top}\widetilde M_1, \\
& \mathcal M_2=\mu_2\Lambda{\widetilde M_2}^{\top}\widetilde M_2, \quad  
 \mathcal G_1(i)=\nu_1{\widetilde G_1^{\top}(i)}\widetilde G_1(i),  \quad 
\mathcal G_2(i)=\mu_2\Lambda{\widetilde G_2}^{\top}(i)\widetilde G_2(i), 
\end{aligned}\right.
\end{equation*}
\[\mathbb L_1(\cdot,i)=
\left(\begin{array}{cc}
-\mathcal L_1 & \mathbf{0} \\
\mathbf{0} & \mathcal L_1 \\
\end{array}\right),\
\mathbb M_1=
\left(\begin{array}{cc}
-\mathcal M_1 & \mathbf{0} \\
\mathbf{0} & \mathcal M_1 \\
\end{array}\right),\
\mathbb G_1(i)=
\left(\begin{array}{cc}
\mathcal G_1 & \mathbf{0} \\
\mathbf{0} & -\mathcal G_1 \\
\end{array}\right),\]
\[\mathbb L_2(\cdot,i)=
\left(\begin{array}{cc}
\mathbf{0} & -\mathcal L_2\\
\mathbf{0} & \mathbf{0}\\
\end{array}\right),\
\mathbb M_2=
\left(\begin{array}{cc}
\mathbf{0} & -\mathcal M_2 \\
\mathbf{0} & \mathbf{0} \\
\end{array}\right),\
\mathbb G_2(i)=
\left(\begin{array}{cc}
\mathbf{0} & \mathcal G_2\\
\mathbf{0} & \mathbf{0}\\
\end{array}\right), \]
and denote $\mathcal G_1'(i)=\nu_1\big(\widetilde G_1'(i)\big)^{\top}\widetilde G_1'(i)$. Similar inventions are used for $\mathcal G_2'(i)$, $\mathcal L_j'(\cdot,i)$,\\ $\mathbb G_j'(i), \mathbb L_j'(\cdot,i)$,
$j=1,2$. Let
\begin{equation}\left\{\begin{aligned}\label{PPG}
& \Psi^0({\color{black}{\textbf{y}}})=\sum_{j=1}^2\mathbb M_j {\color{black}{\textbf{y}}},\quad
\Phi^0\left({\color{black}{\textbf{x}}},\widehat {\color{black}{\textbf{x}}}, i\right)
=\sum_{j=1}^2\mathbb G_j(i){\color{black}{\textbf{x}}}
+\sum_{j=1}^2\mathbb G_j'(i)\widehat{\color{black}{{\textbf{x}}}},\\
& \Gamma^0\big(\cdot,\Theta,\widehat\Theta,i\big)= \sum_{j=1}^2\mathbb L_j(\cdot,i)\Theta
+\sum_{j=1}^2\mathbb L_j'(\cdot,i)\widehat\Theta.
\end{aligned}\right.\end{equation}
For any $\left(\kappa,\xi,\rho(\cdot)\right)\in\mathcal H[0,T]$ with $\rho(\cdot)=\left(\varphi^{\top}(\cdot),\eta^{\top}(\cdot), \gamma^{\top}(\cdot)\right)^{\top}$, we introduce a family of CMF-FBSDEs parameterized by $\delta\in[0,1]$ as follows 
\begin{equation}\label{FBSDE8.1}
\left\{\begin{aligned}
dx_j^{\delta}(t)=& \left[b_j^{\delta}\left(\pmb{\Theta}^{\delta}_t\right)
+\eta_j(t)\right]dt+\left[\sigma_j^{\delta}\left(\pmb{\Theta}^{\delta}_t\right)
+\gamma_j(t)\right]dW(t),\\
dy_j^{\delta}(t)=& \left[g_j^{\delta}\left(\pmb{\Theta}^{\delta}_t\right)
+\varphi_j(t)\right]dt+z_j^{\delta}(t)dW(t)+\tilde z_j^{\delta}(t)\bullet dM(t),\\
x_j^{\delta}(0)=&\ \Psi_j^{\delta}\left({\color{black}{\textbf{y}}}^{\delta}(0)\right)+\kappa_j, \quad
y_j^{\delta}(T)=\Phi_j^{\delta}\left({\color{black}{\textbf{x}}}^{\delta}(T), \widehat{\color{black}{{\textbf{x}}}}^{\delta}(T),\alpha_{T}\right)+\xi_j, \quad j=1,2,
\end{aligned}\right.
\end{equation}
where $\pmb{\Theta}^{\delta}_t=\left(t,\Theta^{\delta}(t),{\widehat\Theta}^{\delta}(t),
\alpha_{t-}\right)$ and 
\begin{equation}\label{alpha0}
\left(\Psi^{\delta}, \Phi^{\delta}, \Gamma^{\delta}\right)
=(1-\delta)\left(\Psi^0, \Phi^0, \Gamma^0\right)+\delta\left(\Psi, \Phi, \Gamma\right). 
\end{equation}
Besides, for any $\left(\kappa', \xi',\rho'(\cdot)\right)\in\mathcal H[0,T]$, \eqref{FBSDE8.1} corresponds to another solution.

{\color{black}{
\begin{Lemma}\label{le3.2}
Let $\left(\kappa,\xi,\rho(\cdot)\right)\in\mathcal H[0,T]$ be given. When $\delta=0$, \eqref{FBSDE8.1} is decoupled and is uniquely solvable.
\end{Lemma}

\emph{Proof of Lemma \ref{le3.2}:}
Let $\delta=0$, for any $i\in\mathcal M$, it follows from \eqref{PPG} that
\[\Psi^0({\color{black}{\textbf{y}}})=
\left(\begin{array}{c}
-\sum_{j=1}^2\mathcal M_jy_j  \\
\mathcal M_1y_2 \\
\end{array}\right),\quad
\Phi^0({\color{black}{\textbf{x}}},\widehat{\color{black}{{\textbf{x}}}},i)=
\left(\begin{array}{c}
\sum_{j=1}^2\left[\mathcal G_j(i)x_j
+\mathcal G_j'(i)\widehat x_j\right]  \\
-\mathcal G_1(i)x_2-\mathcal G_1'(i)\widehat x_2\\
\end{array}\right),\quad \]
\[\Gamma^0\left(\cdot,\Theta, \widehat\Theta,i\right)=
\left(\begin{array}{c}
-\sum_{j=1}^2\left[\mathcal L_j(\cdot,i)\Theta_j+\mathcal L_j'(\cdot,i)\widehat\Theta_j\right]  \\
\mathcal L_1(\cdot,i)\Theta_2+\mathcal L_1'(\cdot,i)\widehat\Theta_2 \\
\end{array}\right).\quad \]
Then, we have
\begin{equation}\label{xyz0}
\left\{\begin{aligned}
dx_2^{0}(t)=& \left[b_2^{0}\left(\pmb{\Theta}^{0}_{2,t}\right)+\eta_2(t)\right]dt
+\left[\sigma_2^{0}\left(\pmb{\Theta}^{0}_{2,t}\right)+\gamma_2(t)\right]dW(t),\\
dy_2^{0}(t)=& \left[g_2^{0}\left(\pmb{\Theta}^{0}_{2,t}\right)
 +\varphi_2(t)\right]dt+z_2^0(t)dW(t)+\tilde z_2^0(t)\bullet dM(t),\\
x_2^{0}(0)=&\ \mathcal M_1y^{0}_2(0)+\kappa_2, \quad
y_2^{0}(T)=-\mathcal G_1(\alpha_{T})x^{0}_2(T)-\mathcal G_1'(\alpha_{T})\widehat x^0_2(T)+\xi_2,
\end{aligned}\right.
\end{equation}
where $\pmb{\Theta}^{0}_{2,t}=\left(t,\Theta_2^{0}(t),\widehat\Theta^0_2(t),\alpha_{t-}\right)$,
$\Theta_2^0(t)=\left((x_2^0(t))^{\top}, ({y_2^0}(t))^{\top},({z_2^0}(t))^{\top}\right)^{\top}$, \\
$\widehat\Theta_2^0(t)=\left((\widehat x_2^0(t))^{\top},(\widehat y_2^0(t))^{\top},
(\widehat z_2^0(t))^{\top}\right)^{\top}$, and
\begin{equation}
\left\{\begin{aligned}
&b_2^0=\sqrt{\mu_1\nu_1}\widetilde B_1^{\top}\widetilde A_1x_2^0+\mu_1\widetilde B_1^{\top}\widetilde B_1y_2^0
+\mu_1\widetilde B_1^{\top}\widetilde C_1z_2^0+\sqrt{\mu_1\nu_1}\left(\widetilde B_1'\right)^{\top}\widetilde A_1'\widehat x_2^0 \\
&\qquad\  +\mu_1\left(\widetilde B_1'\right)^{\top}\widetilde B_1'\widehat y_2^0
+\mu_1\left(\widetilde B_1'\right)^{\top}\widetilde C_1'\widehat z_2^0,\\
&\sigma_2^0=\sqrt{\mu_1\nu_1}\widetilde C_1^{\top}\widetilde A_1x_2^0+\mu_1\widetilde C_1^{\top}\widetilde B_1y_2^0
+\mu_1\widetilde C_1^{\top}\widetilde C_1z_2^0+\sqrt{\mu_1\nu_1}\left(\widetilde C_1'\right)^{\top}
\widetilde A_1'\widehat x_2^0 \\
&\qquad\ +\mu_1\left(\widetilde C_1'\right)^{\top}\widetilde B_1'\widehat y_2^0
+\mu_1\left(\widetilde C_1'\right)^{\top}\widetilde C_1'\widehat z_2^0, \\
&g_2^0=\nu_1\widetilde A_1^{\top}\widetilde A_1x_2^0+\sqrt{\mu_1\nu_1}\widetilde A_1^{\top}\widetilde B_1y_2^0
+\sqrt{\mu_1\nu_1}\widetilde A_1^{\top}\widetilde C_1z_2^0+\nu_1\left(\widetilde A_1'\right)^{\top}\widetilde A_1'\widehat x_2^0 \\
&\qquad\ +\sqrt{\mu_1\nu_1}\left(\widetilde A_1'\right)^{\top}\widetilde B_1'\widehat y_2^0 
 +\sqrt{\mu_1\nu_1}\left(\widetilde A_1'\right)^{\top}\widetilde C_1'\widehat z_2^0. \nonumber
\end{aligned}\right.
\end{equation}
Note that $b_2^0,\sigma_2^0,g_2^0$ depend on $\mu_1,\nu_1$. Recall that there are two cases about $\mu_1, \nu_1$ in Assumption (A2). If $\nu_1=0$, then $g_2^0=0$, $\mathcal G_1=\nu_1{\widetilde G_1}^{\top}\widetilde G_1=0$, $\mathcal G_1'=\nu_1\left(\widetilde G_1'\right)^{\top}\widetilde G_1'=0$. In this case, \eqref{xyz0} is a decoupled CMF-FBSDE, where the backward equation does not depend on the forward one. So we can first solve $\left(y_2^0(\cdot),z_2^0(\cdot),\tilde z_2^0(\cdot)\right)$ from the second equation of \eqref{xyz0} with the terminal condition, and then solve $x_2^0(\cdot)$ from the first equation. On the other hand, if $\mu_1=0$, then $b_2^0=\sigma_2^0=0$, and $\mathcal M_1=\mu_1{\widetilde M_1}^{\top}\widetilde M_1=0$. In this case, we can first solve the forward equation and subsequently handle the backward equation.
The proof is completed.  \hfill$\square$ }}

According to Lemma \ref{le3.2}, \eqref{FBSDE8.1} is uniquely solvable for $\delta=0$. When $\delta=1$ and $\left(\kappa, \xi,\rho(\cdot)\right)=(\mathbf{0},\mathbf{0},\mathbf{0})$, \eqref{FBSDE8.1} turns to \eqref{fbsde}. In what follows, we will prove that there exists a fixed step length $\iota_0\geq 0$, such that if for some $\delta_0\in[0,1)$, \eqref{fbsde} is uniquely solvable for any $\left(\kappa, \xi,\rho(\cdot)\right)\in\mathcal H[0,T]$, then the same conclusion also holds for $\delta_0$ being replaced by $\delta_0+\iota \leq 1$ with $\iota \in [0,\iota_0]$. Once this has been proved, we can increase the parameter $\delta$ step by step until it equals $1$. This approach is formally known as the continuation method, please refer to \cite{Hu7, PW15, Yong21} for more information.

In what follows, we present a priori estimate and establish the continuity lemma (see Lemmas \ref{L3.3} and \ref{20.1} below).

\begin{Lemma}\label{L3.3}
Let Assumptions (A1)-(A2) hold for the given $\left(\Psi, \Phi, \Gamma\right)$. Let $\delta\in[0,1]$ and $\left(\kappa,\xi,\rho(\cdot)\right)$, $\left(\kappa',\xi',\rho'(\cdot)\right)\in\mathcal H[0,T]$. Suppose that $\pi(\cdot), \pi'(\cdot)\in \big(N_{\mathbb F}^2\left(0,T; \mathbb R^{3n}\right)$ \\
$\times \mathcal M_{\mathbb F}^2\left(0,T; \mathbb R\right)\big)^2$ satisfying \eqref{FBSDE8.1} with $\left(\Psi^{\delta}+\kappa, \Phi^{\delta}+\xi, \Gamma+\rho(\cdot)\right)$ and $\big(\Psi^{\delta}+\kappa', \Phi^{\delta}+\xi', \Gamma+\rho'(\cdot)\big)$, respectively. Then, we have
\begin{equation}\begin{aligned}\label{11-1}
&\mathbb{E}\left[\sup_{t\in[0,T]}\left|\Delta{\color{black}{\emph{\textbf{x}}}}(t)\right|^2
+\sup_{t\in[0,T]}\left|\Delta {\color{black}{\emph{\textbf{y}}}}(t)\right|^2
+\int_0^T\left|\Delta {\color{black}{\emph{\textbf{z}}}}(t)\right|^2dt+\int_0^T\left|\Delta \tilde{\color{black}{ {\emph{\textbf{z}}}}}(t)\right|^2dt\right] \\
\leq &\ K\mathbb{E}\left[\left|\Delta\kappa\right|^2+\left|\Delta\xi\right|^2
+\left(\int_0^T\left|\Delta\varphi(t)\right|dt\right)^2
+\left(\int_0^T\left|\Delta\eta(t)\right|dt\right)^2
+\int_0^T\left|\Delta\gamma(t)\right|^2 dt\right],
\end{aligned}\end{equation}
where $\Delta {\color{black}{\emph{\textbf{x}}}}(\cdot)={\color{black}{\emph{\textbf{x}}}}(\cdot)-{\color{black}{\emph{\textbf{x}}}}'(\cdot)$, $\Delta\kappa=\kappa-\kappa'$, etc. Here $K$ is a positive constant, which is independent of $\delta$.
\end{Lemma}

{\color{black}{
\emph{Proof of Lemma \ref{L3.3}:} For the sake of clarity, we divide the proof process into the following steps. 

\textbf{Step 1}. Freeze $\left(x_2(\cdot), y_2(\cdot),z_2(\cdot),\tilde z_2(\cdot)\right)$. Then, utilizing Proposition \ref{prop1} for $\left(x_1, y_1,z_1,\tilde z_1\right)$, it arrives that
\begin{equation}\label{E1}
\begin{aligned}
&\mathbb{E}\left[\sup_{t\in[0,T]}\left|\Delta x_1(t)\right|^2+\sup_{t\in[0,T]}\left|\Delta y_1(t)\right|^2
+\int_0^T\left|\Delta z_1(t)\right|^2dt+\int_0^T\left|\Delta \tilde z_1(t)\right|^2dt\right] \\
\leq &\ K\mathbb{E}\left[\mathbb J_1+\mathbb I_{1,1}+\mathbb I_{1,2}(\alpha_T)
+\int_0^T\mathbb I_{1,3}(\alpha_{t-})dt\right],
\end{aligned}
\end{equation}
where
$$\mathbb J_1=\left|\Delta\kappa_1\right|^2+\left|\Delta\xi_1\right|^2+\left(\int_0^T\left|\Delta\varphi_1(t)\right|dt\right)^2
+\left(\int_0^T\left|\Delta\eta_1(t)\right|dt\right)^2+\left(\int_0^T\left|\Delta\gamma_1(t)\right|dt\right)^2,$$
and
\begin{equation}
\left\{\begin{aligned}
&\mathbb I_{1,1}=\left|\Psi_1^{\delta}\left(y_1'(0),y_2(0)\right)-\Psi_1^{\delta}\left(y_1'(0), y_2'(0)\right)\right|^2,\\
&\mathbb I_{1,2}(\alpha_T)=\big|\Phi_1^{\delta}\left(x_1'(T),x_2(T),\widehat x_1'(T),\widehat x_2(T),
\alpha_T\right) \\
& \qquad\qquad\quad-\Phi_1^{\delta}\left( x_1'(T),x_2'(T),\widehat x_1'(T),\widehat x_2'(T),\alpha_T\right)\big|^2,\\
&\mathbb I_{1,3}(\alpha_{t-})=\Big|\Gamma_1^{\delta}\left(t,\Theta_1'(t),\Theta_2(t), \widehat\Theta_1'(t),
\widehat\Theta_2(t),\alpha_{t-}\right) \\
&\qquad\qquad\quad-\Gamma_1^{\delta}\left(t,\Theta_1'(t), \Theta_2'(t), \widehat\Theta_1'(t), \widehat\Theta_2'(t),\alpha_{t-}\right)\Big|^2. \nonumber
\end{aligned}\right.
\end{equation}
With the definition of $\left(\Psi^{\delta}, \Phi^{\delta}, \Gamma^{\delta}\right)$ given by \eqref{alpha0} and  condition \eqref{A3-1}, we get
\begin{equation}\label{E2}
\begin{aligned}
\mathbb I_{1,1}=&\ \big|(1-\delta)\left[\Psi_1^0\left(y_1'(0),y_2(0)\right)
-\Psi_1^0\left(y_1'(0),y_2'(0)\right)\right]  \\
&\ +\delta\left[\Psi_1\left(y_1'(0),y_2(0)\right)
-\Psi_1\left(y_1'(0),y_2'(0)\right) \right]\big|^2 \\
\leq&\ K \left\{\left|-(1-\delta)\mathcal M_2\Delta y_2(0)\right|^2
+\left|\delta K_0\widetilde M_2\Delta y_2(0)\right|^2\right\} \\
\leq &\ K \left\{ \left|(1-\delta)\mu_2\Lambda\widetilde M_2^{\top}\widetilde M_2\Delta y_2(0)\right|^2
+\left|\delta K_0\widetilde M_2\Delta y_2(0)\right|^2\right\} \\
\leq &\ K\left|\widetilde M_2\Delta y_2(0)\right|^2.
\end{aligned}
\end{equation}
Similarly, we derive 
\begin{equation}\label{E4}
\left\{\begin{aligned}
&\mathbb I_{1,2}(i)\leq K\left|\widetilde G_2(i)\Delta x_2(T)
+\widetilde G_2'(i)\Delta\widehat x_2(T)\right|^2,\\
&\mathbb I_{1,3}(i)\leq K\left|\widetilde L_2(t,i)\Delta\Theta_2(t)+\widetilde L_2'(t,i)\Delta \widehat\Theta_2(t)\right|^2,
\end{aligned}\right.
\end{equation}
$i\in\mathcal M$. Based on \eqref{E1}-\eqref{E4}, we obtain
\begin{equation}\label{13-1}
\begin{aligned}
&\mathbb{E}\left[\sup_{t\in[0,T]}\left|\Delta x_1(t)\right|^2+\sup_{t\in[0,T]}\left|\Delta y_1(t)\right|^2
+\int_0^T\left|\Delta z_1(t)\right|^2dt+\int_0^T\left|\Delta \tilde z_1(t)\right|^2dt\right] \\
\leq &\ K\mathbb{E}\Bigg[\mathbb J_1+\left|\widetilde G_2(\alpha_T)\Delta x_2(T)+\widetilde G_2'(\alpha_T)\Delta \widehat x_2(T)\right|^2+\left|\widetilde M_2\Delta y_2(0)\right|^2 \\
&\ +\int_0^T\left|\widetilde L_2(t,\alpha_{t-})\Delta \Theta_2(t)+\widetilde L_2'(t,\alpha_{t-})\Delta\widehat\Theta_2(t)\right|^2dt\Bigg].
\end{aligned}
\end{equation}

\textbf{Step 2}. Freeze $\left(x_1(\cdot), y_1(\cdot),z_1(\cdot),\tilde z_1(\cdot)\right)$.  Similar to \eqref{E1}, we have
\begin{equation}\label{}
\begin{aligned}
&\mathbb{E}\left[\sup_{t\in[0,T]}\left|\Delta x_2(t)\right|^2+\sup_{t\in[0,T]}\left|\Delta y_2(t)\right|^2
+\int_0^T\left|\Delta z_2(t)\right|^2dt+\int_0^T\left|\Delta \tilde z_2(t)\right|^2dt\right] \\
\leq &\ K\mathbb{E}\left[\mathbb J_2+\mathbb I_{2,1}+\mathbb I_{2,2}(\alpha_T)+\int_0^T\mathbb I_{2,3}(\alpha_{t-})dt\right],  \nonumber
\end{aligned}
\end{equation}
where
$$\mathbb J_2=\left|\Delta\kappa_2\right|^2+\left|\Delta\xi_2\right|^2+\left(\int_0^T\left|\Delta\varphi_2(t)\right|dt\right)^2
+\left(\int_0^T\left|\Delta\eta_2(t)\right|dt\right)^2+\left(\int_0^T\left|\Delta\gamma_2(t)\right|dt\right)^2,$$
and
\begin{equation}
\left\{\begin{aligned}
&\mathbb I_{2,1}=\left|\Psi_2^{\delta}\left(y_1(0),y_2'(0)\right)-\Psi_2^{\delta}\left(y_1'(0),
y_2'(0)\right)\right|^2,\\
&\mathbb I_{2,2}(\alpha_T)=\big|\Phi_2^{\delta}\left(x_1(T),x_2'(T),\widehat x_1(T),\widehat x_2'(T),\alpha_T\right) \\
&\qquad\qquad\quad -\Phi_2^{\delta}\left(x_1'(T),x_2'(T), \widehat x_1'(T),\widehat x_2'(T),\alpha_T\right)\big|^2,\\
&\mathbb I_{2,3}(\alpha_{t-})=\Big|\Gamma_2^{\delta}\left(t,\Theta_1(t),\Theta_2'(t), \widehat\Theta_1(t),
\widehat\Theta_2'(t),\alpha_{t-}\right)  \\
&\qquad\qquad\quad -\Gamma_2^{\delta}\left(t,\Theta_1'(t),\Theta_2'(t),
\widehat\Theta_1'(t),\widehat\Theta_2'(t),\alpha_{t-}\right)\Big|^2. \nonumber
\end{aligned}\right.
\end{equation}
By the definition of $\left(\Psi^{\delta}, \Phi^{\delta}, \Gamma^{\delta}\right)$ given by \eqref{alpha0} and the uniformly Lipschitz conditions of $(\Psi,\Phi,\Gamma)$, we have
\begin{equation}\label{}
\left\{\begin{aligned}
& \mathbb I_{2,1}\leq K\left|\Delta y_1(0)\right|^2,  \quad
\mathbb I_{2,2}(\alpha_T)\leq K\left[\left|\Delta x_1(T)\right|^2
+\left|\Delta \widehat x_1(T)\right|^2\right], \\
& \mathbb I_{2,3}(\alpha_{t-})\leq K\left[\left|\Delta \Theta_1(t)\right|^2
+\left|\Delta \widehat\Theta_1(t)\right|^2\right].\nonumber
\end{aligned}\right.
\end{equation}
Then, we have
\begin{equation}\label{14-1}
\begin{aligned}
&\mathbb{E}\left[\sup_{t\in[0,T]}\left|\Delta x_2(t)\right|^2+\sup_{t\in[0,T]}\left|\Delta y_2(t)\right|^2
+\int_0^T\left|\Delta z_2(t)\right|^2dt+\int_0^T\left|\Delta \tilde z_2(t)\right|^2dt\right] \\
\leq &\ K\mathbb{E}\left[\mathbb J_2+\left|\Delta x_1(T)\right|^2+\left|\Delta \widehat x_1(T)\right|^2
+\left|\Delta y_1(0)\right|^2
+\int_0^T\left(\left|\Delta \Theta_1(t)\right|^2+\left|\Delta \widehat\Theta_1(t)\right|^2\right)dt\right]\\
\leq &\ K \mathbb{E}\left[\mathbb J_2+\sup_{t\in[0,T]}\left|\Delta x_1(t)\right|^2+\sup_{t\in[0,T]}\left|\Delta y_1(t)\right|^2+\int_0^T\left|\Delta \Theta_1(t)\right|^2dt\right].
\end{aligned}
\end{equation}
It follows from \eqref{13-1} and \eqref{14-1} that
\begin{equation}\begin{aligned}\label{15-1-1}
&\mathbb{E}\left[\sup_{t\in[0,T]}\left|\Delta \textbf{x}(t)\right|^2
+\sup_{t\in[0,T]}\left|\Delta \textbf{y}(t)\right|^2
+\int_0^T\left|\Delta \textbf{z}(t)\right|^2dt
+\int_0^T\left|\Delta \tilde{\textbf{z}}(t)\right|^2dt\right] \\
\leq &\ K\mathbb{E}\Bigg[\mathbb J+\left|\widetilde G_2(\alpha_T)\Delta x_2(T)+\widetilde G_2'(\alpha_T)\Delta\widehat x_2(T)\right|^2+\left|\widetilde M_2\Delta y_2(0)\right|^2 \\
&\ +\int_0^T\left(\left|\widetilde L_2(t,\alpha_{t-})\Delta \Theta_2(t)+\widetilde L_2'(t,\alpha_{t-})\Delta \widehat\Theta_2(t)\right|^2\right)dt\Bigg],
\end{aligned}\end{equation}
where
\begin{equation}\label{mathbbJ}
\mathbb J=\left|\Delta\kappa\right|^2+\left|\Delta\xi\right|^2+\left(\int_0^T\left|\Delta\varphi(t)\right|dt\right)^2
+\left(\int_0^T\left|\Delta\eta(t)\right|dt\right)^2+\left(\int_0^T\left|\Delta\gamma(t)\right|dt\right)^2.
\end{equation}

\textbf{Step 3}. Applying the generalized It\^{o}'s formula to $\langle\Delta x_1(\cdot),\Delta y_2(\cdot)\rangle$ \\$+\langle\Delta x_2(\cdot),\Delta y_1(\cdot)\rangle$, we have
\begin{equation}\begin{aligned}\label{16-2}
&\ \mathbb{E}\Big[\left\langle\Phi^{\delta}\left(\textbf{x}(T),\widehat{\textbf{x}}(T),
\alpha_T\right)-\Phi^{\delta}\left(\textbf{x}'(T),\widehat{\textbf{x}}'(T),\alpha_T\right)
+\Delta\xi, \Lambda\Delta \textbf{x}(T)\right\rangle\\
& -\left\langle \Psi^{\delta}(\textbf{y}(0)) 
-\Psi^{\delta}(\textbf{y}'(0))+\Delta\kappa, \Lambda\Delta \textbf{y}(0)\right\rangle\Big] \\
=&\ \mathbb{E}\int_0^T\left\langle \Gamma^{\delta}\left(t, \Theta(t), \widehat\Theta(t),\alpha_{t-}\right)-\Gamma^{\delta}\left(t, \Theta'(t), \widehat\Theta'(t),\alpha_{t-}\right)+\Delta\rho(t),\Lambda\Delta\Theta(t)\right\rangle dt,
\end{aligned}\end{equation}
which implies that
\begin{equation}\begin{aligned}
&\ \mathbb{E}\Bigg[\left\langle\Phi^{\delta}\left(\textbf{x}(T),\widehat{\textbf{x}}(T),
\alpha_T\right)-\Phi^{\delta}\left(\textbf{x}'(T),\widehat{\textbf{x}}'(T),\alpha_T\right), 
\Lambda\Delta \textbf{x}(T)\right\rangle \\
&\ -\left\langle \Psi^{\delta}(\textbf{y}(0))-\Psi^{\delta}(\textbf{y}'(0)), 
\Lambda\Delta \textbf{y}(0)\right\rangle \nonumber
\end{aligned}\end{equation}
\begin{equation}\begin{aligned}
&\ -\int_0^T\left\langle \Gamma^{\delta}\left(t, \Theta(t), \widehat\Theta(t),\alpha_{t-}\right)-\Gamma^{\delta}\left(t, \Theta'(t), \widehat\Theta'(t),\alpha_{t-}\right),\Lambda\Delta\Theta(t)\right\rangle dt \Bigg]\\
=&\ \mathbb{E}\left[-\left\langle\Delta\xi, \Lambda\Delta \textbf{x}(T)\right\rangle
+\left\langle\Delta\kappa, \Lambda\Delta \textbf{y}(0)\right\rangle
+\int_0^T\left\langle\Delta\rho(t),\Lambda\Delta\Theta(t)\right\rangle dt\right].  \nonumber
\end{aligned}\end{equation}
Recalling \eqref{A3-2}, we obtain
\begin{equation}\begin{aligned}\label{18-1}
&\ \mathbb{E}\left[\left|\widetilde G_2(\alpha_T)\Delta x_2(T)+\widetilde G_2'(\alpha_T)\Delta\widehat x_2(T)\right|^2\Big|\mathcal F_{T-}^{\alpha}\right]+\mathbb{E}\left|\widetilde M_2\Delta y_2(0)\right|^2 \\
&\ +\int_0^T\mathbb{E}\left[\left|\widetilde L_2(t,\alpha_{t-})\Delta\Theta_2(t)+\widetilde L_2'(t,\alpha_{t-})\Delta \widehat\Theta_2(t)\right|^2\Big|\mathcal F_{t-}^{\alpha}\right]dt \\
\leq & \ \frac{1}{\mu_2}\mathbb{E}\left[\left\langle \Phi^{\delta}\left(\textbf{x}(T),\widehat{\textbf{x}}(T),
\alpha_T\right)-\Phi^{\delta}\left(\textbf{x}'(T),\widehat{\textbf{x}}'(T),\alpha_T\right),\Lambda\Delta \textbf{x}(T)\right\rangle\Big|\mathcal F_{T-}^{\alpha}\right] \\
&\ -\frac{1}{\mu_2}\mathbb{E}\left[\left\langle \Psi^{\delta}\left(\textbf{y}(0)\right)
-\Psi^{\delta}\left(\textbf{y}'(0)\right), \Lambda \Delta \textbf{y}(0)\right\rangle\right] \\
&\ -\frac{1}{\mu_2}\int_0^T\mathbb{E}\Big[\Big\langle \Gamma^{\delta}\left(t, \Theta(t), \widehat\Theta(t),\alpha_{t-}\right)-\Gamma^{\delta}\left(t, \Theta'(t), \widehat\Theta'(t),\alpha_{t-}\right),\\
&\ \Lambda\Delta\Theta(t)\Big\rangle \Big|\mathcal F_{t-}^{\alpha}\Big]dt.
\end{aligned}\end{equation}
Taking $\mathbb{E}[\cdot]$ on both sides of \eqref{18-1} and applying \eqref{16-2}, we have
\begin{equation}\begin{aligned}\label{}
&\ \mathbb{E}\Bigg[\left|\widetilde G_2(\alpha_T)\Delta x_2(T)+\widetilde G_2'(\alpha_T)\Delta\widehat x_2(T)\right|^2+\left|\widetilde M_2\Delta y_2(0)\right|^2 \\
&\ +\int_0^T\left|\widetilde L_2(t,\alpha_{t-})\Delta\Theta_2(t)+\widetilde L_2'(t,\alpha_{t-})\Delta\widehat\Theta_2(t)\right|^2dt\Bigg] \\
= & \ \frac{1}{\mu_2}\mathbb{E}\left[-\left\langle \Delta\xi,\Lambda\Delta \textbf{x}(T)\right\rangle
+\left\langle \Delta\kappa, \Lambda \Delta \textbf{y}(0)\right\rangle+\int_0^T\left\langle \Delta\rho(t),\Lambda\Delta\Theta(t)\right\rangle dt\right]. \nonumber
\end{aligned}\end{equation}
With Young's inequality, for any $\epsilon>0$, it yields
\begin{equation}\begin{aligned}\label{18-2}
&\ \mathbb{E}\Bigg[\left|\widetilde G_2(\alpha_T)\Delta x_2(T)+\widetilde G_2'(\alpha_T)
\Delta\widehat x_2(T)\right|^2+\left|\widetilde M_2\Delta y_2(0)\right|^2 \\
&\ +\int_0^T\left|\widetilde L_2(t,\alpha_{t-})\Delta\Theta_2(t)
+\widetilde L_2'(t,\alpha_{t-})\Delta\widehat\Theta_2(t)\right|^2dt\Bigg] \\
\leq&\ \frac{1}{\mu_2}\mathbb{E}\Big[\frac{1}{4\epsilon}\left|\Delta\xi\right|^2
+\epsilon\left|\Lambda\Delta \textbf{x}(T)\right|^2
+\frac{1}{4\epsilon}\left|\Delta\kappa\right|^2+\epsilon\left|\Lambda\Delta \textbf{y}(0)\right|^2 \\
&\ +\int_0^T\left(\frac{1}{4\epsilon}\left|\Delta\rho(t)\right|^2
+\epsilon\left|\Lambda\Delta \Theta(t)\right|^2\right)dt\Big]  \\
\leq &\ \frac{1}{4\mu_2\epsilon}\mathbb{E}[\mathbb{J}]
+\frac{2\epsilon}{\mu_2}\mathbb{E}\Bigg[\sup_{t\in[0,T]}\left|\Delta \textbf{x}(t)\right|^2
+\sup_{t\in[0,T]}\left|\Delta \textbf{y}(t)\right|^2+\int_0^T\left|\Delta \textbf{z}(t)\right|^2dt\\
&\ +\int_0^T\left|\Delta \tilde{\textbf{z}}(t)\right|^2dt\Bigg],
\end{aligned}\end{equation}
where $\mathbb{J}$ is given by \eqref{mathbbJ}. It follows from \eqref{15-1-1} and \eqref{18-2} that
\begin{equation}\begin{aligned}\label{}
&\mathbb{E}\left[\sup_{t\in[0,T]}\left|\Delta \textbf{x}(t)\right|^2
+\sup_{t\in[0,T]}\left|\Delta \textbf{y}(t)\right|^2
+\int_0^T\left|\Delta \textbf{z}(t)\right|^2dt
+\int_0^T\left|\Delta \tilde{\textbf{z}}(t)\right|^2dt\right] \\
\leq &\ \left(K+\frac{1}{4\mu_2\epsilon}\right)\mathbb{E}[\mathbb{J}]
+\frac{2\epsilon K}{\mu_2}\mathbb{E}\Bigg[\sup_{t\in[0,T]}\left|\Delta \textbf{x}(t)\right|^2
+\sup_{t\in[0,T]}\left|\Delta \textbf{y}(t)\right|^2+\int_0^T\left|\Delta \textbf{z}(t)\right|^2dt\\
&\ +\int_0^T\left|\Delta \tilde{\textbf{z}}(t)\right|^2dt\Bigg]. \nonumber
\end{aligned}\end{equation}
Choosing $\epsilon=\frac{\mu_2}{4K}$, we obtain \eqref{11-1}. \hfill$\square$

}}

\begin{Lemma}\label{20.1}
Let $\left(\Psi, \Phi, \Gamma\right)$ be given to satisfy Assumptions (A1)-(A2). Then there exists a positive constant $\iota_0>0$, such that if for $\delta_0\in[0,1)$, \eqref{FBSDE8.1} is uniquely solvable for any $(\kappa, \xi, \rho(\cdot))\in\mathcal H[0,T]$, then the same result holds for $\delta=\delta_0+\iota$ with $\iota\in[0,\iota_0]$ and $\delta\leq 1$.
\end{Lemma}

{\color{black}{
\emph{Proof of Lemma \ref{20.1}:} Let $\iota_0>0$ be determined later. Let $\iota\in[0,\iota_0]$. 
For any $\left(\kappa',\xi',\rho'(\cdot)\right)\in\mathcal H[0,T]$ with $\rho'(\cdot)=\left({\varphi'}^{\top}(\cdot),{\eta'}^{\top}(\cdot), {\gamma'}^{\top}(\cdot)\right)^{\top}$,
for any $\pi(\cdot)\in \left(N_{\mathbb F}^2\left(0,T; \mathbb R^{3n}\right)\times \mathcal M_{\mathbb F}^2\left(0,T; \mathbb R\right)\right)^2$, consider a CMF-FBSDE with unknown $\Pi(\cdot)$ as follows:
\begin{equation}\label{20-1}
\left\{\begin{aligned}
dX_j(t)=&\ \left[b_j^{\delta_0}\left(\pmb{\Theta}_t\right)+\eta_j'(t)\right]dt
+\left[\sigma_j^{\delta_0}\left(\pmb{\Theta}_t\right)+\gamma_j'(t)\right]dW(t),\\
dY_j(t)=&\ \left[g_j^{\delta_0}\left(\pmb{\Theta}_t\right)
+\varphi_j'(t)\right]dt+Z_j(t)dW(t)+\widetilde Z_j(t)\bullet dM(t),\\
X_j(0)=&\ \Psi_j^{\delta_0}\left(Y(0)\right)+\kappa_j', \quad
Y_j(T)=\Phi_j^{\delta_0}\left(X(T), \widehat{X}(T),\alpha_T\right)+\xi_j', \quad j=1,2,
\end{aligned}\right.
\end{equation}
where 
\begin{equation}\label{}
\left\{\begin{aligned}
& \pmb{\Theta}_t=\left(t,\Theta(t),\widehat\Theta(t),\alpha_{t-}\right),\quad
\pi=\left(x_1,y_1,z_1,\tilde z_1,x_2,y_2,z_2,\tilde z_2\right), \\
&\Pi=\left(X_1,Y_1,Z_1,\widetilde Z_1,X_2,Y_2,Z_2,\widetilde Z_2\right),\quad
\kappa'=\iota\left[\Psi(\textbf{y}(0))-\sum_{j=1}^2\mathbb M_j\textbf{y}(0)\right]+\kappa, \\
&\rho'(\cdot)=\iota\left[\Gamma\left(\pmb{\Theta}_{\cdot}\right)
-\sum_{j=1}^2\mathbb L_j(\cdot,\alpha_{t-})\Theta(\cdot)
-\sum_{j=1}^2\mathbb L_j'(\cdot,\alpha_{t-})\widehat\Theta(\cdot)\right]+\rho(\cdot),\\
& \xi'=\iota\left[\Phi\left(\textbf{x}(T),\widehat{\textbf{x}}(T),\alpha_T\right)
-\sum_{j=1}^2\mathbb G_j(\alpha_T)\textbf{x}(T)-\sum_{j=1}^2\mathbb G_j'(\alpha_T)
\widehat{\textbf{x}}(T)\right]+\xi. \nonumber
\end{aligned}\right.
\end{equation}
Under our assumption, \eqref{20-1} is uniquely solvable in 
$\left(N_{\mathbb F}^2\left(0,T; \mathbb R^{3n}\right)\times \mathcal M_{\mathbb F}^2\left(0,T; \mathbb R\right)\right)^2$. Define a mapping $f_{\delta_0+\iota}$ as follows:
\begin{equation}\begin{aligned}
\Pi(\cdot)=f_{\delta_0+\iota}\left(\pi(\cdot)\right):&\ 
\left(N_{\mathbb F}^2\left(0,T; \mathbb R^{3n}\right)\times 
\mathcal M_{\mathbb F}^2\left(0,T; \mathbb R\right)\right)^2\longrightarrow \\
&\ \left(N_{\mathbb F}^2\left(0,T; \mathbb R^{3n}\right)\times 
\mathcal M_{\mathbb F}^2\left(0,T; \mathbb R\right)\right)^2.\nonumber
\end{aligned}\end{equation}
Next we proceed to prove that $f_{\delta_0+\iota}$ is contractive.
Let $\pi(\cdot),\pi'(\cdot)\in \Big(N_{\mathbb F}^2\left(0,T; \mathbb R^{3n}\right)$ \\
$\times\mathcal M_{\mathbb F}^2\left(0,T; \mathbb R\right)\Big)^2$, and $\Pi(\cdot)=f_{\delta_0+\iota}\left(\pi(\cdot)\right)$,
$\Pi'(\cdot)=f_{\delta_0+\iota}\left(\pi'(\cdot)\right)$. Let $\Delta\pi(\cdot)=\pi(\cdot)-\pi'(\cdot)$,
$\Delta\Pi(\cdot)=\Pi(\cdot)-\Pi'(\cdot)$, etc. Applying Lemma \ref{L3.3}, we have
\begin{equation}\label{21-2}
\begin{aligned}
& \mathbb{E}\left[\sup_{t\in[0,T]}\left|\Delta X(t)\right|^2+\sup_{t\in[0,T]}\left|\Delta Y(t)\right|^2
+\int_0^T\left|\Delta Z(t)\right|^2dt+\int_0^T\left|\Delta{\widetilde Z}(t)\right|^2dt\right] \\
\leq &\ K\mathbb{E}\Bigg\{\bigg|\iota\Big[\Psi(\textbf{y}(0))-\Psi(\textbf{y}'(0))
-\sum_{j=1}^2\mathbb M_j\Delta \textbf{y}(0)\Big]\bigg|^2 \\
&\ +\bigg|\iota\Big[\Phi\left(\Lambda_T\right)-\Phi\left(\Lambda_T'\right)
-\sum_{j=1}^2\mathbb G_j(\alpha_T)\Delta \textbf{x}(T)
-\sum_{j=1}^2\mathbb G_j'(\alpha_T)\Delta\widehat{\textbf{x}}(T)\Big] \bigg|^2 \\ 
&\ +\int_0^T\Big|\iota\Big[\Gamma\left(\pmb{\Theta}_t\right)-\Gamma\left(\pmb{\Theta}'_t\right) 
 -\sum_{j=1}^2\mathbb L_j(t,\alpha_{t-})\Delta\Theta(t)-\sum_{j=1}^2\mathbb L_j'(t,\alpha_{t-})\Delta\widehat\Theta(t)\Big]\Big|^2dt\Bigg\} \\
\leq &\ K\iota^2\mathbb{E}\left[\sup_{t\in[0,T]}\left|\Delta \textbf{x}(t)\right|^2
+\sup_{t\in[0,T]}\left|\Delta \textbf{y}(t)\right|^2+\int_0^T\left|\Delta \textbf{z}(t)\right|^2dt+\int_0^T\left|\Delta \tilde{\textbf{z}}(t)\right|^2dt\right],
\end{aligned}
\end{equation}
where $\Lambda_T=\left(\textbf{x}(T),\widehat{\textbf{x}}(T),\alpha_T\right)$, $\Lambda_T'$ and $\pmb{\Theta}'_{\cdot}$ are given by \eqref{fuhao}. Take $\iota_0=\frac{1}{2\sqrt{K}}$. Then, for any $\iota\in[0,\iota_0]$, it follows from \eqref{21-2} that $f_{\delta_0+\iota}$ is contractive.  \hfill$\square$

}}

{\color{black}{\emph{Proof of Theorem \ref{eufbsde}:}}}
Combining {\color{black}{Lemmas \ref{le3.2}}} and \ref{20.1}, the unique solvability of \eqref{fbsde} can be obtained. Further, by Lemma 5.2 and \eqref{FBSDE8.1}, letting $\delta=1$,
$\left(\kappa', \xi', \rho'(\cdot)\right)=
\left(\Psi'({\color{black}{\textbf{y}}}'(0))-\Psi({\color{black}{\textbf{y}}}'(0)), \Phi'\left(\Lambda_T'\right)-\Phi\left(\Lambda_T'\right),  
\Gamma'\left(\pmb{\Theta}'_{\cdot}\right)-\Gamma\left(\pmb{\Theta}'_{\cdot}\right)\right)$
and $\left(\kappa, \xi, \rho(\cdot)\right)=(\mathbf{0},\mathbf{0},\mathbf{0})$,
where $\Lambda_T'$ and $\pmb{\Theta}'_{\cdot}$ are given by \eqref{fuhao}, then \eqref{5-1} holds. \hfill$\square$

\section*{Appendix B}\qquad\qquad\qquad\qquad\qquad 

\emph{Proof of Theorem \ref{OL}:}
Let $\mathbf{u^*_{F_{1,2}}}(\cdot)\in\mathbf{\mathcal U_{F_{1,2}}}$ and ${\color{black}{x_1}}(\cdot)$ be the corresponding state, and let $\left({\color{black}{y_1}}(\cdot), {\color{black}{z_1}}(\cdot),
{\color{black}{\tilde z_1}}(\cdot)\right)$ be given by \eqref{p}. $\mathbf{u^*_{F_{1,2}}}(\cdot)$ is an open-loop saddle point if and only if for any $u_{F,j}\in\mathcal U_{F,j}$, $j=1,2$, and any $\varepsilon\in\mathbb R$,
\begin{equation}\left\{\begin{aligned}\label{J12}
&J_F\left(x_0; u_L(\cdot), u_{F,1}^*(\cdot)+\varepsilon u_{F,1}(\cdot), u_{F,2}^*(\cdot)\right)
\geq J_F\left(x_0; u_L(\cdot), u_{F,1}^*(\cdot), u_{F,2}^*(\cdot)\right), \\
&J_F\left(x_0; u_L(\cdot),u_{F,1}^*(\cdot), u_{F,2}^*(\cdot)+\varepsilon u_{F,2}(\cdot)\right)
\leq J_F\left(x_0; u_L(\cdot), u_{F,1}^*(\cdot), u_{F,2}^*(\cdot)\right).
\end{aligned}\right.\end{equation}
For any $u_{F,1}(\cdot)\in\mathcal U_{F,1}$ and any $\varepsilon\in\mathbb R$, let ${\color{black}{x_1^{\varepsilon}}}(\cdot)$ be the solution to \eqref{state} with $\left(u_{F,1}^*(\cdot)+\varepsilon u_{F,1}(\cdot), u_{F,2}^*(\cdot)\right)$. Moreover, let $x_{F,j}(\cdot)$, $j=1,2,$ be the solution to
\begin{equation}
\left\{\begin{aligned}
&dx_{F,j}(t)=\left[A(t,\alpha_{t-})x_{{F,j}}(t)+\bar{A}(t,\alpha_{t-})\widehat x_{F,j}(t)
+B_{F,j}(t,\alpha_{t-})u_{F,j}(t)\right]dt\\
&\quad +\big[C(t,\alpha_{t-})x_{{F,j}}(t)+\bar{C}(t,\alpha_{t-})\widehat x_{F,j}(t)
+D_{F,j}(t,\alpha_{t-})u_{F,j}(t)\big]dW(t), \\
& x_{F,j}(0)=0. \nonumber
\end{aligned}\right.
\end{equation}
Then, 
\vspace*{-0.1cm}
\begin{equation}\begin{aligned}\label{J1-J2}
&\ J_F\left(x_0; u_L(\cdot), u_{F,1}^*(\cdot)+\varepsilon u_{F,1}(\cdot), u_{F,2}^*(\cdot)\right)
-J_F\left(x_0; u_L(\cdot), u_{F,1}^*(\cdot), u_{F,2}^*(\cdot)\right)  \\
=&\ \varepsilon \mathcal I_{F,1}+\frac{\varepsilon^2}{2}\Upsilon_{F,1},
\end{aligned}\end{equation}
where 
\begin{equation}\begin{aligned}
\mathcal I_{F,1}=&\ \mathbb{E}\bigg\{G_F(\alpha_T){\color{black}{x_1}}(T)x_{F,1}(T)+\bar G_F(\alpha_T)
{\color{black}{\widehat x_1}}(T)\widehat x_{F,1}(T) \\
&\ +\int_0^T\Big[Q_F(t,\alpha_{t-}){\color{black}{x_1}}(t)x_{F,1}(t)
 +\bar Q_F(t,\alpha_{t-}){\color{black}{\widehat x_1}}(t)\widehat x_{F,1}(t) \\
&\ +{u_{F,1}^{*\top}}(t)R_{F,1}(t,\alpha_{t-})u_{F,1}(t)
 +{u_{F,1}^{\top}}(t)S_{F}(t,\alpha_{t-})u_{F,2}^*(t)\Big]dt\bigg\},\nonumber
\end{aligned}\end{equation} 
and
\begin{equation}\begin{aligned}
\Upsilon_{F,1}=&\ \mathbb{E}\bigg\{G_F(\alpha_T)x_{F,1}^2(T)+\bar G_F(\alpha_T)\widehat x_{F,1}^2(T)
+\int_0^T\Big[Q_F(t,\alpha_{t-})x_{F,1}^2(t) \\
&\ +\bar Q_F(t,\alpha_{t-})\widehat x_{F,1}^2(t)
+{u_{F,1}^{\top}}(t)R_{F,1}(t,\alpha_{t-})u_{F,1}(t)\Big]dt\bigg\}.  \nonumber
\end{aligned}\end{equation}
Applying It\^{o}'s formula to ${\color{black}{y_1}}(\cdot)x_{F,1}(\cdot)$ with the relation 
$$\mathbb{E}\Big[\bar G_F(\alpha_T){\color{black}{\widehat x_1}}(T)\widehat x_{F,1}(T)\Big] 
=\mathbb{E}\Big[\bar G_F(\alpha_T){\color{black}{\widehat x_1}}(T)x_{F,1}(T)\Big],$$ we obtain
\begin{equation}\begin{aligned}\label{IF1}
\mathcal I_{F,1}=&\ \mathbb{E}\Bigg\{G_F(\alpha_T){\color{black}{x_1}}(T)x_{F,1}(T)
+\bar G_F(\alpha_T){\color{black}{\widehat x_1}}(T) x_{F,1}(T)\\ 
&\ +\int_0^T\Big[Q_F(t,\alpha_{t-}){\color{black}{x_1}}(t)x_{F,1}(t)
+\bar Q_F(t,\alpha_{t-}){\color{black}{\widehat x_1}}(t)\widehat x_{F,1}(t) \\
&\ +{u_{F,1}^{*\top}}(t)R_{F,1}(t,\alpha_{t-})u_{F,1}(t)
+{u_{F,1}^{\top}}(t)S_{F}(t,\alpha_{t-})u_{F,2}^*(t)\Big]dt\Bigg\} \\
=&\ \mathbb{E}\int_0^T\big\langle B_{F,1}^{\top}(t,\alpha_{t-}){\color{black}{y_1}}(t)+D_{F,1}^{\top}(t,\alpha_{t-}){\color{black}{z_1}}(t)
+R_{F,1}^{\top}(t,\alpha_{t-})u_{F,1}^*(t)\\ 
&\ +S_F(t,\alpha_{t-})u_{F,2}^*(t), u_{F,1}(t)\big\rangle dt. 
\end{aligned}\end{equation}
Putting \eqref{IF1} into \eqref{J1-J2}, it yields  
\begin{equation}\begin{aligned}
& J_F\left(x_0; u_L(\cdot), u_{F,1}^*(\cdot)+\varepsilon u_{F,1}(\cdot), u_{F,2}^*(\cdot)\right)
-J_F\left(x_0; u_L(\cdot), u_{F,1}^*(\cdot), u_{F,2}^*(\cdot)\right) \\
=&\ \varepsilon\mathbb{E}\int_0^T\big\langle B_{F,1}^{\top}(t,\alpha_{t-}){\color{black}{y_1}}(t)
+D_{F,1}^{\top}(t,\alpha_{t-}){\color{black}{z_1}}(t)+R_{F,1}(t,\alpha_{t-})u_{F,1}^*(t)  \\
&\ +S_F(t,\alpha_{t-})u_{F,2}^*(t), u_{F,1}(t)\big\rangle dt+\frac{\varepsilon^2}{2}\Upsilon_{F,1}.\nonumber
\end{aligned}
\end{equation}
Similarly, for any $u_{F,2}(\cdot)\in\mathcal U_{F,2}$ and any $\varepsilon\in\mathbb R$, we have
\begin{equation}\begin{aligned}
& J_F\left(x_0; u_L(\cdot), u_{F,1}^*(\cdot), u_{F,2}^*(\cdot)+\varepsilon u_{F,2}(\cdot)\right)
-J_F\left(x_0; u_L(\cdot), u_{F,1}^*(\cdot), u_{F,2}^*(\cdot)\right) \\
=&\ \varepsilon\mathbb{E}\int_0^T\big\langle B_{F,2}^{\top}(t,\alpha_{t-}){\color{black}{y_1}}(t)
+D_{F,2}^{\top}(t,\alpha_{t-}){\color{black}{z_1}}(t)+R_{F,2}(t,\alpha_{t-})u_{F,2}^*(t)  \\
&\ +S_F^{\top}(t,\alpha_{t-})u_{F,1}^*(t), u_{F,2}(t)\big\rangle dt
+\frac{\varepsilon^2}{2}\Upsilon_{F,2},\nonumber
\end{aligned}\end{equation}
where
\begin{equation}\begin{aligned}
\Upsilon_{F,2}=&\ 
\mathbb{E}\bigg\{G_F(\alpha_T)x_{F,2}^2(T)+\bar G_F(\alpha_T)\widehat x_{F,2}^2(T)
+\int_0^T\Big[Q_F(t,\alpha_{t-})x_{F,2}^2(t) \\
&\ +\bar Q_F(t,\alpha_{t-})\widehat x_{F,2}^2(t)
+{u_{F,2}^{\top}}(t)R_{F,2}(t,\alpha_{t-})u_{F,2}(t)\Big]dt\bigg\}.  \nonumber
\end{aligned}\end{equation}
According to Proposition 4.3 in Sun and Yong \cite{2014SJR}, $V_{F}^{+}(x_0)$ and $V_{F}^{-}(x_0)$ are finite,
then $\Upsilon_{F,1}\geq 0$ and $\Upsilon_{F,2}\leq 0$, which implies that \eqref{J12} is equivalent to 
\begin{equation}\left\{\begin{aligned}
&B_{F,1}^{\top}(t,\alpha_{t-}){\color{black}{y_1}}(t)
+D_{F,1}^{\top}(t,\alpha_{t-}){\color{black}{z_1}}(t)
+R_{F,1}(t,\alpha_{t-})u_{F,1}^*(t)+S_F(t,\alpha_{t-})u_{F,2}^*(t)=0, \\
&B_{F,2}^{\top}(t,\alpha_{t-}){\color{black}{y_1}}(t)
+D_{F,2}^{\top}(t,\alpha_{t-}){\color{black}{z_1}}(t)
+R_{F,2}(t,\alpha_{t-})u_{F,2}^*(t)+S_F^{\top}(t,\alpha_{t-})u_{F,1}^*(t)=0.
\end{aligned}\right.\end{equation}
As a result, we obtain \eqref{u-pq}. The proof is completed.  \hfill$\square$

\section*{Appendix C}\qquad\qquad\qquad

To prove Theorem \ref{riccati}, we first present some arguments. Recall \eqref{BRF}, set
\begin{equation*}
\begin{aligned}
\mathbb{R}_F(\cdot,i)
&\triangleq
\left(\begin{array}{cc}
\mathbb{R}_{F,11}  & \mathbb{R}_{F,12} \\
\mathbb{R}_{F,12}^{\top} & \mathbb{R}_{F,22}\\
\end{array}\right)
=\left(\begin{array}{cc}
R_{F,1}+D^{\top}_{F,1}D_{F,1}P_F & D^{\top}_{F,1}D_{F,2}P_F \\
D^{\top}_{F,2}D_{F,1}P_F & R_{F,2}+D^{\top}_{F,2}D_{F,2}P_F\\
\end{array}\right),\\
\end{aligned}
\end{equation*}
and
\begin{equation*}
\begin{aligned}
\mathbb{B}_F(\cdot,i)
&\triangleq
\left(\begin{array}{c}
\mathbb{B}_{F,1}  \\
\mathbb{B}_{F,2} \\
\end{array}\right)
=\left(\begin{array}{cc}
\left(B_{F,1}^{\top}+D_{F,1}^{\top}C\right)P_F \\
\left(B_{F,2}^{\top}+D_{F,2}^{\top}C\right)P_F \\
\end{array}\right).\\
\end{aligned}
\end{equation*}
Notice that \eqref{eq1} is a system of Riccati equations, and the solutions are expect to satisfy
$P_F(\cdot, i)\in[-\bar\varrho,\bar\varrho]$, $i\in\mathcal M$. To prove the solvability of \eqref{eq1}, we first give some estimates. Assume that $P_F(\cdot, i)\in[-\bar\varrho,\bar\varrho]$, then it follows from Assumption (F5) that
\begin{equation}\label{RRf}
-\frac{1}{\varrho}\mathbf{I}_{m_1}\leq  -\mathbb{R}^{-1}_{F,11}<\mathbf{0},\quad
\mathbf{0}<-\mathbb{R}_{F,22}^{-1}\leq  \frac{1}{\varrho}\mathbf{I}_{m_2}.
\end{equation}
Thus, 
\begin{equation}
\check{R}_{F,11}\triangleq\mathbb{R}_{F,11}-\mathbb{R}_{F,12}\mathbb{R}_{F,22}^{-1}\mathbb{R}_{F,12}^{\top}
\geq  \varrho\mathbf{I}_{m_1}, \quad \check{R}_{F,22}\triangleq\mathbb{R}_{F,22}-\mathbb{R}_{F,12}^{\top}\mathbb{R}_{F,11}^{-1}\mathbb{R}_{F,12}
\leq  -\varrho\mathbf{I}_{m_2}, \nonumber
\end{equation}
and
\begin{equation}\label{RR}
\mathbf{0}<\check{R}_{F,11}^{-1}\leq  \frac{1}{\varrho}\mathbf{I}_{m_1},\quad
  -\frac{1}{\varrho}\mathbf{I}_{m_2} \leq\check{R}_{F,22}^{-1}<\mathbf{0}.
\end{equation}
Recalling that $\mathbb{R}_F$ is invertible, applying the block matrix inversion method, we obtain the form of $\mathbb{R}_F^{-1}$ as follows
\begin{equation*}
\begin{aligned}
\mathbb{R}_F^{-1}&=
\left(\begin{array}{cc}
\check{R}_{F,11}^{-1} & -\check{R}_{F,11}^{-1}\mathbb{R}_{F,12}\mathbb{R}_{F,22}^{-1} \\
-\mathbb{R}_{F,22}^{-1}\mathbb{R}_{F,12}^{\top}\check{R}_{F,11}^{-1}
& \mathbb{R}_{F,22}^{-1}+\mathbb{R}_{F,22}^{-1}\mathbb{R}_{F,12}^{\top}\check{R}^{-1}_{F,11}
\mathbb{R}_{F,12}\mathbb{R}_{F,22}^{-1}\\
\end{array}\right).\\
\end{aligned}
\end{equation*}
Further, we define $$\mathcal H\big(\cdot,i,P_F(\cdot,i)\big)= -\mathbb{B}_F^{\top}\mathbb{R}_F^{-1}\mathbb{B}_F,\quad i\in\mathcal M,$$ 
which has the form
\begin{equation}\begin{aligned}\label{H.}
\mathcal H\big(\cdot,i,P_F(\cdot,i)\big)  
= &\ -\mathbb{B}_{F,2}^{\top}\mathbb{R}_{F,22}^{-1}\mathbb{B}_{F,2}
-\left(\mathbb{B}_{F,1}-\mathbb{R}_{F,12}\mathbb{R}_{F,22}^{-1}\mathbb{B}_{F,2}\right)^{\top}\check{R}_{F,11}^{-1}\\
&\  \times\left(\mathbb{B}_{F,1}-\mathbb{R}_{F,12}\mathbb{R}_{F,22}^{-1}\mathbb{B}_{F,2}\right) \\
{\color{black}{ =}}&{\color{black}{\ -\mathbb{B}_{F,2}^{\top}\mathbb{R}_{F,22}^{-1}\mathbb{B}_{F,2}
-\mathcal H_{\mathbb B}^{\top}\check{R}_{F,11}^{-1}\mathcal H_{\mathbb B},}}
\end{aligned}\end{equation}
{\color{black}{ where $\mathcal H_{\mathbb B}=\mathbb{B}_{F,1}-\mathbb{R}_{F,12}\mathbb{R}_{F,22}^{-1}\mathbb{B}_{F,2}$.}}


\vspace*{0.5cm}

\emph{Proof of Theorem \ref{riccati}:} We first prove the existence and uniqueness of \eqref{eq1}. Then, we discuss the wellposedness of \eqref{eq2}-\eqref{eq3}.

\textbf{The existence of \eqref{eq1}.}
Let $\mathbb{P}_F=\left(P_F(\cdot,1),\cdots, P_F(\cdot,m)\right)^{\top}\in\mathbb R^m$, and for any $(t,i)\in[0,T]\times \mathcal M$, let
\begin{equation*}
g(t,i,\mathbb{P}_F)=\Big[2A(t,i)+C^2(t,i)-\sum_{j\neq i}\lambda_{ij}\Big]P_F(t,i)+Q_F(t,i)
+\sum_{j\neq i}\lambda_{ij}P_F(t,j),
\end{equation*}
\begin{equation}\label{H1.}
\mathcal H_1\big(\cdot,i,P_F(\cdot,i)\big)=
\left\{\begin{aligned}
& -\mathbb{B}_{F,2}^{\top}\mathbb{R}_{F,22}^{-1}\mathbb{B}_{F,2},\quad P_F(\cdot,i)\in[-\bar\varrho,\bar\varrho], \\
& 0,\quad otherwise,
\end{aligned}\right.
\end{equation}
and
\begin{equation}\label{H2.}
\mathcal H_2\big(\cdot,i,P_F(\cdot,i)\big)=
\left\{\begin{aligned}
& -\mathcal H_{\mathbb B}^{\top}\check{R}_{F,11}^{-1}\mathcal H_{\mathbb B},\quad P_F(\cdot,i)\in[-\bar\varrho,\bar\varrho],\\
& 0, \quad otherwise.
\end{aligned}\right.
\end{equation}
{\color{black}{For any $P_F(\cdot,i)\in[-\bar\varrho, \bar\varrho]$, it follows from \eqref{H.}-\eqref{H2.} that
\begin{equation}
\mathcal H\big(\cdot,i,P_F(\cdot,i)\big)
=\mathcal H_1\big(\cdot,i,P_F(\cdot,i)\big)+\mathcal H_2\big(\cdot,i,P_F(\cdot,i)\big).
\end{equation} 
}}Besides, by virtue of \eqref{RRf}-\eqref{RR}, we obtain
\begin{equation}\left\{\begin{aligned} 
& -\mathbb{B}_{F,1}^{\top}\mathbb{R}_{F,11}^{-1}\mathbb{B}_{F,1}\leq \mathcal H\big(t,i,P_F(t,i)\big)\leq
-\mathbb{B}_{F,2}^{\top}\mathbb{R}_{F,22}^{-1}\mathbb{B}_{F,2},\\
&-\mathbb{B}_{F,1}^{\top}\mathbb{R}_{F,11}^{-1}\mathbb{B}_{F,1}\geq
-\frac{1}{\varrho}\mathbb{B}_{F,1}^{\top}\mathbb{B}_{F,1}\geq  -\frac{1}{\varrho}c_3\bar\varrho^2,\quad
0<-\mathbb{B}_{F,2}^{\top}\mathbb{R}_{F,22}^{-1}\mathbb{B}_{F,2}\leq  \frac{1}{\varrho}c_3\bar\varrho^2,\nonumber
\end{aligned}\right.
\end{equation}
which gives 
\begin{equation}
0\leq  \mathcal H_1\big(\cdot,i,P_F(\cdot,i)\big)\leq  \frac{1}{\varrho}c_3\bar\varrho^2,\quad
-\frac{1}{\varrho}c_3 \bar\varrho^2 \leq  \mathcal H_2\big(\cdot,i,P_F(\cdot,i)\big)\leq  0,\quad (t,i)\in[0,T]\times \mathcal M.
\end{equation}
For any $k\geq 1$ and $\big(\cdot,i,P_F(\cdot,i)\big)\in[0,T]\times\mathcal M\times[-\bar\varrho, \bar\varrho]$, let
\begin{equation}\label{Hk}
\left\{\begin{aligned} 
&\mathcal H_1^{k_1}\big(\cdot,i,P_F(\cdot,i)\big)
=\inf_{P_F^0\in\mathbb R}\big[\mathcal H_1(\cdot,i,P_F^0)+k_1\left|P_F(\cdot,i)-P_F^0\right|\big],\\
&\mathcal H_2^{k_2}\big(\cdot,i,P_F(\cdot,i)\big)
=\sup_{P_F^0\in\mathbb R}\big[\mathcal H_2(\cdot,i,P_F^0)-k_2\left|P_F(\cdot,i)-P_F^0\right|\big].\\
\end{aligned}\right.
\end{equation}
As a result, $\mathcal H_1^{k_1}\big(\cdot,i,P_F(\cdot,i)\big)$ and $\mathcal H_2^{k_2}\big(\cdot,i,P_F(\cdot,i)\big))$ are increasing and decreasing to $\mathcal H_1\big(\cdot,i,P_F(\cdot,i)\big)$ and $\mathcal H_2\big(\cdot,i,P_F(\cdot,i)\big)$, respectively, as $k_1, k_2\rightarrow +\infty$. {\color{black}{For any $P_F(\cdot,i)\in[-\bar\varrho, \bar\varrho]$, it follows from
\eqref{Hk} that
\begin{equation}\begin{aligned}
&\lim_{k_1,k_2\rightarrow +\infty}
\left[\mathcal H_1^{k_1}\big(\cdot,i,P_F(\cdot,i)\big)+\mathcal H_2^{k_2}\big(\cdot,i,P_F(\cdot,i)\big)\right] \\
=& \lim_{k_1\rightarrow +\infty} \mathcal H_1^{k_1}\big(\cdot,i,P_F(\cdot,i)\big)
+\lim_{k_2\rightarrow +\infty}\mathcal H_2^{k_2}\big(\cdot,i,P_F(\cdot,i)\big) \\
=&\ \mathcal H_1\big(\cdot,i,P_F(\cdot,i)\big)+\mathcal H_2\big(\cdot,i,P_F(\cdot,i)\big) \\
=&\ \mathcal H\big(\cdot,i,P_F(\cdot,i)\big).\nonumber
\end{aligned}\end{equation}
}}Moreover, for any $k_1, k_2\geq 1$, we obtain
\begin{equation*}
0\leq \mathcal H_1^{k_1}\big(\cdot,i,P_F(\cdot,i)\big)\leq  \frac{1}{\varrho}c_3\bar\varrho^2,\ \
-\frac{1}{\varrho}c_3\bar\varrho^2\leq \mathcal H_2^{k_2}\big(\cdot,i,P_F(\cdot,i)\big)\leq  0,\ \ \ (t,i)\in[0,T]\times \mathcal M.
\end{equation*}
Next, we solve two ordinary differential equations (ODEs) as follows
\begin{equation}\label{}
\left\{\begin{aligned}
& \dot{P}_F^{+}(t,i)+g^{+}(t,i,\mathbb{P}_F^{+})=0,\quad
 P_F^{+}(T,i)=g_0, \\
& \dot{P}_F^{-}(t,i)+g^{-}(t,i,\mathbb{P}_F^{-})=0,\quad
P_F^{-}(T,i)=-g_0, \nonumber
\end{aligned}\right.
\end{equation}
where
\begin{equation}\label{}
\left\{\begin{aligned}
&\dot{P}_F^{+}(t,i)=\frac{dP_F^{+}(t,i)}{dt},\quad \dot{P}_F^{-}(t,i)=\frac{dP_F^{-}(t,i)}{dt}, \\
&g^{+}(t,i,\mathbb{P}_F^{+})=c_1\sum_{j\in\mathcal M}P_F^{+}(t,j)+q_0+\frac{1}{\varrho}c_3\bar\varrho^2,\\
&g^{-}(t,i,\mathbb{P}_F^{-})=c_1\sum_{j\in\mathcal M}P_F^{-}(t,j)-q_0-\frac{1}{\varrho}c_3\bar\varrho^2,\nonumber
\end{aligned}\right.
\end{equation}
$(t,i)\in[0,T]\times \mathcal M$. Then we have
\begin{equation*}\left\{\begin{aligned}
&P_F^{+}(t,i)=g_0e^{c_1m(T-t)}+\frac{1}{c_1m}\Big(q_0+\frac{1}{\varrho}c_3\bar\varrho^2\Big)
\left[e^{c_1m(T-t)}-1\right], \\
&P_F^{-}(t,i)=-P_F^{+}(t,i), 
\end{aligned}\right.\end{equation*}
which implies 
\begin{equation*}\left\{\begin{aligned}
&0<P_F^{+}(\cdot,i)\leq  P_F^{+}(0,i)=\bar\varrho, \\
& -\bar\varrho=P_F^{-}(0,i)\leq  P_F^{-}(\cdot,i)<0.
\end{aligned}\right.\end{equation*}
The following ODE
\begin{equation}\left\{\begin{aligned}\label{PFk-}
& \dot{P}_F^{k_1,k_2}(t,i)+g\left(t,i,\mathbb{P}_F^{k_1,k_2}\right)+\mathcal H_1^{k_1}\left(t,i,P_F^{k_1,k_2}\right)
+\mathcal H_2^{k_2}\left(t,i,P_F^{k_1,k_2}\right)=0,\\
&P_F^{k_1,k_2}(T,i)=G_F(T,i),\quad \quad (t,i)\in[0,T]\times \mathcal M  
\end{aligned}\right.\end{equation}
admits a unique solution $\mathbb{P}_F^{k_1,k_2}=(P_F^{k_1,k_2}(\cdot,1),\cdots, P_F^{k_1,k_2}(\cdot,m))^{\top}$, since $g(\cdot,i,\mathbb{P}_F^{k_1,k_2})$, $\mathcal H_1^{k_1}\left(\cdot,i,P_F^{k_1,k_2}(\cdot,i)\right)$ and 
$\mathcal H_2^{k_2}\left(\cdot,i,P_F^{k_1,k_2}(\cdot,i)\right)$ satisfy Lipschitz condition with respect to $\mathbb{P}_F^{k_1,k_2}$. The solution of \eqref{PFk-} is
\begin{equation}\begin{aligned}
P_F^{k_1,k_2}(t,i)=&\ G_F(T,i)+\int_t^T\Big[g\left(t,i,\mathbb{P}_F^{k_1,k_2}\right)
+\mathcal H_1^{k_1}\left(t,i,P_F^{k_1,k_2}(t,i)\right)  \\
&\ +\mathcal H_2^{k_2}\left(t,i,P_F^{k_1,k_2}(t,i)\right)\Big]dt.\nonumber
\end{aligned}\end{equation}
Recall that
$\mathcal H_1^{k_1}\left(\cdot, i, P_F^{k_1,k_2}(\cdot,i)\right)$ and $\mathcal H_2^{k_2}\left(\cdot,i,P_F^{k_1,k_2}(\cdot,i)\right)$ are increasing and decreasing to $\mathcal H_1\left(\cdot,i,P_F^{k_1,k_2}(\cdot,i)\right)$ and $\mathcal H_2\left(\cdot,i,P_F^{k_1,k_2}(\cdot,i)\right)$, respectively. Consequently, $P_F^{k_1,k_2}(\cdot, i)$ is increasing and decreasing with respect to $k_1$ and $k_2$, respectively. For any $k_1,k_2\geq 1$ and $i\in\mathcal M$, we have
\begin{equation*}\begin{aligned}
& g\left(\cdot,i,{P}_F^{+}(\cdot,i)\right)+\mathcal H_1^{k_1}\left(\cdot,i,{P}_F^{+}(\cdot,i)\right)
+\mathcal H_2^{k_2}\left(\cdot,i,{P}_F^{+}(\cdot,i)\right)\leq g^{+}(\cdot,i,\mathbb{P}_F^{+}),\\
&g^{-}(\cdot,i,\mathbb{P}_F^{-})\leq
g\left(\cdot,i,{P}_F^{-}(\cdot,i)\right)+\mathcal H_1^{k_1}\left(\cdot,i,{P}_F^{-}(\cdot,i)\right)
+\mathcal H_2^{k_2}\left(\cdot,i,{P}_F^{-}(\cdot,i)\right).
\end{aligned}\end{equation*}
Thus, for any $k_1,k_2\geq 1$ and $(t,i)\in[0,T]\times \mathcal M$, we obtain
\begin{equation*}
-\bar\varrho\leq  P_F^{-}(\cdot,i)\leq P_F^{k_1,k_2}(\cdot,i)\leq  P_F^{+}(\cdot,i)\leq  \bar\varrho.
\end{equation*}
Recall that $P_F^{k_1,k_2}(\cdot,i)$ is monotonic and bounded, then define
\begin{equation*}
P_F(\cdot,i)=\lim_{k_2\rightarrow +\infty}\lim_{k_1\rightarrow +\infty}P_F^{k_1,k_2}(\cdot,i),
\quad (t,i)\in[0,T]\times \mathcal M,
\end{equation*}
where $P_F(\cdot,i)\in[-\bar\varrho,\bar\varrho]$ and satisfies \eqref{eq1}.
Thus, the existence of \eqref{eq1} is established.

\textbf{The uniqueness of \eqref{eq1}.}
Assume that $P_F(\cdot,i)$ and $P_F'(\cdot,i)$ are two solutions of \eqref{eq1}, $i\in\mathcal M$. Set
$\Delta_{P_F}(\cdot,i)=P_F(\cdot,i)-P_F'(\cdot,i)$, which satisfies
\begin{equation*}\label{}
\left\{\begin{aligned}
& \dot{\Delta}_{P_F}(t,i)+\left(2A(t,i)+C^2(t,i)\right)\Delta_{P_F}(t,i)
-\mathbb{H}^{\top}({\Delta}_{P_F},{\Delta}_{P_F})(t,i)
\mathbb{R}_F^{-1}(t,i)  \\
&\quad \times\mathbb{H}(P_F,P_F)(t,i)
-\mathbb{H}^{\top}(P_F',P_F')(t,i)
\left(\mathbb{R}_F'(t,i)\right)^{-1}\mathbb{H}({\Delta}_{P_F},{\Delta}_{P_F})(t,i) \\
&\quad +\left(\mathbf{B_F}(t,i)P_F'(t,i)+\mathbf{D_F}(t,i)P_F'(t,i)C(t,i)\right)
\mathbb{R}_F^{-1}(t,i)\mathbf{D_F^{\top}}(t,i)\Delta_{P_F}(t,i)\mathbf{D_F}(t,i)  \\
&\quad \times \left(\mathbb{R}_F'(t,i)\right)^{-1}\mathbb{H}(P_F,P_F)(t,i) 
+\sum_{j\in\mathcal M}\lambda_{ij}\left[\Delta_{P_F}(t,j)-\Delta_{P_F}(t,i)\right]=0,\\
&\Delta_{P_F}(T,i)=0,\quad (t,i)\in[0,T]\times \mathcal M,
\end{aligned}\right.
\end{equation*}
where $\mathbb{H}(X,Y)(\cdot,i)=\mathbf{B_F^{\top}}X+\mathbf{D_F^{\top}}YC$,
$\mathbb{R}_F(\cdot,i)=\mathbf{R_F}+\mathbf{D_F^{\top}}P_F\mathbf{D_F}$,
$\mathbb{R}_F'(\cdot,i)=\mathbf{R_F}+\mathbf{D_F^{\top}}P_F'\mathbf{D_F}$.
{\color{black}{For any $(t,i)\in[0,T]\times \mathcal M$, we obtain $P_F(t,i), P_F'(t,i)\in[-\bar\varrho, \bar\varrho]$. Since the coefficients $\mathbf{R_F}$ and $\mathbf{D_F}$ are uniformly bounded, $\left(\mathbb{R}_F(\cdot,i)\right)^{-1}$ and $\left(\mathbb{R}_F'(\cdot,i)\right)^{-1}$    
are uniformly bounded.}} By virtue of Gronwall's inequality, it yields that $\Delta_{P_F}(\cdot,i)=0$, $i\in\mathcal M$. This proves the uniqueness of \eqref{eq1}.

Based on the above analysis, \eqref{eq1} is uniquely solvable. Similar to \eqref{eq1}, we can also prove that \eqref{eq2} admits a unique solution. Further, applying Theorem 3.4 of Nguyen et al. \cite{Yin2020}, we get that \eqref{eq3} is uniquely solvable.

Putting \eqref{pq-x} into \eqref{u-pq} and the first equation of \eqref{p}, and assuming that $\mathbf{R_F}(\cdot,i)$\\ 
$+\mathbf{D_F^{\top}}(\cdot,i)\mathbf{D_F}(\cdot,i)P_F(\cdot,i)$ is invertible, $i\in\mathcal M$, we derive the closed-loop representation \eqref{closed-ll} with \eqref{closed-x}.
Besides, the unique solvability of \eqref{closed-ll} is also obtained based on \eqref{eq1}-\eqref{eq3}. \hfill$\square$

\section*{Appendix D} \qquad\qquad\qquad\qquad\qquad\qquad\qquad\qquad\qquad

\emph{Proof of Lemma \ref{L4.3}:}
{\color{black}{
Firstly, we verify the conditions stated in \eqref{2.2}, \eqref{2.4}, and \eqref{2.5}. 

1) By virtue of the initial and terminal conditions in \eqref{2fbsde}, we have
\begin{equation}
\left\{\begin{aligned}
& \left|\Psi_1(y_1,y_2)-\Psi_1(y_1',y_2)\right|=0, \quad 
\left|\Psi_2(y_1,y_2)-\Psi_2(y_1,y_2')\right|=0, \\
& \left\langle \Psi_1(y_1,y_2)-\Psi_1(y_1',y_2), \Delta y_1\right\rangle=0, \\
& \left\langle \Psi_2(y_1,y_2)-\Psi_2(y_1,y_2'), \Delta y_2\right\rangle=0, \\
\end{aligned}\right.\nonumber
\end{equation}
which imply that the first inequalities of \eqref{2.2}, \eqref{2.4}, and \eqref{2.5} hold with 
$\widetilde M_1=0$.

2) Under Assumption (F3), we can also check that the second and the third inequalities of \eqref{2.2}, and the second inequalities of \eqref{2.4}-\eqref{2.5} hold, with $\widetilde G_1(i)=\sqrt{G_F(i)}, \widetilde G_1'(i)=0, \widetilde A_1(\cdot,i)=\sqrt{Q_F(\cdot,i)}, \widetilde A_1'(\cdot,i)=0$, $K_0'=\max\limits_{(t,i)\in[0,T]\times \mathcal M}\Big\{\sqrt{G_F(i)},$\\ $\sqrt{Q_F(t,i)}\Big\}+1>0$ and $\nu_1=1$, $i\in\mathcal M$.

3) For any $\widetilde B_1(\cdot,i), \widetilde C_1(\cdot,i), \widetilde B_1'(\cdot,i), \widetilde C_1'(\cdot,i)\in\mathbb R$, $i\in\mathcal M$, the fourth inequality of \eqref{2.2} can be verified with $\mu_1=0$ (that is, $\frac{1}{\mu_1}=+\infty$). Under Assumption (F3), with $\Gamma_1=\left(g_1^{\top}, b_1^{\top}, \sigma_1^{\top}\right)^{\top}$, 
for any $(t, i)\in [0,T]\times\mathcal M$, any $\Theta_1=(x_1,y_1,z_1)$,
$\widehat\Theta_1=(\widehat x_1,\widehat y_1,\widehat z_1)$, 
$\Theta_1'=\big(x_1',y_1',z_1'\big)$,
$\widehat\Theta_1'=\big(\widehat x_1',\widehat y_1',\widehat z_1'\big)$, 
$\Theta_2=(x_2,y_2,z_2)$, $\widehat\Theta_2=\left(\widehat x_2,\widehat y_2,\widehat z_2\right)$, 
we have
\begin{equation}
\begin{aligned}
&\mathbb{E}\left[\left\langle \Gamma_1\left(t,\Theta_1,\widehat\Theta_1,\Theta_2,\widehat\Theta_2,i\right)
-\Gamma_1\left(t,\Theta_1',\widehat\Theta_1',\Theta_2,\widehat\Theta_2,i\right),\Delta\Theta_1\right\rangle
\Big|\mathcal F_{t-}^{\alpha}\right] \\
=&\ \mathbb{E}\Bigg[\Bigg\langle 
\left(\begin{array}{c}
-\left(A\Delta y_1+\bar A\Delta\widehat y_1+C\Delta z_1+\bar C\Delta\widehat z_1
+Q_F\Delta x_1+\bar Q_F\Delta\widehat x_1\right) \\
A\Delta x_1+\bar A\Delta\widehat x_1-\mathbf{B_F}\mathbf{R_F^{-1}}\mathbf{B_F^{\top}}\Delta y_1
-\mathbf{B_F}\mathbf{R_F^{-1}}\mathbf{D_F^{\top}}\Delta z_1 \\
C\Delta x_1+\bar C\Delta\widehat x_1-\mathbf{D_F}\mathbf{R_F^{-1}}\mathbf{B_F^{\top}}\Delta y_1
-\mathbf{D_F}\mathbf{R_F^{-1}}\mathbf{D_F^{\top}}\Delta z_1 \\
\end{array}\right), \\
&\ \left(\begin{array}{c}
\Delta x_1 \\
\Delta y_1  \\
\Delta z_1  \\
\end{array}\right)\Bigg\rangle\Bigg|\mathcal F_{t-}^{\alpha}\Bigg],\\
=&-\mathbb{E}\left[Q_F\left|\Delta x_1\right|^2\Big|\mathcal F_{t-}^{\alpha}\right] -\mathbb{E}\bigg[\Big|\sqrt{\mathbf{B_F}\mathbf{R_F^{-1}}\mathbf{B_F^{\top}}}\Delta y_1
+\sqrt{\mathbf{D_F}\mathbf{R_F^{-1}}\mathbf{D_F^{\top}}}\Delta z_1\Big|^2\Big|\mathcal F_{t-}^{\alpha}\bigg] \\
\leq &-\mathbb{E}\left[Q_F(t,i)\left|\Delta x_1\right|^2\Big|\mathcal F_{t-}^{\alpha}\right] \\
\leq & -\nu_1\mathbb{E}\left[\left|\widetilde A_1(t,i)\Delta x_1\right|^2\Big|\mathcal F_{t-}^{\alpha}\right], \\
\end{aligned}\nonumber
\end{equation}
which implies that the third inequality of \eqref{2.4} holds. Using a similar argument, the third inequality of \eqref{2.5} also holds.

Based on above discussion, \eqref{2fbsde} satisfy the conditions in \eqref{2.2}, \eqref{2.4}, and \eqref{2.5}.

{\color{black}{Next, we verify the conditions in \eqref{A3-1} and \eqref{A3-2}.}} }}

1) We check that the first and the second inequalities of \eqref{A3-1} and \eqref{A3-2} hold with 
$\widetilde M_2=\widetilde G_2(i)=\widetilde G_2'(i)=0$, $i\in\mathcal M$. 

2) Under Assumption (L2)-(i), for any $(t, i)\in[0,T]\times\mathcal M$, any $\Theta_1=({\color{black}{x_1,y_1,z_1}})$, $\widehat\Theta_1=({\color{black}{\widehat x_1,\widehat y_1,\widehat z_1}})$, $\Theta_2=({\color{black}{x_2,y_2,z_2}})$, $\widehat\Theta_2=\left({\color{black}{\widehat x_2,\widehat y_2,\widehat z_2}}\right)$, $\Theta_2'=({\color{black}{x_2',y_2',z_2'}})$, $\widehat\Theta_2'=\left({\color{black}{\widehat x_2',\widehat y_2',\widehat z_2'}}\right)$,
we have
\begin{equation}\label{}
\begin{aligned}
& \left|\Gamma_1\left(t,\Theta_1,\widehat\Theta_1,\Theta_2,\widehat\Theta_2,i\right)
-\Gamma_1\left(t,\Theta_1,\widehat\Theta_1,\Theta_2',\widehat\Theta_2',i\right)\right| \\
= &\left|
\left(\begin{array}{c}
0 \\
B_LR_L^{-1}B_L^{\top}\Delta {\color{black}{y_2}}+B_LR_L^{-1}D_L^{\top}\Delta {\color{black}{z_2}} \\
D_LR_L^{-1}B_L^{\top}\Delta {\color{black}{y_2}}+D_LR_L^{-1}D_L^{\top}\Delta {\color{black}{z_2}} \\
\end{array}\right)
\right| \\
\leq & \left|
\left(\begin{array}{c}
0 \\
\sqrt{B_LR_L^{-1}B_L^{\top}} \\
\sqrt{ D_LR_L^{-1}D_L^{\top}} \\
\end{array}\right)
\right|
\left|\left(0,\sqrt{B_LR_L^{-1}B_L^{\top}},\sqrt{D_LR_L^{-1}D_L^{\top}}\right)
\left(\begin{array}{c}
\Delta {\color{black}{x_2}} \\
\Delta {\color{black}{y_2}}  \\
\Delta {\color{black}{z_2}}  \\
\end{array}\right)\right|, \nonumber
\end{aligned}
\end{equation}
then the third inequality of \eqref{A3-1} holds with 
$\widetilde L_2(\cdot,i)=\Big(0,\sqrt{B_L(\cdot,i)R_L^{-1}(\cdot,i)B_L^{\top}(\cdot,i)}$, 
$\sqrt{D_L(\cdot,i)R_L^{-1}(\cdot,i)D_L^{\top}(\cdot,i)}\Big)$, $\widetilde L_2'(\cdot,i)=0$, 
and $K_0=\max\limits_{(t,i)\in[0,T]\times \mathcal M}\Big|\widetilde L_2(t,i)\Big|+1$.

3) Under Assumptions (L1) and (L2)-(i), for any $(t, i)\in[0,T]\times \mathcal M$, and any $\Theta, {\widehat\Theta}, \Theta', \widehat\Theta'\in{\mathbb R}^{6}$, we obtain
\begin{equation}\label{}
\begin{aligned}
& \mathbb{E}\left[\left\langle \Gamma\left(t, \Theta,\widehat{\Theta},i\right)-\Gamma\left(t, \Theta',\widehat\Theta',i\right),\Lambda\Delta\Theta\right\rangle\Big|\mathcal F_{t-}^{\alpha}\right] \\
\leq & -\mathbb{E}\left[\left|\left(0,\sqrt{B_LR_L^{-1}B_L^{\top}},\sqrt{D_LR_L^{-1}D_L^{\top}}\right)
\left(\begin{array}{c}
\Delta {\color{black}{x_2}} \\
\Delta {\color{black}{y_2}}  \\
\Delta {\color{black}{z_2}}  \\
\end{array}\right)\right|^2\Bigg|\mathcal F_{t-}^{\alpha}\right],
\nonumber
\end{aligned}
\end{equation}
which implies that the third inequality of \eqref{A3-2} is verified with $\mu_2=1$. \hfill$\square$

\section*{Appendix E}\qquad\qquad\qquad

\emph{Proof of Theorem \ref{th-F}:}
For any $\widetilde u_L(\cdot)\in \mathcal U_L$, the corresponding state is $\widetilde x_1(\cdot)$, we have
\begin{equation}\label{DJ}
\begin{aligned}
& J_L\big(x_0; \widetilde u_L(\cdot), \mathbf{u^*_{F_{1,2}}}(\cdot)\big)-J_L\big(x_0; u_L^*(\cdot), \mathbf{u^*_{F_{1,2}}}(\cdot)\big) \\
=&\ \frac{1}{2}\mathbb{E}\bigg\{\mathlarger\int_0^T\Big[Q_L(t,\alpha_{t-})
\left({\color{black}{\widetilde x_1}}(t)-{\color{black}{x_1}}(t)\right)^2
+\bar Q_L(t,\alpha_{t-})\left({\color{black}{\widehat{\widetilde x}_1}}(t)
-{\color{black}{\widehat x_1}}(t)\right)^2 \\
&\ +\big(\widetilde u_L(t)-u_L^*(t)\big)^{\top}R_L(t,\alpha_{t-})\big(\widetilde u_L(t)-u_L^*(t)\big)\Big]dt
+G_L(\alpha_{T})\left({\color{black}{\widetilde x_1}}(T)-{\color{black}{x_1}}(T)\right)^2\\
&\ +\bar G_L(\alpha_{T})\left({\color{black}{\widehat{\widetilde x}_1}}(T)
-{\color{black}{\widehat x_1}}(T)\right)^2\bigg\}
+\Upsilon_L, \nonumber
\end{aligned}
\end{equation}
where 
\begin{equation}\label{Upsilon}
\begin{aligned}
\Upsilon_L=&\ \mathbb{E}\bigg\{\mathlarger\int_0^T\Big[Q_L(t,\alpha_{t-}){\color{black}{x_1}}(t)
\left({\color{black}{\widetilde x_1}}(t)-{\color{black}{x_1}}(t)\right)+\bar Q_L(t,\alpha_{t-}){\color{black}{\widehat x_1}}(t)\left({\color{black}{\widehat{\widetilde x}_1}}(t)-{\color{black}{\widehat x_1}}(t)\right) \\
&\ +u_L^{*\top}(t)R_L(t,\alpha_{t-})\big(\widetilde u_L(t)-u_L^*(t)\big)\Big]dt
+G_L(\alpha_{T}){\color{black}{x_1}}(T)\left({\color{black}{\widetilde x_1}}(T)-{\color{black}{x_1}}(T)\right)  \\
&\ +\bar G_L(\alpha_{T}){\color{black}{\widehat x_1}}(T)
\left({\color{black}{\widehat{\widetilde x}_1}}(T)-{\color{black}{\widehat x_1}}(T)\right)\bigg\},
\end{aligned}
\end{equation}
${\color{black}{x_1}}(\cdot)$ the unique solution of \eqref{2fbsde}. Let $\widetilde y_1(\cdot)$ be the backward state corresponding to $\widetilde u_L(\cdot)$. Using the generalized It\^{o}'s formula to $\left({\color{black}{\widetilde x_1}}(\cdot)-{\color{black}{x_1}}(\cdot)\right){\color{black}{y_2}}(\cdot)+
{\color{black}{x_2}}(\cdot)\left({\color{black}{\widetilde y_1}}(\cdot)-{\color{black}{y_1}}(\cdot)\right)$, it yields 
\begin{equation}\label{EE}
\begin{aligned}
&\ \mathbb{E}\Big[G_L(\alpha_{T}){\color{black}{x_1}}(T)\left({\color{black}{\widetilde x_1}}(T)
-{\color{black}{x_1}}(T)\right)+\bar G_L(\alpha_{T}){\color{black}{\widehat x_1}}(T)
\left({\color{black}{\widehat{\widetilde x}_1}}(T)-{\color{black}{\widehat x_1}}(T)\right)\Big] \\
=&\ \mathbb{E}\bigg\{\int_0^T\Big[B_L(t,\alpha_{t-})u_L(t){\color{black}{y_2}}(t)
+B_L(t,\alpha_{t-})R_L^{-1}(t,\alpha_{t-})\big(B_L^{\top}(t,\alpha_{t-}){\color{black}{y_2}}(t)  \\
&\ +D_L^{\top}(t,\alpha_{t-}){\color{black}{z_2}}(t)\big){\color{black}{y_2}}(t)   -\left({\color{black}{\widetilde x_1}}(t)-{\color{black}{x_1}}(t)\right)\left(Q_L(t,\alpha_{t-}){\color{black}{x_1}}(t)
+\bar Q_L(t,\alpha_{t-}){\color{black}{\widehat x_1}}(t)\right) \\
&\ +D_L(t,\alpha_{t-})\widetilde u_L(t){\color{black}{z_2}}(t)
+D_L(t,\alpha_{t-})R_L^{-1}(t,\alpha_{t-})\big(B_L^{\top}(t,\alpha_{t-}){\color{black}{y_2}}(t) \\
&\ +D_L^{\top}(t,\alpha_{t-}){\color{black}{z_2}}(t)\big){\color{black}{z_2}}(t)\Big]dt\bigg\}.
\end{aligned}
\end{equation}
Substituting \eqref{EE} and \eqref{uL*} into \eqref{Upsilon}, we obtain
$\Upsilon_L=0. $
Thus, for any $\widetilde u_L(\cdot)\in \mathcal U_L$, we have
$$J_L\big(x_0; \widetilde u_L(\cdot), \mathbf{u^*_{F_{1,2}}}(\cdot)\big)-J_L\big(x_0; u_L^*(\cdot), \mathbf{u^*_{F_{1,2}}}(\cdot)\big)\geq 0,$$
which implies that $u_L^*(\cdot)$ is the unique optimal strategy of the leader. \hfill$\square$

\section*{Appendix F}\qquad\qquad\qquad

\emph{Proof of Theorem \ref{th3.5}:}
\eqref{PL-1} can be written as 
\begin{equation}\label{PL-1b}
\left\{\begin{aligned}
& \dot{\mathbf{P}}_L(t)+\mathbf{P}_L(t)\pmb{\mathcal A}(t)+\pmb{\mathcal A}(t)\mathbf{P}_L(t)
+\mathbf{P}_L(t)\pmb{\mathcal B}_1(t)\mathbf{P}_L(t)+\pmb{\mathcal Q}(t)  \\
&\qquad +\sum_{k=1}^{m-1}\mathcal N_k\mathbf{P}_L(t)\mathcal N_k^{\top}=\textbf{0}, \\
& \mathbf{P}_L(T)=\pmb{\mathcal G}, 
\end{aligned}\right.
\end{equation}
where 
\begin{equation} 
\left\{\begin{aligned}
& \mathbf{P}_L(\cdot)=diag\big(P_L(\cdot,1),\cdots,P_L(\cdot,m)\big), \ \
\pmb{\mathcal G}=diag\big(\mathcal G(1),\cdots,\mathcal G(m)\big), \\
& \pmb{\mathcal B}_1(\cdot)=diag\big(\mathcal B_1(\cdot,1),\cdots,\mathcal B_1(\cdot,m)\big), \ \
\pmb{\mathcal Q}(\cdot)=diag\big(\mathcal Q(\cdot,1),\cdots,\mathcal Q(\cdot,m)\big), \\
& \pmb{\mathcal A}(\cdot)=diag\big(\mathcal A(\cdot,1)+0.5\lambda_{11}\textbf{I},\cdots, \mathcal A(\cdot,m)+0.5\lambda_{mm}\textbf{I}\big), \nonumber
\end{aligned}\right.
\end{equation}
\[\mathcal N_1=
\left(\begin{array}{ccccc}
0&\sqrt{\lambda_{12}}&0& \cdots   &0\\
0 & 0&\sqrt{\lambda_{23}}& \cdots  &0 \\
\vdots &\vdots & \vdots & \ddots & \vdots  \\
0&0&0&\cdots &\sqrt{\lambda_{m-1,m}} \\
\sqrt{\lambda_{m,1}}& 0& 0&\cdots &0\\
\end{array}\right),\]

\[\mathcal N_2=
\left(\begin{array}{cccccc}
0 &0 &\sqrt{\lambda_{13}} &0& \cdots   &0 \\
0 &0 &0 &\sqrt{\lambda_{24}}& \cdots   &0  \\
\vdots &\vdots &\vdots & \vdots &\ddots & \vdots  \\
0&0&0&0&\cdots &\sqrt{\lambda_{m-2,m}} \\
\sqrt{\lambda_{m-1,1}} &0 &0 &0& \cdots   &0 \\
0 &\sqrt{\lambda_{m,2}} &0 &0& \cdots   &0 \\
\end{array}\right),\quad \cdots, \]

\[\mathcal N_{k}=
\left(\begin{array}{cccccccc}
0  &\cdots &0 & \sqrt{\lambda_{1,k+1}}  &\cdots &0  \\
\vdots  &  &\vdots &\vdots  &\ddots &\vdots \\
0  &\cdots &0 &0  &\cdots & \sqrt{\lambda_{m-k,m}}   \\
\sqrt{\lambda_{m-(k-1),1}}  &\cdots &0 &0  &\cdots &0 \\
\vdots  &\ddots &\vdots &\vdots  & &\vdots \\
0  &\cdots &\sqrt{\lambda_{m,k}} &0  &\cdots &0 \\
\end{array}\right),\quad \cdots, \]
and
\[\mathcal N_{m-1}=
\left(\begin{array}{ccccc}
0 &0&\cdots &0 &\sqrt{\lambda_{1,m}}  \\
\sqrt{\lambda_{21}} &0&\cdots &0 &0 \\
0 &\sqrt{\lambda_{32}}&\cdots &0 &0   \\
\vdots &\vdots  & \ddots &\vdots &\vdots \\
0 &0&\cdots &\sqrt{\lambda_{m,m-1}} &0  \\
\end{array}\right).\ \]
For any $i\in\mathcal M$, $\mathcal B_1(\cdot,i)< \textbf{0}$, $\mathcal Q(\cdot,i)\geq \textbf{0}$ and $\mathcal G(i)\geq \textbf{0}$, then $\pmb{\mathcal B}_1(\cdot)<\textbf{0}$, $\pmb{\mathcal Q}(\cdot)\geq \textbf{0}$ and $\pmb{\mathcal G}\geq \textbf{0}$. With this fact, it follows from Theorem 7.2 of Chapter 6 in Yong and Zhou \cite{ZY2003} that \eqref{PL-1b} admits a unique solution, which implies that \eqref{PL-1} is uniquely solvable. Similar to \eqref{PL-1}, the unique solvability of \eqref{PL-2} can be obtained. With the solvability of \eqref{PL-1} and \eqref{PL-2}, applying Theorem \ref{eufbsde}, one derives that \eqref{PL-3} is uniquely solvable.

Putting \eqref{YZ-X} into \eqref{opuL} and the first equation of \eqref{XY}, we derive the closed-loop representation \eqref{closed-ld} with \eqref{closed-xx}. Besides, the unique solvability of \eqref{closed-ld} is also obtained based on \eqref{PL-1}-\eqref{PL-3}. 
\hfill$\square$

\end{document}

%% file: ex_shared.tex
\usepackage{lipsum}
\usepackage{amsfonts}
\usepackage{graphicx}
\usepackage{epstopdf}
\usepackage{algorithmic}
\ifpdf
  \DeclareGraphicsExtensions{.eps,.pdf,.png,.jpg}
\else
  \DeclareGraphicsExtensions{.eps}
\fi

\newsiamremark{remark}{Remark}
\newsiamremark{hypothesis}{Hypothesis}
\crefname{hypothesis}{Hypothesis}{Hypotheses}
\newsiamthm{claim}{Claim}
\newsiamremark{fact}{Fact}
\crefname{fact}{Fact}{Facts}

\headers{LQ mixed Stackelberg-zero-sum game}{Pengyan Huang, Na Li, Zuo Quan Xu, and Harry Zheng}

\title{Linear-quadratic mixed Stackelberg-zero-sum game for mean-field regime switching system\thanks{Submitted to the editors DATE.
\funding{This work was supported by the National Natural Science Foundation of China under Grants 12571475, 12401582 and 12571480, and the Natural Science Foundation of Shandong Province under Grant ZR2024QA240.}}}

\author{Pengyan Huang\thanks{School of Statistics and Mathematics, Shandong University of Finance and Economics,
Jinan, China (\email{huang\_pengyan@163.com}).}
\and Na Li\thanks{Corresponding author. School of Mathematical Sciences, Dalian University of Technology, Dalian, China 
  (\email{lina2025@dlut.edu.cn}).}
\and Zuo Quan Xu\thanks{Department of Applied Mathematics, The Hong Kong Polytechnic University, Hong Kong, China
  (\email{maxu@polyu.edu.hk}).}
\and Harry Zheng\thanks{Department of Mathematics, Imperial College, London SW7 2AZ, UK
  (\email{h.zheng@imperial.ac.uk}).}
}

\usepackage{amsopn}
